\documentclass[english,svgnames]{article}
\usepackage[margin=1.5in]{geometry}
\usepackage[utf8]{inputenc}
\usepackage[english]{babel}
\usepackage{hyperref}
\usepackage{csquotes}
\usepackage{amsmath,amssymb,amsthm,calligra,mathrsfs}
\usepackage{enumerate}
\usepackage{tensor}
\usepackage{extarrows}
\usepackage{tikz-cd}
\usepackage[capitalize]{cleveref}
\usepackage{import}
\usepackage{todo}
\usepackage{comment}

\newtheorem{thm}{Theorem}
\numberwithin{thm}{section}
\newtheorem{prp}[thm]{Proposition}
\newtheorem{cor}[thm]{Corollary}
\newtheorem{lem}[thm]{Lemma}

\theoremstyle{definition}
\newtheorem{dfn}[thm]{Definition}

\theoremstyle{remark}
\newtheorem{remark}[thm]{Remark}
\newtheorem{exm}[thm]{Example}

\numberwithin{equation}{section}
\numberwithin{figure}{section}

\newcommand{\R}{\mathbb{R}}

\newcommand{\Z}{\mathbb{Z}}
\newcommand{\N}{\mathbb{N}}
\newcommand{\M}{\mathbb{M}}
\newcommand{\F}{\mathbb{F}}
\newcommand{\op}[1]{\operatorname{#1}}

\newcommand{\Com}{\mathrm{C}^+}

\newcommand{\Der}{\mathrm{D}^+}
\newcommand{\Hom}{\mathrm{Hom}}
\newcommand{\Ext}{\mathrm{Ext}}
\newcommand{\VectF}{\mathrm{Vect}_{\F}}
\newcommand{\vectF}{\mathrm{vect}_{\F}}
\DeclareMathOperator{\cHom}{\mathscr{H}\text{\kern -3pt {\calligra\large om}}\,}

\begin{document}

\title{Relative Interlevel Set Cohomology Categorifies Extended Persistence Diagrams}

\author{Ulrich Bauer \and Benedikt Fluhr}

\maketitle

\begin{abstract}
  The \emph{extended persistence diagram} introduced by
  Cohen-Steiner, Edelsbrunner, and Harer
  is an invariant of real-valued continuous functions,
  which are \emph{$\mathbb{F}$-tame} in the sense that all open interlevel sets
  have degree-wise finite-dimensional cohomology
  with coefficients in a fixed field $\mathbb{F}$.
  We show that
  \emph{relative interlevel set cohomology (RISC)},
  which is based on the \emph{Mayer--Vietoris pyramid}
  by Carlsson, de Silva, and Morozov,
  categorifies this invariant.
  More specifically,
  we define
  an abelian Frobenius category $\mathrm{pres}(\mathcal{J})$ of presheaves,
  which are presentable in a certain sense,
  such that the RISC $h(f)$
  of an $\mathbb{F}$-tame function
  $f \colon X \rightarrow \mathbb{R}$
  is an object of $\mathrm{pres}(\mathcal{J})$,
  and moreover
  the extended persistence diagram
  of $f$ uniquely determines
  -- and is determined by --
  the corresponding element
  $[h(f)] \in K_0 (\mathrm{pres}(\mathcal{J}))$ in the Grothendieck group
  $K_0 (\mathrm{pres}(\mathcal{J}))$
  of
  the abelian category $\mathrm{pres}(\mathcal{J})$.
  As an intermediate step we show that $\mathrm{pres}(\mathcal{J})$
  is the abelianization
  of the (localized) category of complexes
  of $\mathbb{F}$-linear sheaves on $\mathbb{R}$,
  which are \emph{tame} in the sense that sheaf cohomology
  of any open interval is finite-dimensional in each degree.
  This yields a close link between derived level set persistence
  by Curry, Kashiwara, and Schapira
  and the categorification of extended persistence diagrams.
\end{abstract}

\section{Introduction}
\label{sec:intro}

Based on prior findings
by \cite{Carlsson:2009:ZPH:1542362.1542408,MR3031814}
we introduced the notion of
\emph{relative interlevel set cohomology (RISC)} over a fixed field $\F$
in \cite{2021arXiv210809298B}
as an invariant of real-valued continuous functions,
which are \emph{$\F$-tame}
in the sense that all open interlevel sets
have degree-wise finite-dimensional cohomology with coefficients in $\F$.
More specifically,
for an $\F$-tame function
$f \colon X \rightarrow \R$
its RISC
is a functor
$h(f) \colon \M^{\circ} \rightarrow \VectF$
defined on the opposite poset $\M^{\circ}$
of a sublattice $\M \subset \R^{\circ} \times \R$
in the shape of an infinite strip
depicted in \cref{fig:contraBlock}.
\begin{figure}[t]
  \centering
\begingroup%
  \makeatletter%
  \providecommand\color[2][]{%
    \errmessage{(Inkscape) Color is used for the text in Inkscape, but the package 'color.sty' is not loaded}%
    \renewcommand\color[2][]{}%
  }%
  \providecommand\transparent[1]{%
    \errmessage{(Inkscape) Transparency is used (non-zero) for the text in Inkscape, but the package 'transparent.sty' is not loaded}%
    \renewcommand\transparent[1]{}%
  }%
  \providecommand\rotatebox[2]{#2}%
  \newcommand*\fsize{\dimexpr\f@size pt\relax}%
  \newcommand*\lineheight[1]{\fontsize{\fsize}{#1\fsize}\selectfont}%
  \ifx\svgwidth\undefined%
    \setlength{\unitlength}{275.57849121bp}%
    \ifx\svgscale\undefined%
      \relax%
    \else%
      \setlength{\unitlength}{\unitlength * \real{\svgscale}}%
    \fi%
  \else%
    \setlength{\unitlength}{\svgwidth}%
  \fi%
  \global\let\svgwidth\undefined%
  \global\let\svgscale\undefined%
  \makeatother%
  \begin{picture}(1,0.48987858)%
    \lineheight{1}%
    \setlength\tabcolsep{0pt}%
    \put(0,0){\includegraphics[width=\unitlength,page=1]{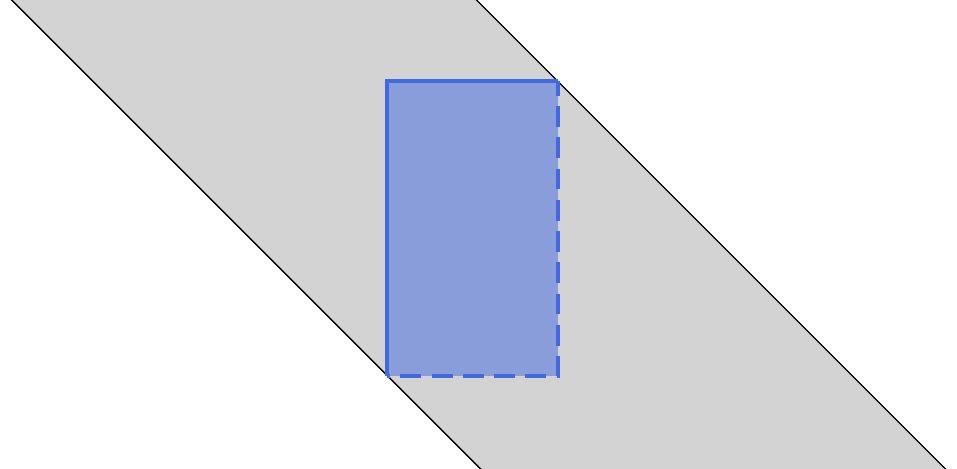}}%
    \put(0.29554647,0.34008106){\makebox(0,0)[t]{\lineheight{1.25}\smash{\begin{tabular}[t]{c}$\{0\}$\end{tabular}}}}%
    \put(0.70040499,0.14574895){\makebox(0,0)[t]{\lineheight{1.25}\smash{\begin{tabular}[t]{c}$\{0\}$\end{tabular}}}}%
    \put(0.49392715,0.23886643){\makebox(0,0)[t]{\lineheight{1.25}\smash{\begin{tabular}[t]{c}$\F$\end{tabular}}}}%
    \put(0.38866395,0.41700397){\makebox(0,0)[t]{\lineheight{1.25}\smash{\begin{tabular}[t]{c}$v$\end{tabular}}}}%
    \put(0.60728751,0.06477746){\makebox(0,0)[t]{\lineheight{1.25}\smash{\begin{tabular}[t]{c}$T^{-1}(v)$\end{tabular}}}}%
    \put(0,0){\includegraphics[width=\unitlength,page=2]{contravariant-block.pdf}}%
    \put(0.75107177,0.26917113){\makebox(0,0)[t]{\lineheight{1.25}\smash{\begin{tabular}[t]{c}$l_1$\end{tabular}}}}%
    \put(0.28850147,0.17303682){\makebox(0,0)[t]{\lineheight{1.25}\smash{\begin{tabular}[t]{c}$l_0$\end{tabular}}}}%
  \end{picture}%
\endgroup%

  \caption{
    The sublattice $\M \subset \R^{\circ} \times \R$
    is shaded in gray and the support of
    the indecomposable cohomological functor
    $B_v \colon \M^{\circ} \rightarrow \mathrm{Vect}_{\F}$
    is shaded in blue.
  }
  \label{fig:contraBlock}
\end{figure}
Here $\R^{\circ} \times \R$ is
the lattice which is given as the product of the reals
$\R^{\circ}$ with the opposite order
and the reals $\R$ with the ordinary \mbox{$\leq$-order}.
The values of
${h(f) \colon \M^{\circ} \rightarrow \VectF}$
are all relative cohomology groups of the form
${H^n(f^{-1}(I, C); \F)}$
for some ${n \in \Z}$,
some open interval ${I \subseteq \R}$
and some subset ${C \subseteq I}$,
that is the complement of a closed interval in $\R$.
Now taking the difference
\begin{equation*}
  (I, C) \mapsto I \setminus C
\end{equation*}
yields a bijection
between the set of all non-empty intervals in $\R$
and the set of all such pairs $(I, C)$ with $I \neq C$.
Moreover, for any pair of open subspaces $(U, V)$ of $\R$
with $U \setminus V = I \setminus C$
we have \[{H^n(f^{-1}(U, V); \F) \cong H^n(f^{-1}(I, C); \F)}\] by excision.
Furthermore,
given any pair of open subspaces $(U, V)$ of $\R$
such that any connected component of $U$ contains finitely many
connected components of $V$
we have
\begin{equation*}
  H^n(f^{-1}(U, V); \F) \cong
  \prod_{(I, C)} H^n(f^{-1}(I, C); \F)
  ,
\end{equation*}
where the product ranges over all pairs $(I, C)$ as above
with $I \setminus C$ a connected component of $U \setminus V$.
Thus,
the RISC ${h(f) \colon \M^{\circ} \rightarrow \VectF}$
contains enough information to obtain the cohomology
$H^n(f^{-1}(U, V); \F)$ for any such pair $(U, V)$.
As it turns out,
the support of
${h(f) \colon \M^{\circ} \rightarrow \VectF}$
is bounded above by the diagonal ${\{(t, t) \mid t \in \R\} \subset \R^2}$.
Moreover, we show in \mbox{\cite[Section 2]{2021arXiv210809298B}} that
${h(f) \colon \M^{\circ} \rightarrow \VectF}$
is point-wise finite-dimensional (pfd),
cohomological (\cref{dfn:cohomological}),
and sequentially continuous (\cref{dfn:cont}).
Throughout this document
-- excluding the appendix --
we denote the full subcategory of pfd functors
${\M^{\circ} \rightarrow \VectF}$
that are cohomological, sequentially continuous,
and have bounded above support by $\mathcal{J}$.
In \cite[Theorem 3.5]{2021arXiv210809298B} we show in particular
that any functor in $\mathcal{J}$ decomposes into a (potentially infinite)
direct sum of indecomposables,
whose supports each have the shape of a maximal axis-aligned rectangle in $\M$
as shown in \cref{fig:contraBlock}.
We denote such an indecomposable by
${B_v \colon \M^{\circ} \rightarrow \VectF}$,
where ${v \in \op{int} \M}$ inside the interior of $\M$
is the upper left corner
of the rectangular support of $B_v$,
see also \cref{dfn:contraBlock}.
In \cite[Section 3.2.2]{2021arXiv210809298B}
we also show
that for an $\F$-tame function ${f \colon X \rightarrow \R}$
the multiplicities of indecomposables
${B_v \colon \M^{\circ} \rightarrow \VectF}$ 
in ${h(f) \colon \M^{\circ} \rightarrow \VectF}$
determine the \emph{extended persistence diagram} of $f$
due to \cite{MR2472288}
and vice versa.
Thus, the extended persistence diagram
provides a complete isomorphism invariant
for the class of functors obtained as RISC.
Now the full subcategory $\mathcal{J}$
is not an abelian subcategory of the category of all functors
${\M^{\circ} \rightarrow \VectF}$.
This raises the question
about the smallest full abelian subcategory of $\VectF^{\M^{\circ}}$
containing $\mathcal{J}$.
Now such a subcategory has to contain in particular all cokernels
of natural transformations in $\mathcal{J}$.
This leads us to considering the closure of $\mathcal{J}$ by cokernels,
consisting of all functors
${F \colon \M^{\circ} \rightarrow \VectF}$
that fit into an exact sequence
\begin{equation}
  \label{eq:presFn}
  H \rightarrow G \rightarrow F \rightarrow 0
\end{equation}
with ${H, G}$ functors in $\mathcal{J}$.
As in \cref{dfn:presentable} we refer to the exact sequence \eqref{eq:presFn}
as a \emph{presentation of ${F}$}
and we say that $F$ is \emph{$\mathcal{J}$-presentable}.
The full subcategory of all $\mathcal{J}$-presentable functors
is denoted by $\mathrm{pres}(\mathcal{J})$.
By the following theorem,
which we prove in \cref{sec:abelianizationSheaves},
the subcategory of all $\mathcal{J}$-presentable functors
is an abelian subcategory.
To state the theorem, we recall that an abelian category is \emph{Frobenius}
if it has enough projectives and enough injectives
sand if both coincide.

\begin{thm}
  \label{thm:frobenius}
  The category of $\mathcal{J}$-presentable functors
  $\mathrm{pres}(\mathcal{J})$
  is an abelian subcategory of $\VectF^{\M^{\circ}}$,
  which is also Frobenius with $\mathcal{J}$ being the subcategory of projective-injectives.
\end{thm}

Now for any point $v \in \M$,
there is an \emph{associated simple functor}
\begin{equation}
  \label{eq:dfnSimpleFunc}
  S_v \colon \M^{\circ} \rightarrow \mathrm{Vect}_{\F},\,
  u \mapsto
  \begin{cases}
    \F & u = v
    \\
    \{0\} & \text{otherwise}
  \end{cases}
\end{equation}
with all internal maps necessarily trivial.
Suppose
${F \colon \M^{\circ} \rightarrow \VectF}$
is some $\mathcal{J}$-presentable functor
and let ${v \in \op{int} \M}$
be a point inside the interior of $\M$.
Even though
${S_v \colon \M^{\circ} \rightarrow \mathrm{Vect}_{\F}}$
is not $\mathcal{J}$-presentable,
both $F$ and $S_v$ are functors vanishing on $\partial \M$.
Now let $\mathcal{C}$ be the category of functors
${\M^{\circ} \rightarrow \mathrm{Vect}_{\F}}$
vanishing on $\partial \M$.
Then $\mathcal{C}$
is an abelian subcategory of $\VectF^{\M^{\circ}}$.
As any functor in $\mathcal{J}$
is projective in $\mathcal{C}$
by \mbox{\cite[Corollary 3.6]{2021arXiv210809298B}},
any projective resolution of
${F \colon \M^{\circ} \rightarrow \VectF}$
by $\mathcal{J}$-presentable functors
also is a projective resolution of $F$
as a functor in $\mathcal{C}$
by \cref{thm:frobenius}.
Thus,
we may use a projective resolution
of ${F \colon \M^{\circ} \rightarrow \VectF}$ in $\mathrm{pres}(\mathcal{J})$
to compute the Ext vector spaces
${\Ext_{\mathcal{C}}^n (F, S_v)}$
with respect to $\mathcal{C}$
for all ${n \in \N_0}$.
(Here it is crucial,
that we work in $\mathcal{C}$
and not all of $\VectF^{\M^{\circ}}$,
as this category has other projectives.)
Moreover,
as all $\mathcal{J}$-presentable functors are pfd,
the Ext vector space ${\Ext_{\mathcal{C}}^n (F, S_v)}$
is finite-dimensional for each ${n \in \N_0}$.

\begin{dfn}
  \label{dfn:betti}
  For any $\mathcal{J}$-presentable functor
  ${F \colon \M^{\circ} \rightarrow \VectF}$
  we define its \emph{$n$-th Betti function}
  to be
  \begin{equation*}
    \beta^n (F) \colon \op{int} \M \rightarrow \N_0, \,
    u \mapsto \dim_{\F} \mathrm{Ext}_{\mathcal{C}}^n (F, S_u)
  \end{equation*}
  for any $n \in \N_0$.
\end{dfn}

In \cref{sec:abelianizationSheaves} we show the following proposition,
which illustrates how above defined Betti functions
are closely related to
the bigraded Betti numbers of \cite[Defintion 2.1]{2019arXiv190205708L}.
(They are not quite the same, because of our condition
on all functors to vanish on $\partial \M$.)

\begin{prp}
  \label{prp:minProjRes}
  Any $\mathcal{J}$-presentable functor
  $F \colon \M^{\circ} \rightarrow \VectF$
  admits a
  (potentially infinite)
  minimal projective resolution
  \begin{equation*}
    \dots \rightarrow
    P_n \rightarrow
    \dots \rightarrow
    P_2 \rightarrow
    P_1 \rightarrow
    P_0 \rightarrow
    F \rightarrow 0
  \end{equation*}
  by functors $P_n \colon \M^{\circ} \rightarrow \VectF$
  in $\mathcal{J}$.
  Moreover,
  the multiplicity of
  $B_u \colon \M^{\circ} \rightarrow \VectF$
  in $P_n$
  is equal to $\beta^n(F)(u) = \dim \Ext_{\mathcal{C}}^n (F, S_u)$
  for each $u \in \op{int} \M$ and $n \in \N_0$.
\end{prp}

Furthermore, we show with \cref{cor:admissibleBetti}
that all Betti functions satisfy a certain finiteness constraint.
We refer to all functions
$\op{int} \M \rightarrow \N_0$
satisfying this constraint
as \emph{admissible Betti functions}
and we denote the commutative monoid of all admissible Betti functions
by $\mathbb{B}$;
this is \cref{dfn:admissibleBetti}.
Then we show in \cref{sec:categorification}
that the alternating sum of functions
- we call it the \emph{Euler function} in \cref{dfn:eulerFn} -
\begin{equation*}
  \chi(F) := \sum_{n \in \N_0} (-1)^n \beta^n(F) \colon
  \op{int} \M \rightarrow \Z, \,
  u \mapsto \sum_{n \in \N_0} (-1)^n \dim \Ext_{\mathcal{C}}^n(F, S_u)
\end{equation*}
is well-defined and moreover,
that the point-wise absolute value
\begin{equation*}
  |\chi(F)| \colon \op{int} \M \rightarrow \N_0,\,
  u \mapsto |\chi(F)(u)|
\end{equation*}
is an admissible Betti function as well.
In \cref{dfn:admissibleEuler}
we refer to such functions as \emph{admissible Euler functions}
and we denote the abelian group of all admissible Euler functions
by $G(\mathbb{B})$
as it is the Grothendieck group
of the commutative monoid $\mathbb{B}$ up to unique isomorphism.
Then we show with \cref{prp:eulerAdditive} that
$\chi \colon \mathrm{Ob}(\mathrm{pres}(\mathcal{J})) \rightarrow G(\mathbb{B})$
is an additive invariant,
hence there is a homomorphism
\begin{equation*}
  [\chi] \colon K_0(\mathrm{pres}(\mathcal{J})) \rightarrow G(\mathbb{B}), \,
  [F] \mapsto \chi(F)
\end{equation*}
of abelian groups.
Finally we obtain the following result.
\begin{thm}
  \label{thm:eulerIso}
  The group homomorphism
  $[\chi] \colon K_0(\mathrm{pres}(\mathcal{J})) \rightarrow G(\mathbb{B})$
  is an isomorphism.
\end{thm}

\sloppy
Now suppose that 
$f \colon X \rightarrow \R$
is an $\F$-tame function.
Then $h(f) \colon \M^{\circ} \rightarrow \VectF$
is a functor in $\mathcal{J}$,
hence $h(f)$ is its own minimal projective resolution.
So by \cref{prp:minProjRes} the function
$\chi(h(f)) = \beta^0(h(f)) \colon \op{int} \M \rightarrow \Z$
counts for each point $u \in \op{int} \M$ inside the interior of $\M$
the multiplicity of
$B_u \colon \M^{\circ} \rightarrow \VectF$.
As noted above,
the multiplicities of functors
$B_u \colon \M^{\circ} \rightarrow \VectF$ for $u \in \op{int} \M$
determine the \emph{extended persistence diagram}
of $f \colon X \rightarrow \R$
by \cite{MR2472288}
and vice versa.
Thus, we may think of the function
\[\chi(h(f)) \colon \op{int} \M \rightarrow \Z\]
as the extended persistence diagram
of ${f \colon X \rightarrow \R}$
and of ${G(\mathbb{B}) \cong K_0(\mathrm{pres}(\mathcal{J}))}$
as \enquote{the abelian group of extended persistence diagrams}
in quotes since not every function in $G(\mathbb{B})$
can be obtained as the extended persistence diagram
of some real-valued continuous function.
Now \mbox{\cite[Definition 1.8]{MR2918217}} provides three notions of
a categorification for an abelian group $G$.
In some sense the strongest of these three notions is that of
a pair of an abelian category $\mathcal{A}$ and a group isomorphism
${G \xrightarrow{\cong} K_0(\mathcal{A})}$
from $G$ to the Grothendieck group $K_0(\mathcal{A})$.
As we provide the functor $h$ from the category of $\F$-tame functions
to the abelian category $\mathrm{pres}(\mathcal{J})$
such that ${[h(f)] \in K_0(\mathrm{pres}(\mathcal{J}))}$
is a faithful representation of the extended peristence diagram of
${f \colon X \rightarrow \R}$,
we think of $\mathrm{pres}(\mathcal{J})$
as the categorification of extended persistence diagrams.
This is also in close analogy to the following categorification
of the Euler characteristic:
Given a topological space $X$ the Euler characteristic $\chi(X) \in \Z$
uniquely determines $[\Delta_{\bullet} (X)] \in K_0 (\mathrm{Ab})$,
which is the element in the Grothendieck group $K_0 (\mathrm{Ab})$
of the category of abelian groups $\mathrm{Ab}$
corresponding to the singular chain complex $\Delta_{\bullet} (X)$.

Throughout this work we make extensive use
of homological algebra and sheaf theory.
For most results needed from these two areas
we will draw on references to chapters 1 and 2
of \cite{Kashiwara1990} respectively.

\subsection{Equivalence of RISC and Derived Level Set Persistence}
\label{sec:introEquiv}

Before we show \cref{thm:frobenius},
we need to provide an intermediate result,
which is interesting in its own right.
In \cref{sec:derivedLevelSet} we recall derived level set persistence
by \cite{MR3259939,MR3873181}
as a functor
\begin{equation*}
  R (-)_* \F_{(-)} \colon
  (\mathrm{Top} / \R)^{\circ} \rightarrow D^+(\R),\,
  (f \colon X \rightarrow \R) \mapsto R f_* \F_X
\end{equation*}
from the opposite category of
the category of topological spaces over the reals $\mathrm{Top} / \R$
to the bounded below derived category $D^+(\R)$
of $\F$-linear sheaves on the real numbers.
Then we describe a cohomological functor
${h_{\R} \colon D^+ (\R) \rightarrow \VectF^{\M^{\circ}}}$
from the derived category $D^+ (\R)$
to the category of functors ${\M^{\circ} \rightarrow \VectF}$.
With \cref{prp:h0Iso} we show
that RISC factors through $h_{\R}$ and
derived level set persistence
up to natural isomorphism $\zeta$:
\begin{equation}
  \label{eq:introTriaZeta}
  \begin{tikzcd}[row sep=12ex, column sep=7ex]
    (\mathrm{lcContr} / \R)^{\circ}
    \arrow[d , "R (-)_* \F_{(-)}"']
    \arrow[dr, "h"{name=h0}]
    \\
    D^+ (\R)
    \arrow[r, "h_{\R}"']
    \arrow[to=h0, Leftarrow, "\zeta", shorten >=1.5ex, shorten <=1.5ex]
    &
    \VectF^{\M^{\circ}}
    .
  \end{tikzcd}
\end{equation}
Here $(\mathrm{lcContr} / \R)^{\circ}$
denotes the full subcategory of all functions
on locally contractible spaces.
Then we consider the full triangulated subcategory $D^+_t(\R)$
of the derived category $D^+(\R)$
consisting of all complexes of sheaves $F$,
which are \emph{tame} in the sense
that all open intervals $I \subseteq \R$
have degree-wise finite-dimensional sheaf cohomology
$H^n (I; F)$ with coefficients in $F$.
With \cref{cor:equivBdClosed} from \cref{sec:equiv}
we provide a statement,
which immediately implies the equivalence
of categories $D^+_t (\R)$ and $\mathcal{J}$:

\begin{thm}
  \label{thm:hIRequiv}
  The functor $h_{\R}$
  (defined in \cref{sec:derivedLevelSet})
  restricts to an equivalence of categories
  ${h_{\R,t} \colon D^+_t (\R) \rightarrow \mathcal{J}}$.
\end{thm}

In particular the diagram \eqref{eq:introTriaZeta}
restricts to the diagram
\begin{equation*}
  \begin{tikzcd}[row sep=12ex, column sep=7ex]
    (\mathrm{lcContr} / \R)^{\circ}_t
    \arrow[d , "R (-)_* \F_{(-)}"']
    \arrow[dr, "h"{name=h0}]
    \\
    D^+_t (\R)
    \arrow[r, "h_{\R,t}"']
    \arrow[to=h0, Leftarrow, "\zeta", shorten >=1.5ex, shorten <=1.5ex]
    &
    \mathcal{J}
  \end{tikzcd}
\end{equation*}
with the lower horizontal functor an equivalence of categories,
where $(\mathrm{lcContr} / \R)^{\circ}_t$
denotes the full subcategory of all $\F$-tame functions
on locally contractible spaces,
see also \cref{dfn:FtameFunctions}.
Now $D^+_t (\R)$ is a full triangulated subcategory
of the derived category $D^+ (\R)$,
while $\mathcal{J}$ is a full subcategory of the abelian category
of $\mathcal{J}$-presentable functors.
We may put this into perspective
with \cref{prp:abelianization},
which is a slight generalization
of a result by \mbox{\cite[Section 4.2]{MR2355771}}.
Using \cref{thm:hIRequiv}
(or rather a generalization thereof)
and \cref{prp:abelianization}
we show in \cref{sec:abelianizationSheaves} that the composition of functors
\begin{equation}
  \label{eq:compAbelianization}
  D^+_t (\R) \xrightarrow{h_{\R,t}} \mathcal{J} \hookrightarrow
  \mathrm{pres}(\mathcal{J})
\end{equation}
is the abelianization of $D^+_t (\R)$,
i.e. \eqref{eq:compAbelianization}
is the universal or initial cohomological functor
on the triangulated category $D^+_t (\R)$;
this is \cref{cor:abelianizationReals}.
Thus, we obtain a close link between derived level set persistence
and the categorification of extended persistence diagrams.

\subsection{Related Work}
\label{sec:relatedWork}

Originally extended persistence diagrams
have been defined by \cite{MR2472288}
in terms of ranks of internal maps
of corresponding \emph{(extended) persistence modules}
\cite{MR2121296}.
Moreover,
these ranks of internal maps form an invariant in their own right,
the \emph{rank invariant}
of multi-dimensional persistence modules by \cite{MR2506738}.
The correspondence between the rank invariant and persistence diagrams
has been generalized by \cite{MR3975559}.
Now persistence modules form an abelian category themselves
and under suitable finiteness assumptions
any $1$-dimensional extended persistence module
decomposes into a direct sum of indecomposables.
As it turns out,
these indecomposables are in a one-to-one correspondence
with the vertices of the corresponding extended persistence diagram.
In particular,
the extended persistence diagram determines the isomorphism class
of the corresponding extended persistence modules and vice versa.
Now the Grothendieck group of the abelian category
of extended persistence modules (with suitable finiteness assumptions)
is generated by the isomorphism classes.
Thus,
the corresponding element in the Grothendieck group
is determined by the extended persistence diagram as well.
However, the converse is not true,
as the category of extended persistence modules
has too many short exact sequences.
For example,
if we consider a persistence module $M$
provided as a representation of an $A_n$-quiver,
then the corresponding element $[M]$ in the Grothendieck group
determines the dimension vector of $M$
and vice versa.
One way to eliminate this mismatch
between the Grothendieck group and extended persistence diagrams
is to \enquote{omit} some of the exact sequences
when forming the Grothendieck group.
This leads to the notion of a Quillen exact category \cite{MR0338129},
which can be defined as an abelian category with a certain class
of \emph{distinguished short exact sequences},
and the corresponding generalized notion of an
associated Grothendieck group \cite{MR0435185}.
With this notion this mismatch between the Grothendieck group
and extended persistence diagrams
can be eliminated by considering the Quillen exact category
with respect to the class of split short exact sequences
of extended persistence modules.
A similar approach has been used by \cite{2021arXiv210706800B}
in the multi-dimensional setting.
There the authors introduce their notion of
\emph{rank-exact} short exact sequences,
which coincide with split short exact sequences
in the $1$-dimensional setting (under suitable finiteness assumptions).
Further invariants of multi-dimensional peristence modules
defined in terms of Grothendieck groups
of Quillen exact categories
were introduced by \cite{2021arXiv211207632B}.
In the present work we propose another approach
to eliminate the mismatch between the Grothendieck group
and (extended) persistence diagrams.
Instead of constraining ourselves to short exact sequences
of extended persistence modules (or of sheaves on the reals)
that split,
we define the cohomological functor \eqref{eq:compAbelianization}
from the full subcategory $D^+_t(\R)$
of the derived category $D^+(\R)$
to $\mathcal{J}$-presentable functors
and then we consider the bona fide Grothendieck group
of $\mathrm{pres}(\mathcal{J})$.
(We note that we have to constrain ourselves to some subcategory
of tame objects since otherwise,
the Grothendieck group would be trivial by the Eilenberg swindle.)
In future work we intend to explore similar techniques
to obtain invariants of multi-dimensional persistence modules
in terms Grothendieck groups of abelian categories as well,
whereas in the present work we focus on $\F$-tame functions,
which we consider to be one of several sweet spots
within the realm of persistence theory.

Originally persistence diagrams were introduced as multisets
by \cite{MR2279866},
which we may think of as functions taking values in the natural numbers $\N_0$
or non-negative functions to the integers $\Z$.
This notion of a persistence diagram has been generalized
by \cite{MR3975559}
to functions taking values in an abelian group.
One particular generalization of a persistence diagram by \cite{MR3975559}
is defined for persistence modules
taking values in an abelian category $\mathcal{A}$.
For such a persistence module
its associated persistence diagram
is a function taking values in the Grothendieck group
$K_0(\mathcal{A})$.
In the present work we consider extended persistence diagrams themselves
as elements of the Grothendieck group $K_0(\mathrm{pres}(\mathcal{J}))$.

Now in order to show
that ${h_{\R, t} \colon D^+_t (\R) \rightarrow \mathcal{J}}$
is an equivalence of categories,
we employ similar techniques as \cite{MR935124}
used to describe derived categories of Dynkin quivers.
More specifically, we may think of the category
of $\F$-linear sheaves on the reals
as a continuous counterpart
to the category of representations of an $A_n$-quiver
with alternating orientations.
Now \cite{MR935124} associates to any quiver
an $\F$-linear \emph{mesh category} in such a way
that any two quivers of type $A_n$ have isomorphic mesh categories.
As a result of this theory,
the derived category of a Dynkin quiver is equivalent
to the category of
projective $\F$-linear pfd presheaves on the mesh category
with finite support.
In particular the derived categories
of any two quivers of type $A_n$ are equivalent.
In the present work we use the poset ${\M \subset \R^{\circ} \times \R}$
in place of the mesh category.
To be more precise, the quotient category $\F \M / \partial \M$
of the linearization $\F \M$ modulo $\partial \M$
can be seen as a continuous counterpart
of the mesh category of an $A_n$-quiver.
Moreover,
the category of $\F$-linear presheaves on $\F \M / \partial \M$
is equivalent to the category $\mathcal{C}$
of all functors
$\M^{\circ} \rightarrow \VectF$ vanishing on $\partial \M$,
which we work with throughout.
Furthermore,
in the same way that \cite{MR935124}
abstracts over the orientation of arrows of say an $A_n$-quiver,
we abstract over the specific topology on the real numbers.
Of course,
we cannot use any topology,
but there is a whole family of topologies
we can use in place of the Euclidean topology.
This is also closely related to the way in which
\cite{Carlsson:2009:ZPH:1542362.1542408,MR3031814}
abstract over the orientations of zigzags.
Moreover,
for finitely indexed persistence modules as in
\cite{Carlsson:2009:ZPH:1542362.1542408},
the results by \cite{MR935124} imply
that the derived category of ordinary persistence modules
and the derived category of zigzag persistence modules
(with corresponding vertex sets)
are equivalent.
This functorial extension of the correspondence of barcodes
(i.e. isomorphism classes)
by \cite{Carlsson:2009:ZPH:1542362.1542408}
through the equivalence by \cite{MR935124}
has been investigated by \cite{2020arXiv200606924H}.
In the present work we make use of sheaf theory
to extend this equivalence
(on the level of $1$-categories)
to the continuously indexed setting.

We also note that \cite{2019arXiv190709759B}
obtained a closely related invariant of continuous functions
and of sheaves on the reals,
namely \emph{Mayer--Vietoris systems}.
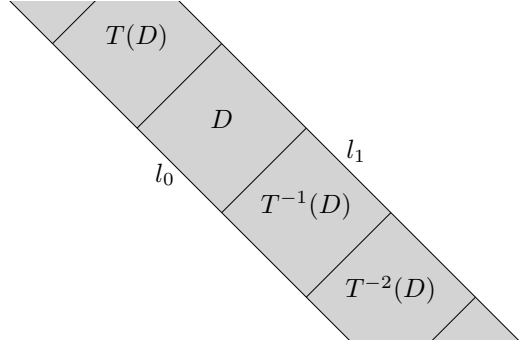
\begin{figure}[t]
  \centering
\begingroup%
  \makeatletter%
  \providecommand\color[2][]{%
    \errmessage{(Inkscape) Color is used for the text in Inkscape, but the package 'color.sty' is not loaded}%
    \renewcommand\color[2][]{}%
  }%
  \providecommand\transparent[1]{%
    \errmessage{(Inkscape) Transparency is used (non-zero) for the text in Inkscape, but the package 'transparent.sty' is not loaded}%
    \renewcommand\transparent[1]{}%
  }%
  \providecommand\rotatebox[2]{#2}%
  \newcommand*\fsize{\dimexpr\f@size pt\relax}%
  \newcommand*\lineheight[1]{\fontsize{\fsize}{#1\fsize}\selectfont}%
  \ifx\svgwidth\undefined%
    \setlength{\unitlength}{194.4375bp}%
    \ifx\svgscale\undefined%
      \relax%
    \else%
      \setlength{\unitlength}{\unitlength * \real{\svgscale}}%
    \fi%
  \else%
    \setlength{\unitlength}{\svgwidth}%
  \fi%
  \global\let\svgwidth\undefined%
  \global\let\svgscale\undefined%
  \makeatother%
  \begin{picture}(1,0.6557377)%
    \lineheight{1}%
    \setlength\tabcolsep{0pt}%
    \put(0,0){\includegraphics[width=\unitlength,page=1]{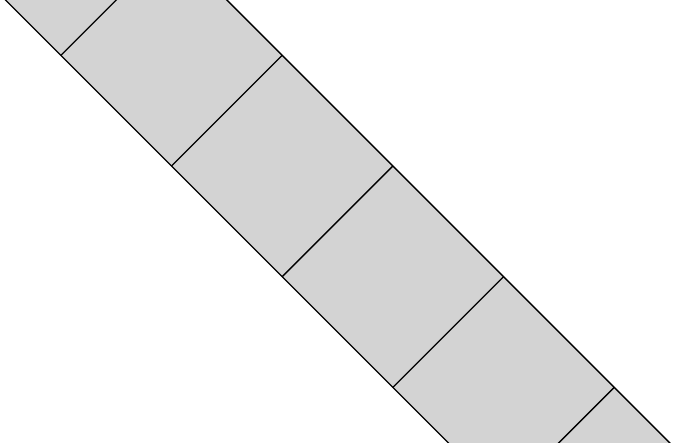}}%
    \put(0.74590164,0.08196721){\makebox(0,0)[t]{\lineheight{1.25}\smash{\begin{tabular}[t]{c}$T^{-2}(D)$\end{tabular}}}}%
    \put(0.58196721,0.24590164){\makebox(0,0)[t]{\lineheight{1.25}\smash{\begin{tabular}[t]{c}$T^{-1}(D)$\end{tabular}}}}%
    \put(0.25409836,0.57377049){\makebox(0,0)[t]{\lineheight{1.25}\smash{\begin{tabular}[t]{c}$T(D)$\end{tabular}}}}%
    \put(0.41803279,0.40983607){\makebox(0,0)[t]{\lineheight{1.25}\smash{\begin{tabular}[t]{c}$D$\end{tabular}}}}%
    \put(0.6789188,0.35222874){\makebox(0,0)[t]{\lineheight{1.25}\smash{\begin{tabular}[t]{c}$l_1$\end{tabular}}}}%
    \put(0.3084351,0.306319){\makebox(0,0)[t]{\lineheight{1.25}\smash{\begin{tabular}[t]{c}$l_0$\end{tabular}}}}%
  \end{picture}%
\endgroup%

  \caption{The tessellation of $\M$ induced by $T$ and $D$.}
  \label{fig:tessellationReals}
\end{figure}
If we consider the tessellation of $\M$
shown in \cref{fig:tessellationReals},
then we may think of the restriction of
$h(f) \colon \M^{\circ} \rightarrow \VectF$
for some function $f \colon X \rightarrow \R$
to each tile,
as a layer of the \emph{Mayer--Vietoris pyramid}
introduced by \cite{Carlsson:2009:ZPH:1542362.1542408}.
Following \cite{Carlsson:2009:ZPH:1542362.1542408}
we may further subdivide each pyramid
into their north, south, west, and east faces.
Roughly speaking,
the Mayer--Vietoris system associated to $f \colon X \rightarrow \R$
can be obtained as the pointwise dual of the restriction of
$h(f) \colon \M^{\circ} \rightarrow \VectF$
to the subposet $P \subset \M$
that is the union of all south faces -- one from each pyramid.
Furthermore,
the authors of \cite{2019arXiv190709759B}
provide a functor
$\Psi \colon D^+(\R) \rightarrow \text{M--V}(\R)$,
where $\text{M--V}(\R)$ is the category of Mayer--Vietoris systems.
In a similar way, provided an object $F$ of $D^+(\R)$
we may think of $\Psi(F) \colon P \rightarrow \VectF$
as the pointwise dual of ${h_{\R}(F) |_P \colon P^{\circ} \rightarrow \VectF}$.
Now thinking of Mayer--Vietoris systems as a subcategory
of $\VectF^P$,
we may consider the cohomological functor
\begin{equation}
  \label{eq:MVSysCoho}
  D^+(\R) \xrightarrow{\Psi}
  \text{M--V}(\R) \hookrightarrow
  \VectF^P
  .
\end{equation}
Considering that
\begin{equation*}
  D^+_t (\R) \xrightarrow{h_{\R,t}} \mathcal{J} \hookrightarrow
  \mathrm{pres}(\mathcal{J})
\end{equation*}
is the universal or initial cohomological functor
on the triangulated category $D^+_t (\R)$
by \cref{cor:abelianizationReals},
the restriction of \eqref{eq:MVSysCoho}
to $D^+_t(\R)$
has to factor through the latter universal cohomological functor.
This factorization in turn is provided by
restriction to $P \subset \M$ and pointwise dualization as described above.
We also note that while
${h_{\R,t} \colon D^+_t (\R) \rightarrow \mathcal{J}}$
is an equivalence of categories by \cref{thm:hIRequiv},
the functor $\Psi \colon D^+(\R) \rightarrow \text{M--V}(\R)$ is not.
An explicit example of ours showing that $\Psi$ is not faithful
is provided in \cite[Remark 4.9]{2019arXiv190709759B}.
In \cref{exm:hood} we show that this particular example
may appear in nature as well.

\subsection{Outline}

In \cref{sec:happel}
we describe a family of topological spaces,
which can be seen as a continuous counterpart
to the family of $A_n$-quivers.
Then we consider the associated categories of $\F$-linear sheaves
on these topological spaces
as continuously indexed counterparts
to representations of $A_n$-quivers.
Moreover, we provide a sheaf-theoretical counterpart $\iota$
to Happel's embedding \mbox{\cite[Section I.5.6]{MR935124}}
of the mesh category of a quiver into its derived category,
see also \mbox{\cite[Theorem 2.2]{MR3483626}}.
Then with \cref{prp:iota} we provide the property of the functor $\iota$,
that is the most fundamental to our work.

In \cref{sec:cohoFunctors} we describe how $\iota$ can be used
to obtain cohomological functors $\M^{\circ} \rightarrow \VectF$.
In \cref{sec:tamingCohoFunctors} we introduce a \emph{tameness} assumption
for objects in each of the categories we considered up to this point
and we show that tameness is an invariant under most of the functors
we consider between any two of these categories.
In Sections \ref{sec:altConstr} and \ref{sec:derivedLevelSet}
we connect our developments up to this point
with derived level set persistence by \cite{MR3259939,MR3873181}.

In \cref{sec:projCover} we show that any $\mathcal{J}$-presentable functor
admits a projective cover by a functor in $\mathcal{J}$
drawing upon theory and techniques of \cite[Section 3]{MR3431480}.
Later this result is strengthened to \cref{prp:minProjRes}.

In \cref{sec:equiv} we provide an equivalence of categories,
which eventually implies the equivalence of categories from \cref{thm:hIRequiv}.
To this end, we show that certain infinite coproducts of sheaves
satisfy the universal property of the product
in Sections \ref{sec:partFaithful} and \ref{sec:dirSumProd}.

In \cref{sec:abelianizationSheaves} we show that the composition of functors
\begin{equation*}
  D^+_t (\R) \xrightarrow{h_{\R,t}} \mathcal{J} \hookrightarrow
  \mathrm{pres}(\mathcal{J})
\end{equation*}
is the abelianization of $D^+_t (\R)$.
This result and its proof are inspired by \mbox{\cite[Section 4.2]{MR2355771}}.
Then we use this result to show
\cref{thm:frobenius,prp:minProjRes}.

Finally, in \cref{sec:categorification} we first show
that the Euler function
$\chi(F) \colon \op{int} \M \rightarrow \Z$
is an additive invariant
and then we prove \cref{thm:eulerIso}.
This concludes our paper and supports the idea
that RISC is a categorification of extended persistence diagrams.

\section{A Sheaf-Theoretical Happel Functor}
\label{sec:happel}

Let $\R$ and $\R^{\circ}$ denote the posets given by
the orders $\leq$ and $\geq$ on $\R$, respectively.
Then we may form the product poset $\R^{\circ} \times \R$
whose underlying set is the Euclidean plane.
Let $l_0$ and $l_1$ be two lines of slope $-1$
in $\R^{\circ} \times \R$
with $l_1$ sitting above $l_0$
as shown in \cref{fig:incidenceT}.
Moreover, let $\M$ be the convex hull of $l_0$ and $l_1$,
then $\M$ is a sublattice of $\R^{\circ} \times \R$.
The
\href{
  https://en.wikipedia.org/wiki/Center_(group_theory)
}{central} automorphism $T \colon \M \rightarrow \M$
with the following defining property
will be essential to many of our constructions
(also see \cref{fig:incidenceT}):
\begin{quote}
  Let $u \in \M$,
  $h$ be the horizontal line through $u$,
  let $g_0$ be the vertical line through $u$,
  let $h_1$ be the horizontal line through $T(u)$, and let
  $g_1$ be the vertical line through $T(u)$.
  Then the lines $l_0$, $h$, and $g_1$ intersect in a common point, and
  the same is true for the lines $l_1$, $g_0$, and $h_1$.  
\end{quote}
We also note that $T$ is a glide reflection along
the bisecting line between $l_0$ and $l_1$,
and the amount of translation is the distance
of $l_0$ and $l_1$.
Moreover, as a space, $\M / \langle T \rangle$ is a Möbius strip;
see also \cite{Carlsson:2009:ZPH:1542362.1542408}.

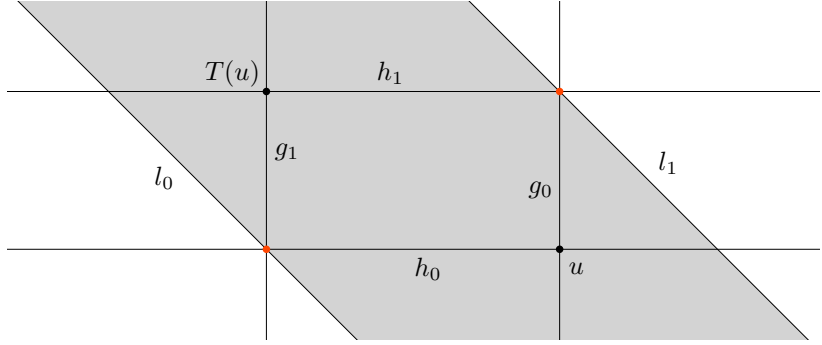
\begin{figure}[t]
  \centering
\begingroup%
  \makeatletter%
  \providecommand\color[2][]{%
    \errmessage{(Inkscape) Color is used for the text in Inkscape, but the package 'color.sty' is not loaded}%
    \renewcommand\color[2][]{}%
  }%
  \providecommand\transparent[1]{%
    \errmessage{(Inkscape) Transparency is used (non-zero) for the text in Inkscape, but the package 'transparent.sty' is not loaded}%
    \renewcommand\transparent[1]{}%
  }%
  \providecommand\rotatebox[2]{#2}%
  \newcommand*\fsize{\dimexpr\f@size pt\relax}%
  \newcommand*\lineheight[1]{\fontsize{\fsize}{#1\fsize}\selectfont}%
  \ifx\svgwidth\undefined%
    \setlength{\unitlength}{306bp}%
    \ifx\svgscale\undefined%
      \relax%
    \else%
      \setlength{\unitlength}{\unitlength * \real{\svgscale}}%
    \fi%
  \else%
    \setlength{\unitlength}{\svgwidth}%
  \fi%
  \global\let\svgwidth\undefined%
  \global\let\svgscale\undefined%
  \makeatother%
  \begin{picture}(1,0.41666667)%
    \lineheight{1}%
    \setlength\tabcolsep{0pt}%
    \put(0,0){\includegraphics[width=\unitlength,page=1]{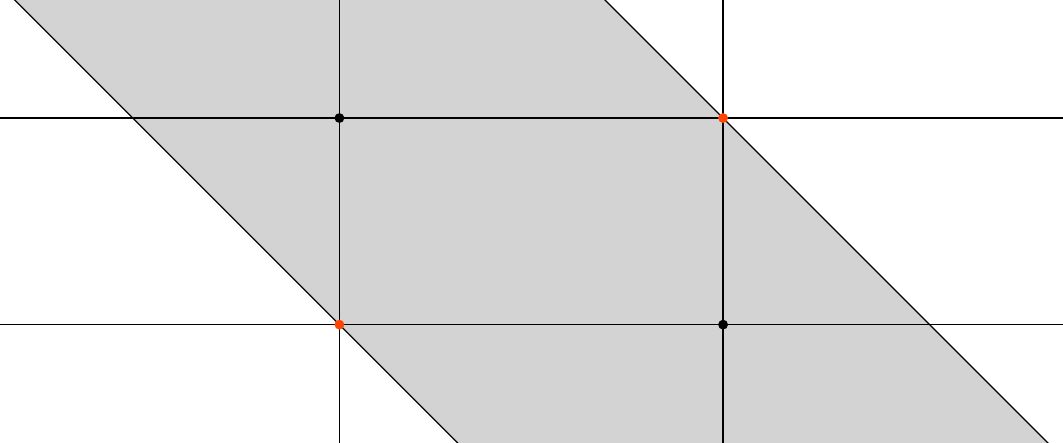}}%
    \put(0.34490735,0.22685196){\makebox(0,0)[t]{\lineheight{1.25}\smash{\begin{tabular}[t]{c}$g_1$\end{tabular}}}}%
    \put(0.4722223,0.31944436){\makebox(0,0)[t]{\lineheight{1.25}\smash{\begin{tabular}[t]{c}$h_1$\end{tabular}}}}%
    \put(0.65740735,0.18055564){\makebox(0,0)[t]{\lineheight{1.25}\smash{\begin{tabular}[t]{c}$g_0$\end{tabular}}}}%
    \put(0.51851863,0.07870368){\makebox(0,0)[t]{\lineheight{1.25}\smash{\begin{tabular}[t]{c}$h_0$\end{tabular}}}}%
    \put(0.2777777,0.31944436){\makebox(0,0)[t]{\lineheight{1.25}\smash{\begin{tabular}[t]{c}$T(u)$\end{tabular}}}}%
    \put(0.70138897,0.08101863){\makebox(0,0)[t]{\lineheight{1.25}\smash{\begin{tabular}[t]{c}$u$\end{tabular}}}}%
    \put(0.81454216,0.20860588){\makebox(0,0)[t]{\lineheight{1.25}\smash{\begin{tabular}[t]{c}$l_1$\end{tabular}}}}%
    \put(0.19392279,0.19033652){\makebox(0,0)[t]{\lineheight{1.25}\smash{\begin{tabular}[t]{c}$l_0$\end{tabular}}}}%
  \end{picture}%
\endgroup%

  \caption{Incidences defining $T$.}
  \label{fig:incidenceT}
\end{figure}

\begin{figure}[t]
  \centering
\begingroup%
  \makeatletter%
  \providecommand\color[2][]{%
    \errmessage{(Inkscape) Color is used for the text in Inkscape, but the package 'color.sty' is not loaded}%
    \renewcommand\color[2][]{}%
  }%
  \providecommand\transparent[1]{%
    \errmessage{(Inkscape) Transparency is used (non-zero) for the text in Inkscape, but the package 'transparent.sty' is not loaded}%
    \renewcommand\transparent[1]{}%
  }%
  \providecommand\rotatebox[2]{#2}%
  \newcommand*\fsize{\dimexpr\f@size pt\relax}%
  \newcommand*\lineheight[1]{\fontsize{\fsize}{#1\fsize}\selectfont}%
  \ifx\svgwidth\undefined%
    \setlength{\unitlength}{267.1875bp}%
    \ifx\svgscale\undefined%
      \relax%
    \else%
      \setlength{\unitlength}{\unitlength * \real{\svgscale}}%
    \fi%
  \else%
    \setlength{\unitlength}{\svgwidth}%
  \fi%
  \global\let\svgwidth\undefined%
  \global\let\svgscale\undefined%
  \makeatother%
  \begin{picture}(1,0.53333333)%
    \lineheight{1}%
    \setlength\tabcolsep{0pt}%
    \put(0,0){\includegraphics[width=\unitlength,page=1]{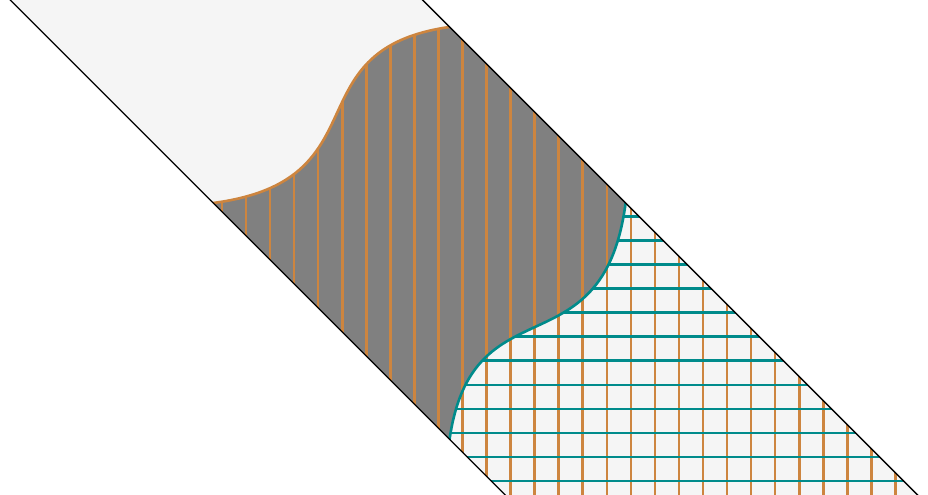}}%
    \put(0.76253053,0.26191382){\makebox(0,0)[t]{\lineheight{1.25}\smash{\begin{tabular}[t]{c}$l_1$\end{tabular}}}}%
    \put(0.28031972,0.21968028){\makebox(0,0)[t]{\lineheight{1.25}\smash{\begin{tabular}[t]{c}$l_0$\end{tabular}}}}%
  \end{picture}%
\endgroup%

  \caption{
    The fundamental domain
    $\textcolor{DimGrey}{D} :=
    \textcolor{Peru}{Q} \setminus
    \textcolor{DarkCyan}{T^{-1}(Q)}$.
  }
  \label{fig:defFundamentalDomain}
\end{figure}

\begin{figure}[t]
  \centering
\begingroup%
  \makeatletter%
  \providecommand\color[2][]{%
    \errmessage{(Inkscape) Color is used for the text in Inkscape, but the package 'color.sty' is not loaded}%
    \renewcommand\color[2][]{}%
  }%
  \providecommand\transparent[1]{%
    \errmessage{(Inkscape) Transparency is used (non-zero) for the text in Inkscape, but the package 'transparent.sty' is not loaded}%
    \renewcommand\transparent[1]{}%
  }%
  \providecommand\rotatebox[2]{#2}%
  \newcommand*\fsize{\dimexpr\f@size pt\relax}%
  \newcommand*\lineheight[1]{\fontsize{\fsize}{#1\fsize}\selectfont}%
  \ifx\svgwidth\undefined%
    \setlength{\unitlength}{194.4375bp}%
    \ifx\svgscale\undefined%
      \relax%
    \else%
      \setlength{\unitlength}{\unitlength * \real{\svgscale}}%
    \fi%
  \else%
    \setlength{\unitlength}{\svgwidth}%
  \fi%
  \global\let\svgwidth\undefined%
  \global\let\svgscale\undefined%
  \makeatother%
  \begin{picture}(1,0.6557377)%
    \lineheight{1}%
    \setlength\tabcolsep{0pt}%
    \put(0,0){\includegraphics[width=\unitlength,page=1]{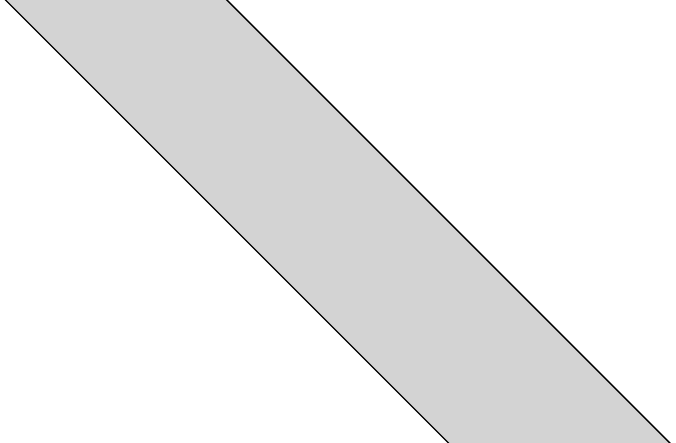}}%
    \put(0.77322392,0.06557377){\makebox(0,0)[t]{\lineheight{1.25}\smash{\begin{tabular}[t]{c}$T^{-2}(D)$\end{tabular}}}}%
    \put(0.59836066,0.21857936){\makebox(0,0)[t]{\lineheight{1.25}\smash{\begin{tabular}[t]{c}$T^{-1}(D)$\end{tabular}}}}%
    \put(0.2704918,0.54644822){\makebox(0,0)[t]{\lineheight{1.25}\smash{\begin{tabular}[t]{c}$T(D)$\end{tabular}}}}%
    \put(0.44535506,0.39344262){\makebox(0,0)[t]{\lineheight{1.25}\smash{\begin{tabular}[t]{c}$D$\end{tabular}}}}%
    \put(0,0){\includegraphics[width=\unitlength,page=2]{general-tessellation.pdf}}%
    \put(0.64823722,0.38728177){\makebox(0,0)[t]{\lineheight{1.25}\smash{\begin{tabular}[t]{c}$l_1$\end{tabular}}}}%
    \put(0.27283587,0.336454){\makebox(0,0)[t]{\lineheight{1.25}\smash{\begin{tabular}[t]{c}$l_0$\end{tabular}}}}%
  \end{picture}%
\endgroup%

  \caption{The tessellation of $\M$ induced by $T$ and $D$.}
  \label{fig:tessellation}
\end{figure}

Now let $Q \subset \M$ be a closed proper downset
and let $q := \partial Q$ be the boundary of $Q$ in $\M$.
Then $D := Q \setminus T^{-1}(Q)$
is a fundamental domain of $\M$
with respect to the action of $\langle T \rangle$,
see \cref{fig:defFundamentalDomain} for an example.
Moreover, $T$ and $D$ induce the tessellation of $\M$
shown in \cref{fig:tessellation}.
Now let
$\gamma := [0, \infty) \times (-\infty, 0] \subset \R^{\circ} \times \R$
be the downset of the origin.
As in \cite[Section 3.5]{Kashiwara1990} let $\R^2_{\gamma}$
be the plane endowed with the $\gamma$-topology
\cite[Definition 3.5.1]{Kashiwara1990}.
Moreover, let $q_{\gamma}$ be the subspace of $\R^2_{\gamma}$,
whose underlying set is $q$,
and let
\[\phi_{\gamma} \colon q \rightarrow q_{\gamma},\, t \mapsto t\]
be the natural continuous map.
Fixing a field $\F$ throughout this document,
we consider the category $\mathrm{Sh} (q_{\gamma})$
of $\F$-linear sheaves on $q_{\gamma}$.
Let $\partial q := q \cap \partial \M$
and let $j \colon \partial q \hookrightarrow q$
be the inclusion.
Moreover, let
$\mathrm{Sh}(q_{\gamma}, \partial q)$
be the full subcategory of $\F$-linear sheaves
in $\mathrm{Sh}(q_{\gamma})$ vanishing on $\partial q$.
Now considering the unit $\eta^j \colon \op{id} \to j_* \circ j^{-1}$ of the adjunction $j^{-1} \dashv j_*$ (viewed as a functor from $\mathrm{Sh} (q_{\gamma})$ to homomorphisms in $\mathrm{Sh} (q_{\gamma})$) and taking the kernel
yields a
\href{https://ncatlab.org/nlab/show/coreflective+subcategory}{coreflection}
$\flat = \op{ker} \circ \eta^j \colon \mathrm{Sh} (q_{\gamma}) \rightarrow
\mathrm{Sh}(q_{\gamma}, \partial q)$:

\begin{equation}
  \label{eq:vanishingBdCorefl}
  \begin{tikzcd}
      \mathrm{Sh}(q_{\gamma}, \partial q)
      \arrow[r, ""{name=I}, hook, bend right]
      &
      \mathrm{Sh}(q_{\gamma})
      \arrow[l, "\flat"'{name=b}, bend right]
      \arrow[phantom, from=I, to=b, "\dashv" rotate=90]
      .
  \end{tikzcd}
\end{equation}

\begin{lem}
  \label{lem:unitAtFlabbySheaf}
  If $F$ is a flabby sheaf on $q_{\gamma}$,
  then the unit
  ${\eta^j_F \colon F \rightarrow j_* j^{-1} F}$
  is an epimorphism.
\end{lem}

\begin{proof}
  Let $U \subseteq q_{\gamma}$ be an open subset.
  If $U \cap \partial q$ contains just a single point $t$,
  then we have the commutative diagram
  \begin{equation*}
    \begin{tikzcd}[column sep=8ex, row sep=5ex]
      F(U)
      \arrow[r, "\eta^j_{F,U}"]
      \arrow[rd, two heads]
      &
      (j_* j^{-1} F)(U) = (j^{-1} F)(\{t\})
      \arrow[d, equal]
      \\
      &
      F_t
      ,
    \end{tikzcd}
  \end{equation*}
  hence $\eta^j_F$ is an epimorphism at $U$.
  Now suppose $\partial q \subset U$.
  If $\partial q$ is discrete,
  then we can find disjoint open subsets $V_0, V_1 \subset U$
  with $\partial q \subset V_0 \cup V_1$
  and each containing a single point of $\partial q$.
  Moreover,
  we have the commutative diagram
  \begin{equation*}
    \begin{tikzcd}[column sep=12ex]
      F(U)
      \arrow[r, "\eta^j_{F,U}"]
      \arrow[d, two heads]
      &
      (j_* j^{-1} F)(U) = (j^{-1} F)(\partial q)
      \arrow[d, equal]
      \\
      F(V_0 \cup V_1)
      \arrow[d, "\cong"']
      \arrow[r, "\eta^j_{F, V_0 \cup V_1}"]
      &
      (j_* j^{-1} F)(V_0 \cup V_1) 
      \arrow[d, "\cong"]
      \\
      F(V_0) \oplus F(V_1)
      \arrow[r, "\eta^j_{F,V_0} \oplus \eta^j_{F,V_1}"', two heads]
      &
      (j_* j^{-1} F)(V_0) \oplus (j_* j^{-1} F)(V_1)
    \end{tikzcd}
  \end{equation*}
  with the lower horizontal map an epimorphism
  by our reasoning above.
  If $\partial q$ is not discrete,
  then $U = q_{\gamma}$ and moreover,
  ${\eta^j_{F,q_{\gamma}} \colon F(q_{\gamma}) \rightarrow
    (j_* j^{-1} F)(q_{\gamma}) = (j^{-1} F)(\partial q)}$
  is even an isomorphism.
\end{proof}

\begin{lem}
  The induced adjunction on the level of derived categories
  \begin{equation*}
    \begin{tikzcd}
      D^+ (\mathrm{Sh}(q_{\gamma}, \partial q))
      \arrow[r, ""{name=I}, bend right]
      &
      D^+ (\mathrm{Sh}(q_{\gamma}))
      \arrow[l, "R \flat"'{name=b}, bend right]
      \arrow[phantom, from=I, to=b, "\dashv" rotate=90]
    \end{tikzcd}
  \end{equation*}
  is coreflective as well.
\end{lem}

\begin{proof}
  By \cref{lem:coreflection} it suffices to show
  that all sheaves vanishing on $\partial q$
  are $\flat$-acyclic.
  To this end, we again view $\eta^j$ as an exact functor
  from $\mathrm{Sh} (q_{\gamma})$ to homomorphisms
  in $\mathrm{Sh} (q_{\gamma})$.
  In particular the subcategory of flabby sheaves is $\eta^j$-injective
  \cite[Definition 1.8.2]{Kashiwara1990}.
  Suppose that $F$ is a flabby sheaf on $q_\gamma$.
  Then $\eta^j(F) = \eta^j_F$ is an epimorphism by \cref{lem:unitAtFlabbySheaf}.
  Moreover, as $\mathrm{Sh} (q_{\gamma})$ is additive,
  the full subcategory of all epimorphisms
  in the arrow category
  is injective with respect to taking kernels
  \cite[Definition 1.8.2]{Kashiwara1990}.
  From this we obtain
  \[R \flat \cong R \ker \circ R \eta^j \cong
  R \ker \circ \, \Der (\eta^j)\]
  by \cite[Proposition 1.8.7]{Kashiwara1990}
  and by exactness of $\eta^j$.
  
  Now suppose that $G$ is an arbitrary sheaf vanishing on $\partial q$.
  Then $j^{-1} G = 0$, and so $\eta^j_G \colon G \to j_* j^{-1} G \cong 0$ is an epimorphism.
  We thus get 
  \[R \flat (G) \cong
    R \ker (\Der (\eta^j) (G)) =
    R \ker ( \eta^j_G ) \cong
    \ker (\eta^j_G) = G.\]
  In particular $G$ is $\flat$-acyclic.
  (Here the second isomorphism of the previous equation
  follows from the fact that
  $\eta^j_G \colon G \to j_* j^{-1} G$
  is a $\ker$-injective object of the arrow category.)
\end{proof}

From this point onwards we write
$D^+ (q_{\gamma})$ for the derived category
$D^+ (\mathrm{Sh}(q_{\gamma}))$
and we write
$D^+ (q_{\gamma}, \partial q)$
for the full subcategory of complexes of sheaves in $D^+ (q_{\gamma})$,
whose cohomology sheaves vanish on $\partial q$.
Then we obtain the following in conjunction with \cref{lem:cohomChar}.

\begin{cor}
  \label{cor:vanishingCoreflector}
  The category
  $D^+ (q_{\gamma}, \partial q)$
  is a triangulated coreflective subcategory of
  $D^+ (q_{\gamma})$
  with coreflector
  $R \flat \colon D^+ (q_{\gamma}) \rightarrow D^+ (q_{\gamma}, \partial q)$.
\end{cor}

Now let $\dot{q} := q_{\gamma} \setminus \partial q$
and let $i \colon \dot{q} \hookrightarrow q_{\gamma}$
be the corresponding subspace inclusion.
Then we have
\begin{equation}
  \label{eq:dirImProperSuppAsFlat}
  i_! = \flat \circ i_* \colon
  \mathrm{Sh}(\dot{q}) \rightarrow \mathrm{Sh}(q_{\gamma}, \partial q)
  .
\end{equation}
Moreover, let $\ddot{q} \subseteq q_{\gamma}$
be the smallest open subset of $q_{\gamma}$
containing $\dot{q}$
and let
$i_1 \colon \dot{q} \hookrightarrow \ddot{q}$
and
$i_2 \colon \ddot{q} \hookrightarrow q_{\gamma}$
be the corresponding inclusions.

\begin{lem}
  \label{lem:dotsAdjEquiv}
  The adjunction
  \begin{equation*}
    \begin{tikzcd}
      \mathrm{Sh}(\ddot{q})
      \arrow[r, "i_1^{-1}"'{name=la}, bend right]
      &
      \mathrm{Sh}(\dot{q})
      \arrow[l, "i_{1 *}"'{name=ra}, bend right]
      \arrow[phantom, from=la, to=ra, "\dashv" rotate=90]
    \end{tikzcd}    
  \end{equation*}
  is an exact adjoint equivalence.
\end{lem}

\begin{proof}
  The inclusion
  $i_1 \colon \dot{q} \hookrightarrow \ddot{q}$
  induces a lattice isomorphism between the topologies
  of $\dot{q}$ and $\ddot{q}$.
  Moreover,
  this isomorphism preserves all covers.
\end{proof}

\begin{lem}
  \label{lem:flatUnitIso}
  For any sheaf $F$ on $q_{\gamma}$
  the naturally induced map
  \begin{equation*}
    (\flat \circ \eta^i)_F \colon
    \flat F \rightarrow
    \flat i_* i^{-1} F = i_! i^{-1} F
  \end{equation*}
  is an isomorphism.
\end{lem}

\begin{proof}
  By \cref{lem:dotsAdjEquiv} it suffices to check that
  \begin{equation*}
    (\flat \circ \eta^{i_2})_F \colon
    \flat F \rightarrow
    \flat i_{2 *} i_2^{-1} F
  \end{equation*}
  is an isomorphism.
  This in turn follows from \cref{lem:kernelIso}.
\end{proof}

Composing the adjunctions
\begin{equation*}
  \begin{tikzcd}
    \mathrm{Sh}(q_{\gamma}, \partial q)
    \arrow[r, ""{name=I}, hook, bend right]
    &
    \mathrm{Sh}(q_{\gamma})
    \arrow[l, "\flat"'{name=b}, bend right]
    \arrow[phantom, from=I, to=b, "\dashv" rotate=90]
    \arrow[r, "i^{-1}"'{name=la}, bend right]
    &
    \mathrm{Sh}(\dot{q})
    \arrow[l, "i_*"'{name=ra}, bend right]
    \arrow[phantom, from=la, to=ra, "\dashv" rotate=90]
  \end{tikzcd}  
\end{equation*}
we obtain the adjunction
\begin{equation}
  \label{eq:qAdjProperSupp}
  \begin{tikzcd}
    \mathrm{Sh}(q_{\gamma}, \partial q)
    \arrow[r, "i^{-1}"'{name=la}, bend right]
    &
    \mathrm{Sh}(\dot{q})
    .
    \arrow[l, "i_!"'{name=ra}, bend right]
    \arrow[phantom, from=la, to=ra, "\dashv" rotate=90]
  \end{tikzcd}    
\end{equation}

\begin{lem}
  \label{lem:qAdjEquivProperSupp}
  If $\partial q$ is closed in $q_{\gamma}$,
  then \eqref{eq:qAdjProperSupp} is an exact adjoint equivalence.
\end{lem}

\begin{proof}
  By definition \mbox{\cite[Page 93]{Kashiwara1990}} of the functor
  ${(-)_{\dot{q}} \colon
    \mathrm{Sh}(q_{\gamma}) \rightarrow \mathrm{Sh}(q_{\gamma}, \partial q),\,
    F \mapsto F_{\dot{q}}}$
  we have $\flat = (-)_{\dot{q}}$.
  With this the result follows from
  \cref{lem:adjEquivProperSupp}.
\end{proof}

Now whenever $\partial q$ is closed in $q_{\gamma}$,
we define the composition
\begin{equation*}
  \varepsilon^!_F \colon
  i_! i^{-1} F \xrightarrow{(\flat \circ \eta^i)_F^{-1}}
  \flat F \xrightarrow{\varepsilon_F}
  F
\end{equation*}
for any sheaf $F$ on $q_{\gamma}$
harnessing \cref{lem:flatUnitIso},
where $\varepsilon_F \colon \flat F \rightarrow F$
is the counit of the adjunction \eqref{eq:vanishingBdCorefl}.

\begin{lem}
  \label{lem:adjDirImProperSupp}
  If $\partial q$ is closed in $q_{\gamma}$,
  then we have the $\F$-linear adjunction
  \begin{equation*}
    \begin{tikzcd}
      \mathrm{Sh}(q_{\gamma})
      \arrow[r, "i^{-1}"'{name=ra}, bend right]
      &
      \mathrm{Sh}(\dot{q})
      \arrow[l, "i_!"'{name=la}, bend right]
      \arrow[phantom, from=la, to=ra, "\dashv" rotate=-90]
    \end{tikzcd}
  \end{equation*}
  with counit
  $\varepsilon^! \colon i_! \circ i^{-1} \rightarrow \op{id}$.
\end{lem}

\begin{proof}
  Let $F$ be a sheaf on $\dot{q}$,
  let $G$ be a sheaf on $q_{\gamma}$,
  and let $\varphi \colon i_! F \rightarrow G$
  be a sheaf homomorphism.
  We have to show there is a unique sheaf homomorphism
  $\varphi^{\# !} \colon F \rightarrow i^{-1} G$
  such that the diagram
  \begin{equation}
    \label{eq:adjDirImProperSupp}
    \begin{tikzcd}[column sep=8ex, row sep=6ex]
      i_! F
      \arrow[r, "\varphi"]
      \arrow[rd, "\varphi^{\#}", dashed]
      \arrow[d, "i_! \varphi^{\# !}"']
      &
      G
      \\
      i_! i^{-1} G
      &
      \flat G
      \arrow[u, "\varepsilon_G"']
      \arrow[l, "(\flat \circ \eta^i)_G"]
    \end{tikzcd}
  \end{equation}
  commutes.
  Now the adjunction \eqref{eq:vanishingBdCorefl}
  yields a unique sheaf homomorphism
  $\varphi^{\#} \colon i_! F \rightarrow \flat G$
  as indicated in \eqref{eq:adjDirImProperSupp}.
  Moreover,
  the homomorphism
  $(\flat \circ \eta^i)_G \colon \flat G \rightarrow i_! i^{-1} G$
  is an isomorphism by \cref{lem:flatUnitIso},
  hence the vertical arrow on the left hand side
  of \eqref{eq:adjDirImProperSupp}
  exists uniquely as indicated.
  Furthermore,
  $i_! \colon \mathrm{Sh}(\dot{q}) \rightarrow \mathrm{Sh}(q_{\gamma})$
  is fully faithful by \cref{lem:qAdjEquivProperSupp}
  and this implies the claim.
\end{proof}

We now construct a functor
\[\iota_0 \colon D \rightarrow \mathrm{Sh}(q_{\gamma}, \partial q)\]
from $D$ to the category of $\F$-linear sheaves
on $q_{\gamma}$ vanishing on $\partial q$.
This will be our first step towards the construction of a functor
\[\iota \colon \M \rightarrow D^+(q_{\gamma}, \partial q),\]
which assigns an indecomposable object
$\iota(u)$ of $D^+(q_{\gamma}, \partial q)$
to each $u \in \op{int} \M$.
This functor $\iota$ is closely related to a construction
due to \mbox{\cite[Section I.5.6]{MR935124}},
see also \mbox{\cite[Theorem 2.2]{MR3483626}}.
\begin{figure}[t]
  \centering
\begingroup%
  \makeatletter%
  \providecommand\color[2][]{%
    \errmessage{(Inkscape) Color is used for the text in Inkscape, but the package 'color.sty' is not loaded}%
    \renewcommand\color[2][]{}%
  }%
  \providecommand\transparent[1]{%
    \errmessage{(Inkscape) Transparency is used (non-zero) for the text in Inkscape, but the package 'transparent.sty' is not loaded}%
    \renewcommand\transparent[1]{}%
  }%
  \providecommand\rotatebox[2]{#2}%
  \newcommand*\fsize{\dimexpr\f@size pt\relax}%
  \newcommand*\lineheight[1]{\fontsize{\fsize}{#1\fsize}\selectfont}%
  \ifx\svgwidth\undefined%
    \setlength{\unitlength}{262.08045959bp}%
    \ifx\svgscale\undefined%
      \relax%
    \else%
      \setlength{\unitlength}{\unitlength * \real{\svgscale}}%
    \fi%
  \else%
    \setlength{\unitlength}{\svgwidth}%
  \fi%
  \global\let\svgwidth\undefined%
  \global\let\svgscale\undefined%
  \makeatother%
  \begin{picture}(1,0.54372615)%
    \lineheight{1}%
    \setlength\tabcolsep{0pt}%
    \put(0,0){\includegraphics[width=\unitlength,page=1]{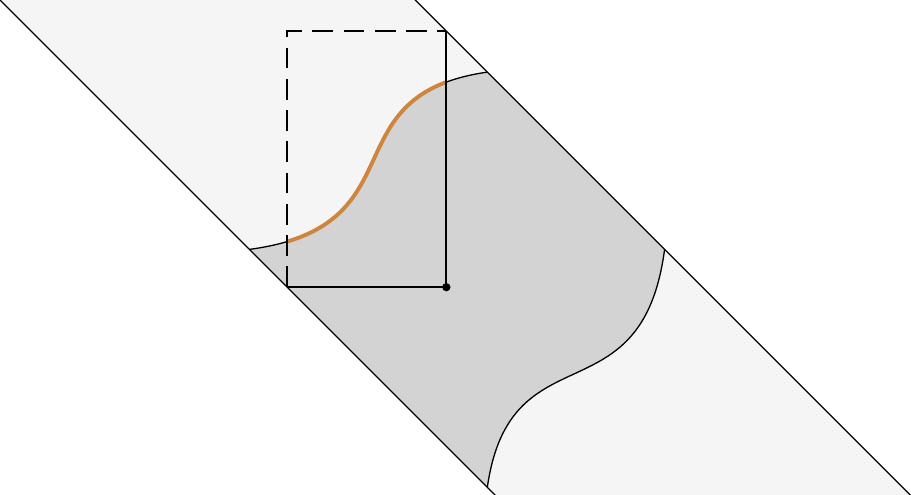}}%
    \put(0.37119755,0.37309888){\makebox(0,0)[t]{\lineheight{1.25}\smash{\begin{tabular}[t]{c}\textcolor{Peru}{$Z(u)$}\end{tabular}}}}%
    \put(0.49999998,0.2053233){\makebox(0,0)[lt]{\lineheight{1.25}\smash{\begin{tabular}[t]{l}$u$\end{tabular}}}}%
    \put(0.82084878,0.21565295){\makebox(0,0)[t]{\lineheight{1.25}\smash{\begin{tabular}[t]{c}$l_1$\end{tabular}}}}%
    \put(0.3795758,0.11852314){\makebox(0,0)[t]{\lineheight{1.25}\smash{\begin{tabular}[t]{c}$l_0$\end{tabular}}}}%
  \end{picture}%
\endgroup%

  \caption{
    The locally closed subset
    $\textcolor{Peru}{Z(u)}
    := q \cap (\uparrow u) \cap \op{int} (\downarrow T(u)).$
  }
  \label{fig:iotaSupport}
\end{figure}
For $u \in D$ we define the locally closed subset
\[Z(u) := q \cap (\uparrow u) \cap \op{int} (\downarrow T(u)).\]
Here $\uparrow u$ denotes the upset of $u$ and
$\op{int}(\downarrow T(u))$ denotes the interior of the downset of $T(u)$.
With this we may use the construction from \cite[Page 93]{Kashiwara1990}
to define
\[\iota_0(u) := \F_{Z(u)}.\]
For $u \preceq v \in D$ the homomorphism
\[\iota_0 (u \preceq v) \colon \F_{Z(u)} \rightarrow \F_{Z(v)}\]
is uniquely determined by the property that it induces
the identity on stalks whenever the corresponding stalks
are isomorphic to $\F$.
We note the following two properties of $\iota_0$, which we will need later.

\begin{lem}
  \label{lem:iota0}
  For $u, v \in D$ we have
  \begin{equation*}
    \Hom_{\mathrm{Sh}(q_{\gamma})} (\iota_0 (u), \iota_0 (v)) =
    \langle \iota_0 (u \preceq v) \rangle \cong 
    \begin{cases}
      \F
      &
      v \in (\uparrow u) \cap \op{int} (\downarrow T(u))
      \\
      \{0\}
      &
      \text{otherwise}
      .
    \end{cases}
  \end{equation*}
\end{lem}

\begin{lem}
  \label{lem:iotaLocalSections}
  Let $F$ be an $\F$-linear sheaf on $q_{\gamma}$ and let $u \in D$.
  Then we have a natural $\F$-linear isomorphism
  \begin{equation*}
    \Hom_{\mathrm{Sh}(q_{\gamma})} (\iota_0(u), F)
    \cong
    \Gamma_{q \cap (\uparrow u)} (q \cap \op{int}(\downarrow T(u)); F)
    .
  \end{equation*}
\end{lem}

\begin{proof}
  We have
  \begin{align*}
    \Hom_{\mathrm{Sh}(q_{\gamma})} (\iota_0(u), F)
    & =
      \Hom_{\mathrm{Sh}(q_{\gamma})} (\F_{Z(u)}, F)
    \\
    & =
      \Gamma(q_{\gamma}; \cHom(\F_{Z(u)}, F))
    \\
    & \cong
      \Gamma(q_{\gamma}; \Gamma_{Z(u)} (F))
    \\
    & =
      \Gamma(
      q_{\gamma};
      \Gamma_{q \cap \op{int}(\downarrow T(u))} (\Gamma_{q \cap (\uparrow u)} (F))
      )
    \\
    & =
      \Gamma_{q \cap (\uparrow u)} (q \cap \op{int}(\downarrow T(u)); F)
      .
  \end{align*}
  Here $\cHom$ denotes the internal homomorphism functor of sheaves
  as defined in \cite[Definition 2.2.7]{Kashiwara1990}.
  The second equality follows from \cite[(2.2.6)]{Kashiwara1990}.
  The isomorphism follows from \cite[(2.3.16)]{Kashiwara1990}.
  The fourth equality follows from \cite[Proposition 2.3.9.(ii)]{Kashiwara1990}
  and the last equality from
  \cite[Proposition 2.3.9.(iii) and (2.3.12)]{Kashiwara1990}.
\end{proof}

Now let $u \in D$.
As $\phi_\gamma \colon q \rightarrow q_{\gamma}$ is continuous,
the subset $Z(u) = \phi_\gamma^{-1} (Z(u)) \subseteq q$
is locally closed in $q$ as well.
Moreover, we have the pair of adjoint functors
\begin{equation*}
  \begin{tikzcd}
    \mathrm{Sh}(q_{\gamma})
    \arrow[r, "\phi_{\gamma}^{-1}"'{name=F}, bend right]
    &
    \mathrm{Sh}(q)
    \arrow[l, "\phi_{\gamma *}"'{name=G}, bend right]
    \arrow[phantom, from=F, to=G, "\dashv" rotate=90]
  \end{tikzcd}
\end{equation*}
between categories of $\F$-linear sheaves on $q$ and $q_\gamma$.
Thus, we have
\[
  (\phi_\gamma^{-1} \circ \iota_0)(u) =
  \phi_\gamma^{-1} \F_{Z(u)} \cong \F_{Z(u)}.
\]
Furthermore, the adjunction unit
\begin{equation}
  \label{eq:phiUnit}
  \left(\eta \circ \iota_0\right)_u \colon
  \F_{Z(u)} \rightarrow \phi_{\gamma *} \phi_\gamma^{-1} \F_{Z(u)}
  \cong \phi_{\gamma *} \F_{Z(u)}
\end{equation}
is an isomorphism,
hence $\phi_\gamma^{-1}$ is fully faithful
when restricted to the essential image of $\iota_0$.
It will be convenient
to have a description of ${\phi_\gamma^{-1} \circ \iota_0}$
in terms of complexes of sheaves that are $\Gamma(I; -)$-acyclic \cite[Exercise I.19]{Kashiwara1990}
for any connected open subset $I \subseteq q$.
More specifically,
we resolve each sheaf $\F_{Z(u)}$ in $\mathrm{Sh} (q)$ for $u \in D$
by sheaves of the form $\F_A$,
where $A \subseteq q$ is closed
with at most two connected components
as opposed to being merely locally closed.
Such a sheaf $\F_A$ is $\Gamma(I; -)$-acyclic
as $\F_A |_I$ can be written as the pushforward of the locally constant sheaf
on $A \cap I$ along the inclusion $A \cap I \hookrightarrow I$
and by the homotopy invariance
of sheaf cohomology \cite[Proposition 2.7.5]{Kashiwara1990}.
To this end, we have the following generalization
of \cite[Proposition 2.1.1]{2020arXiv200701834B}.

\begin{prp}
  \label{prp:rho}
  Let $\mathcal{P}$ denote the set of pairs of closed subspaces
  of $q$.
  Then there is a unique order-reversing map
  \begin{equation*}
    \rho \colon D^{\circ} \rightarrow \mathcal{P}
  \end{equation*}
  with the following three properties:
  \begin{enumerate}[(1)]
  \item
    For any $t \in q$ we have
    $\rho(t) = (q \cap (\uparrow t), \partial q \cap (\uparrow t))$,
    where $\partial q := q \cap \partial \M$
    and $\partial \M := l_0 \cup l_1$.
  \item
    For any $u \in D \cap \partial \M$
    the two components of $\rho(u)$ are identical.
  \item
    For any axis-aligned rectangle contained in $D^{\circ}$
    the corresponding joins and meets are preserved by $\rho$.
    (A join in $D^{\circ}$ is a meet in $D$ and vice versa.)
  \end{enumerate}  
\end{prp}

More concretely, the map $\rho$ can be described by the formula
\begin{equation}
  \label{eq:rhoExplicit}
  \rho(u) = (\rho_0(u), \rho_1(u)),
  \quad \text{where} ~~
  \begin{cases}
    \rho_0 (u) = q \cap (\uparrow u)
    \qquad \text{and}
    \\
    \rho_1 (u) = q \setminus \op{int} (\downarrow T(u))
    ,
  \end{cases}
\end{equation}
for any $u \in D^{\circ}$.
Moreover,
we have
\begin{equation}
  \label{eq:Zrho}
  Z(u) = \rho_0 (u) \setminus \rho_1 (u)
  .
\end{equation}
Now for $u \in D$ let $\kappa_0 (u)$
be the complex of sheaves with
\begin{equation*}
  \left(\kappa_0 (u)\right)^n =
  \begin{cases}
    \F_{\rho_0 (u)} & n = 0 \\
    \F_{\rho_1 (u)} & n = 1 \\
    0            & \text{otherwise}
  \end{cases}
\end{equation*}
and the differential
\[\delta^0_u \colon \F_{\rho_0 (u)} \rightarrow \F_{\rho_1 (u)}\]
being the homomorphism
which induces the identity on all stalks isomorphic to $\F$.
Similarly, we may extend $\kappa_0$ to a functor
\begin{equation*}
  \kappa_0 \colon D \rightarrow \Com (q) := \Com (\mathrm{Sh}(q))
\end{equation*}
from $D$ to the category of complexes
of $\F$-linear sheaves on $q$.
In conjunction with \eqref{eq:Zrho} we obtain
\begin{equation}
  \label{eq:acyclicRes}
  H^0 \circ \kappa_0 \cong \phi_\gamma^{-1} \circ \iota_0
  \quad \text{and} \quad
  H^1 \circ \kappa_0 \cong 0
  .
\end{equation}
This equation yields the following two lemmas.

\begin{lem}
  \label{lem:exact}
  For any axis-aligned rectangle $u \preceq v_1, v_2 \preceq w \in D$
  as shown in \cref{fig:constrBoundary} the sequence
  \begin{equation}
    \label{eq:exact}
    0 \rightarrow
    \iota_0 (u) \xrightarrow{\begin{pmatrix} 1 \\ 1 \end{pmatrix}}
    \iota_0 (v_1) \oplus \iota_0 (v_2)
    \xrightarrow{\begin{pmatrix} 1 & -1 \end{pmatrix}}
    \iota_0 (w) \rightarrow
    0
  \end{equation}
  is exact.
  Here we use $1$ as a shorthand for the corresponding induced map,
  e.g. $\iota_0 (u \preceq v_1)$.
\end{lem}

\begin{proof}
  By \cref{prp:rho}.(3) the sequence of complexes
  \begin{equation*}
    0 \rightarrow
    \kappa_0 (u) \xrightarrow{\begin{pmatrix} 1 \\ 1 \end{pmatrix}}
    \kappa_0 (v_1) \oplus \kappa_0 (v_2)
    \xrightarrow{\begin{pmatrix} 1 & -1 \end{pmatrix}}
    \kappa_0 (w) \rightarrow
    0
  \end{equation*}
  is exact.
  In conjunction with \eqref{eq:acyclicRes} and the
  \href{
    https://en.wikipedia.org/wiki/Snake_lemma
  }{Snake Lemma}
  we obtain the exactness of
  \begin{equation*}
    0 \rightarrow
    \phi_\gamma^{-1} \iota_0 (u)
    \xrightarrow{\begin{pmatrix} 1 \\ 1 \end{pmatrix}}
    \phi_\gamma^{-1} \iota_0 (v_1) \oplus \phi_\gamma^{-1} \iota_0 (v_2)
    \xrightarrow{\begin{pmatrix} 1 & -1 \end{pmatrix}}
    \phi_\gamma^{-1} \iota_0 (w) \rightarrow
    0
    .
  \end{equation*}
  As we can check the exactness of \eqref{eq:exact}
  on the level of stalks,
  this is sufficient.
\end{proof}

\begin{lem}
  \label{lem:derUnit}
  For any $u \in D$ the derived unit
  \begin{equation*}
    \left(\eta^{\Der (q_\gamma)} \circ \iota_0\right)_u \colon
    \F_{Z(u)} \rightarrow R \phi_{\gamma *} \phi_\gamma^{-1} \F_{Z(u)}
    \cong R \phi_{\gamma *} \F_{Z(u)}
  \end{equation*}
  is an isomorphism of the derived category
  $\Der (q_\gamma)$,
  where we view $\F_{Z(u)}$ as a complex concentrated in degree $0$.
\end{lem}

\begin{proof}
  \sloppy
  As $\left(\eta \circ \iota_0\right)_u$
  is an isomorphism
  it suffices to show that
  $\phi_\gamma^{-1} \F_{Z(u)} \cong \F_{Z(u)}$
  is \mbox{$\phi_{\gamma *}$-acyclic}
  by \cref{lem:coreflection} from \cref{sec:homologicalAlg}.
  To this end, it suffices to check
  that ${\F_{Z(u)} \in \mathrm{Sh} (q)}$ is $\Gamma(U; -)$-acyclic
  for each open subset $U$ of some basis $\mathcal{B}$ of $q_{\gamma}$
  by \cref{lem:compose}.
  Now let
  $\mathcal{B} := \{q \cap \op{int}(\downarrow v) \mid T(u) \not\preceq v\}$
  be our choice of a basis and let
  $U = q \cap \op{int}(\downarrow v) \in \mathcal{B}$.
  As $U$ is connected, $\kappa_0 (u)$ is a
  $\Gamma(U; -)$-acyclic resolution of $\F_{Z(u)}$.
  Moreover, as $T(u) \not\preceq v$, the differential
  $\Gamma(U; \delta^0_u)$ is surjective,
  hence $\F_{Z(u)}$ is \mbox{$\Gamma(U; -)$-acyclic}.
\end{proof}

\begin{figure}[t]
  \centering
\begingroup%
  \makeatletter%
  \providecommand\color[2][]{%
    \errmessage{(Inkscape) Color is used for the text in Inkscape, but the package 'color.sty' is not loaded}%
    \renewcommand\color[2][]{}%
  }%
  \providecommand\transparent[1]{%
    \errmessage{(Inkscape) Transparency is used (non-zero) for the text in Inkscape, but the package 'transparent.sty' is not loaded}%
    \renewcommand\transparent[1]{}%
  }%
  \providecommand\rotatebox[2]{#2}%
  \newcommand*\fsize{\dimexpr\f@size pt\relax}%
  \newcommand*\lineheight[1]{\fontsize{\fsize}{#1\fsize}\selectfont}%
  \ifx\svgwidth\undefined%
    \setlength{\unitlength}{307.5bp}%
    \ifx\svgscale\undefined%
      \relax%
    \else%
      \setlength{\unitlength}{\unitlength * \real{\svgscale}}%
    \fi%
  \else%
    \setlength{\unitlength}{\svgwidth}%
  \fi%
  \global\let\svgwidth\undefined%
  \global\let\svgscale\undefined%
  \makeatother%
  \begin{picture}(1,0.48780488)%
    \lineheight{1}%
    \setlength\tabcolsep{0pt}%
    \put(0,0){\includegraphics[width=\unitlength,page=1]{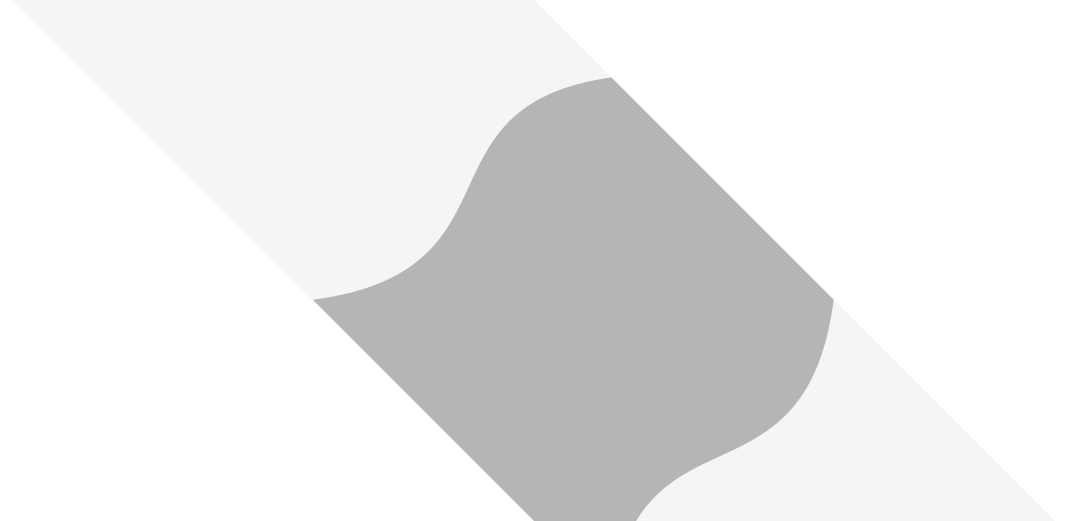}}%
    \put(0.36585366,0.34959341){\makebox(0,0)[t]{\lineheight{1.25}\smash{\begin{tabular}[t]{c}$\hat{u}$\end{tabular}}}}%
    \put(0.27439024,0.40243902){\makebox(0,0)[t]{\lineheight{1.25}\smash{\begin{tabular}[t]{c}$T(D)$\end{tabular}}}}%
    \put(0.51219512,0.2195122){\makebox(0,0)[t]{\lineheight{1.25}\smash{\begin{tabular}[t]{c}$w$\end{tabular}}}}%
    \put(0.67886171,0.24390244){\makebox(0,0)[t]{\lineheight{1.25}\smash{\begin{tabular}[t]{c}$v_2$\end{tabular}}}}%
    \put(0.49187,0.08943098){\makebox(0,0)[t]{\lineheight{1.25}\smash{\begin{tabular}[t]{c}$v_1$\end{tabular}}}}%
    \put(0.66666659,0.09349585){\makebox(0,0)[t]{\lineheight{1.25}\smash{\begin{tabular}[t]{c}$u$\end{tabular}}}}%
    \put(0,0){\includegraphics[width=\unitlength,page=2]{general-constr-boundary.pdf}}%
    \put(0.77409512,0.25273415){\makebox(0,0)[t]{\lineheight{1.25}\smash{\begin{tabular}[t]{c}$l_1$\end{tabular}}}}%
    \put(0.23334146,0.21787805){\makebox(0,0)[t]{\lineheight{1.25}\smash{\begin{tabular}[t]{c}$l_0$\end{tabular}}}}%
  \end{picture}%
\endgroup%

  \caption{
    An axis-aligned rectangle $u \preceq v_1, v_2 \preceq w \in D$
    and the corresponding pair $(w, \hat{u}) \in R_D$.
  }
  \label{fig:constrBoundary}
\end{figure}

\sloppy
The short exact sequence \eqref{eq:exact}
yields a particular distinguished triangle
\begin{equation}
  \label{eq:triangle}
  \iota_0 (u) \xrightarrow{\begin{pmatrix} 1 \\ 1 \end{pmatrix}}
  \iota_0 (v_1) \oplus \iota_0 (v_2)
  \xrightarrow{\begin{pmatrix} 1 & -1 \end{pmatrix}}
  \iota_0 (w) \xrightarrow{\partial'}
  \iota_0 (u)[1]
\end{equation}
in the derived category
$\Der (q_\gamma)$,
where we view each sheaf as a complex concentrated in degree $0$.
Now let $\Sigma := (-)[1]$ and let
\[R_D :=
  \{
  (w, \hat{u}) \in D \times T(D) \mid
  w \preceq \hat{u} \preceq T(w)
  \}\]
as in \cite[Definition A.8]{2021arXiv210809298B}.
As shown in \cref{fig:constrBoundary}, any pair $(w, \hat{u}) \in R_D$
determines an axis-aligned rectangle $u \preceq v_1, v_2 \preceq w \in D$
with $T(u) = \hat{u}$.
Moreover, the collection of homomorphisms $\partial'$
for the corresponding triangles \eqref{eq:triangle}
forms a natural transformation as in the diagram
\begin{equation*}
  \begin{tikzcd}
    R_D
    \arrow[rrr, "\op{pr}_1"]
    \arrow[ddd, "\op{pr}_2"']
    &[-23pt] & &[-18pt]
    D
    \arrow[ddd, "\iota_0"]
    \\[-17pt]
    & &
    {}
    \arrow[ld, Rightarrow, "{\partial'}"]
    \\
    &
    {}
    \\[-20pt]
    T(D)
    \arrow[rrr, "\Sigma \circ \iota_0 \circ T^{-1}"']
    & & &
    \Der (q_{\gamma})
    ,
  \end{tikzcd}
\end{equation*}
where we consider
$R_D$ as a subposet of $D \times T(D)$ with the product order and with
${\op{pr}_1 \colon R_D 
  \rightarrow D}$ and
${\op{pr}_2 \colon R_D \rightarrow T(D)}$ being the canonical projections
from $R_D$ to the first and second component, respectively.
Thus, \cite[Proposition A.14]{2021arXiv210809298B}
yields a functor
\[\iota \colon \M \rightarrow \Der (q_\gamma, \partial q)\]
that is a \emph{strictly stable} extension
of
$\iota_0 \colon
D \rightarrow
\mathrm{Sh} (q_\gamma, \partial q) \hookrightarrow
\Der (q_\gamma, \partial q)$
in the sense that
\begin{equation*}
  \iota \circ T = \Sigma \circ \iota = \iota(-)[1]
  \quad \text{and} \quad
  \iota |_D = \iota_0
  .
\end{equation*}
We show the following counterpart to \cref{lem:iota0}.

\begin{prp}
  \label{prp:iota}
  For $u, v \in \M$ we have
  \begin{equation*}
    \Hom_{\Der (q_{\gamma})} (\iota (u), \iota (v)) =
    \langle \iota (u \preceq v) \rangle
    \cong
    \begin{cases}
      \F
      &
      v \in (\uparrow u) \cap \op{int} (\downarrow T(u))
      \\
      \{0\}
      &
      \text{otherwise}
      .
    \end{cases}
  \end{equation*}
\end{prp}

The long exact sequence from the following lemma
is one ingredient to our proof of this proposition.

\begin{lem}
  \label{lem:les}
  For all $u, w \in D$ there is a long exact sequence
  \begin{equation}
    \label{eq:les}
    \!\!\!\!\!\!\!\!\!\!\!\!\!\!\!\!\!\!\!\!\!\!\!\!\!\!\!\!\!
    \begin{tikzcd}[row sep=6ex, column sep=2.5ex]
      &
      \Hom_{\Der (q)} (\F_{Z(u)}, \F_{Z(w)} [3])
      \arrow[r]
      &
      0
      \arrow[r]
      &
      \cdots
      ,
      \\
      &
      \Hom_{\Der (q)} (\F_{Z(u)}, \F_{Z(w)} [2])
      \arrow[r]
      &
      0
      \arrow[r]
      &
      0
      \arrow[ull, out=0, in=180, looseness=1.4, overlay]      
      \\
      &
      \Hom_{\Der (q)} (\F_{Z(u)}, \F_{Z(w)} [1])
      \arrow[r]
      &
      \op{coker} \Gamma \left(U; \delta^0_w\right)
      \arrow[r]
      &
      \op{coker} \Gamma \left(V; \delta^0_w\right)
      \arrow[ull, out=0, in=180, looseness=1.4, overlay]      
      \\
      0
      \arrow[r]
      &
      \Hom_{\Der (q)} (\F_{Z(u)}, \F_{Z(w)})
      \arrow[r]
      &
      \Gamma \left(U; \F_{Z(w)}\right)
      \arrow[r]
      &
      \Gamma \left(V; \F_{Z(w)}\right)
      \arrow[ull, out=0, in=180, looseness=1.4, overlay]
    \end{tikzcd}
  \end{equation}
  where $U := q \cap \op{int}(\downarrow T(u))$
  and $V := q \setminus (\uparrow u)$.
\end{lem}

\begin{proof}
  We have $V = U \setminus Z(u)$
  and thus there is a distinguished triangle
  \begin{equation*}
    R \Gamma_{Z(u)} \left(\F_{Z(w)}\right)
    \rightarrow
    R \Gamma_{U} \left(\F_{Z(w)}\right)
    \rightarrow
    R \Gamma_{V} \left(\F_{Z(w)}\right)
    \rightarrow
    R \Gamma_{Z(u)} \left(\F_{Z(w)}\right) [1]
  \end{equation*}
  by \cite[Equation (2.6.32)]{Kashiwara1990}.
  Applying the \emph{cohomological} functor
  \cite[Definition 1.5.2]{Kashiwara1990}
  $H^0 (q; -)$ to this triangle we obtain the long exact sequence
  \begin{equation*}
    \!\!\!\!\!\!\!\!\!\!\!\!\!\!\!\!\!\!\!\!\!
    \begin{tikzcd}[row sep=6ex, column sep=2.5ex]
      &
      H^2_{Z(u)} \left(q; \F_{Z(w)}\right)
      \arrow[r]
      &
      H^2_{U} \left(q; \F_{Z(w)}\right)
      \arrow[r]
      &
      \cdots
      .
      \\
      &
      H^1_{Z(u)} \left(q; \F_{Z(w)}\right)
      \arrow[r]
      &
      H^1_{U} \left(q; \F_{Z(w)}\right)
      \arrow[r]
      &
      H^1_{V} \left(q; \F_{Z(w)}\right)
      \arrow[ull, out=0, in=180, looseness=1.4, overlay]      
      \\
      0
      \arrow[r]
      &
      H^0_{Z(u)} \left(q; \F_{Z(w)}\right)
      \arrow[r]
      &
      H^0_{U} \left(q; \F_{Z(w)}\right)
      \arrow[r]
      &
      H^0_{V} \left(q; \F_{Z(w)}\right)
      \arrow[ull, out=0, in=180, looseness=1.4, overlay]      
    \end{tikzcd}
  \end{equation*}
  Now by \cite[Equation (2.6.9)]{Kashiwara1990} we have
  \begin{equation*}
    H^n_{Z(u)} \left(q; \F_{Z(w)}\right) \cong
    \Hom_{\Der (q)} \left(\F_{Z(u)}, \F_{Z(w)} [n]\right)
    .
  \end{equation*}
  By substituting the corresponding terms in above exact sequence
  we obtain an exact sequence whose left column coincides
  with the sequence \eqref{eq:les}.
  Moreover,
  by \mbox{\cite[Proposition 2.3.9.(iii) and Remark 2.6.9]{Kashiwara1990}}
  we have
  \begin{equation*}
    H^n_{U} \left(q; \F_{Z(w)}\right) \cong
    H^n \left(U; \F_{Z(w)} |_U\right) \cong
    H^n \left(U; \F_{Z(w)}\right)
  \end{equation*}
  and similarly
  \begin{equation*}
    H^n_{V} \left(q; \F_{Z(w)}\right) \cong
    H^n \left(V; \F_{Z(w)}\right)
  \end{equation*}
  for all $n \in \N_0$.
  Furthermore,
  as $U$ is open and connected and $V$ is a disjoint union of at most
  two connected open subsets of $q$,
  the complex $\kappa_0 (w)$ is a resolution of $\F_{Z(w)}$
  by $\Gamma(U; -)$-acyclic and $\Gamma(V; -)$-acyclic sheaves
  by \eqref{eq:acyclicRes}.
  Thus, we may replace each term in the center and right column
  of above exact sequence by
  $H^n (\Gamma(U; \kappa_0 (w))$ and $H^n (\Gamma(V; \kappa_0 (w))$,
  respectively.
  Simplifying these terms using \eqref{eq:acyclicRes}
  for each $n \in \N_0$ we obtain
  the long exact sequence \eqref{eq:les}.
\end{proof}

To make use of the long exact sequence from the previous lemma,
we need to understand the map
$\op{coker} \Gamma \left(U; \delta^0_w\right) \rightarrow
\op{coker} \Gamma \left(V; \delta^0_w\right)$
from this sequence.

\begin{lem}
  \label{lem:auxIsoCoker}
  Suppose we have $u, w \in D$ with $w \preceq u \not\preceq T(w)$
  and $U$ and $V$ as in the previous lemma,
  then the linear map
  $\op{coker} \Gamma \left(U; \delta^0_w\right) \rightarrow
  \op{coker} \Gamma \left(V; \delta^0_w\right)$
  is an isomorphism.
\end{lem}

\begin{proof}
  We distinguish between two cases.
  If ${q \subset (\uparrow w)}$,
  then two connected components of $\rho_1 (w)$
  are contained in the same connected component of $\rho_0 (w)$.
  As $u \notin [w, T(w)]$ this property still holds
  after we intersect both sets with $U$ or $V$.
  Thus, when we apply the cokernel to the map of maps
  $\Gamma \left(U; \delta^0_w\right) \rightarrow
  \Gamma \left(V; \delta^0_w\right)$,
  we obtain the identity
  $K \xrightarrow{\op{id}} K$
  (up to isomorphism of maps).
  If $q \not\subset (\uparrow w)$ both cokernels vanish.
\end{proof}

Using the previous two lemmas
we show the following \enquote{vanishing theorem}.

\begin{lem}\label{lem:vanishing}
  If we have $u, v \in \M$ with
  $v \notin (\uparrow u) \cap \op{int} (\downarrow T(u))$,
  then
  \begin{equation*}
    \Hom_{\Der (q_{\gamma})} (\iota (u), \iota (v)) \cong \{0\}
    .
  \end{equation*}
\end{lem}

\begin{figure}[t]
  \centering
\begingroup%
  \makeatletter%
  \providecommand\color[2][]{%
    \errmessage{(Inkscape) Color is used for the text in Inkscape, but the package 'color.sty' is not loaded}%
    \renewcommand\color[2][]{}%
  }%
  \providecommand\transparent[1]{%
    \errmessage{(Inkscape) Transparency is used (non-zero) for the text in Inkscape, but the package 'transparent.sty' is not loaded}%
    \renewcommand\transparent[1]{}%
  }%
  \providecommand\rotatebox[2]{#2}%
  \newcommand*\fsize{\dimexpr\f@size pt\relax}%
  \newcommand*\lineheight[1]{\fontsize{\fsize}{#1\fsize}\selectfont}%
  \ifx\svgwidth\undefined%
    \setlength{\unitlength}{265.9914093bp}%
    \ifx\svgscale\undefined%
      \relax%
    \else%
      \setlength{\unitlength}{\unitlength * \real{\svgscale}}%
    \fi%
  \else%
    \setlength{\unitlength}{\svgwidth}%
  \fi%
  \global\let\svgwidth\undefined%
  \global\let\svgscale\undefined%
  \makeatother%
  \begin{picture}(1,0.47933879)%
    \lineheight{1}%
    \setlength\tabcolsep{0pt}%
    \put(0,0){\includegraphics[width=\unitlength,page=1]{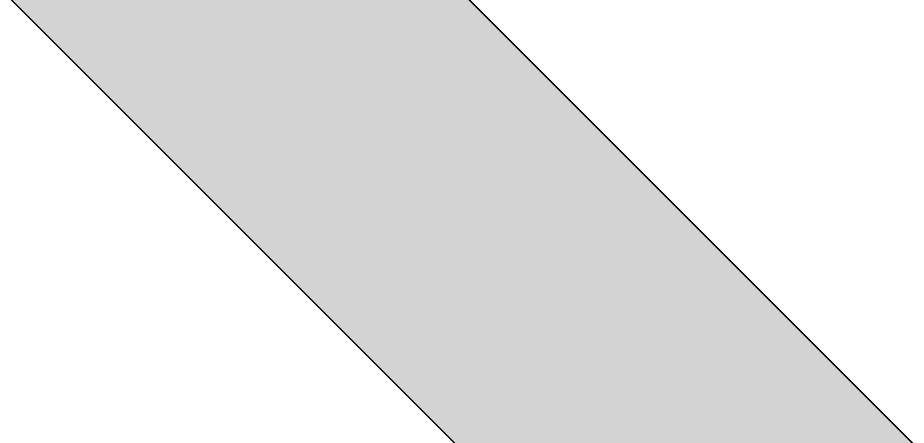}}%
    \put(0.48398753,0.2525825){\makebox(0,0)[t]{\lineheight{1.25}\smash{\begin{tabular}[t]{c}$q$\end{tabular}}}}%
    \put(0.35950409,0.32644635){\makebox(0,0)[t]{\lineheight{1.25}\smash{\begin{tabular}[t]{c}$T(w)$\end{tabular}}}}%
    \put(0.55371892,0.17768581){\makebox(0,0)[t]{\lineheight{1.25}\smash{\begin{tabular}[t]{c}$u$\end{tabular}}}}%
    \put(0.69421481,0.06198358){\makebox(0,0)[t]{\lineheight{1.25}\smash{\begin{tabular}[t]{c}$w$\end{tabular}}}}%
    \put(0,0){\includegraphics[width=\unitlength,page=2]{vanishing-coker.pdf}}%
  \end{picture}%
\endgroup%

  \caption{Here $u \in [w, T(w)]$ and $V \cap \rho_1 (w)$
    has two connected components.}
  \label{fig:vanishingCoker}
\end{figure}

\begin{figure}[t]
  \centering
\begingroup%
  \makeatletter%
  \providecommand\color[2][]{%
    \errmessage{(Inkscape) Color is used for the text in Inkscape, but the package 'color.sty' is not loaded}%
    \renewcommand\color[2][]{}%
  }%
  \providecommand\transparent[1]{%
    \errmessage{(Inkscape) Transparency is used (non-zero) for the text in Inkscape, but the package 'transparent.sty' is not loaded}%
    \renewcommand\transparent[1]{}%
  }%
  \providecommand\rotatebox[2]{#2}%
  \newcommand*\fsize{\dimexpr\f@size pt\relax}%
  \newcommand*\lineheight[1]{\fontsize{\fsize}{#1\fsize}\selectfont}%
  \ifx\svgwidth\undefined%
    \setlength{\unitlength}{264.50001526bp}%
    \ifx\svgscale\undefined%
      \relax%
    \else%
      \setlength{\unitlength}{\unitlength * \real{\svgscale}}%
    \fi%
  \else%
    \setlength{\unitlength}{\svgwidth}%
  \fi%
  \global\let\svgwidth\undefined%
  \global\let\svgscale\undefined%
  \makeatother%
  \begin{picture}(1,0.65217388)%
    \lineheight{1}%
    \setlength\tabcolsep{0pt}%
    \put(0,0){\includegraphics[width=\unitlength,page=1]{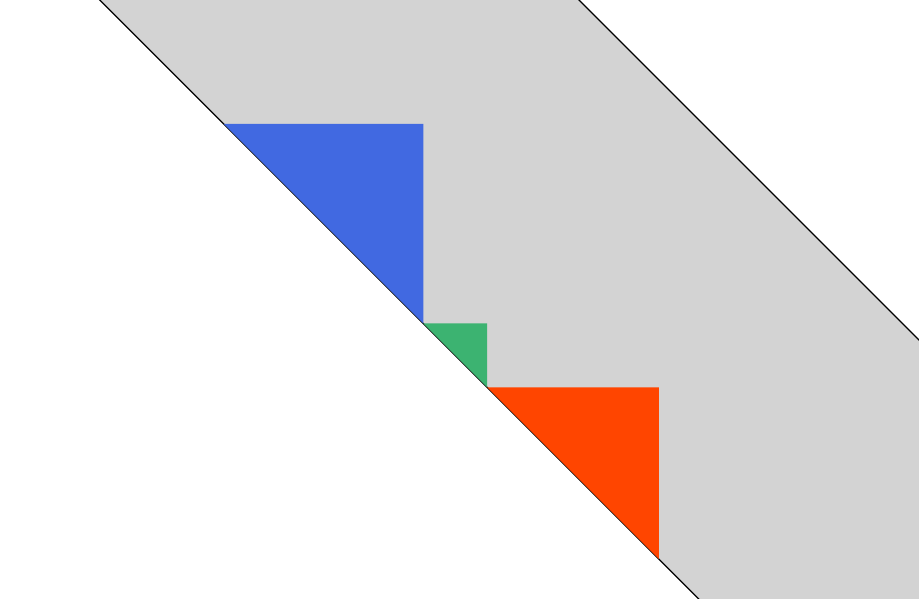}}%
    \put(0.53043485,0.57826074){\makebox(0,0)[t]{\lineheight{1.25}\smash{\begin{tabular}[t]{c}$v$\end{tabular}}}}%
    \put(0.44608693,0.53478258){\makebox(0,0)[t]{\lineheight{1.25}\smash{\begin{tabular}[t]{c}$T(u)$\end{tabular}}}}%
    \put(0.74565223,0.22391313){\makebox(0,0)[t]{\lineheight{1.25}\smash{\begin{tabular}[t]{c}$w$\end{tabular}}}}%
    \put(0.78695638,0.29347824){\makebox(0,0)[t]{\lineheight{1.25}\smash{\begin{tabular}[t]{c}$u$\end{tabular}}}}%
    \put(0,0){\includegraphics[width=\unitlength,page=2]{vanishing.pdf}}%
  \end{picture}%
\endgroup%

  \caption{Here $u \preceq v$ factors through a point in $\partial \M$.}
  \label{fig:vanishing}
\end{figure}

\begin{proof}
  By \cref{lem:derUnit} it suffices to show that
  \begin{equation*}
    \Hom_{\Der (q)} (\phi_\gamma^{-1} \iota (u), \phi_\gamma^{-1} \iota (v))
    \cong \{0\}
    .
  \end{equation*}
  Without loss of generality we assume $u \in D$.
  For
  \[v \notin D \cup T(D) \cup T^2 (D)\]
  we have
  \begin{equation*}
    \Hom_{\Der (q)} (\phi_\gamma^{-1} \iota (u), \phi_\gamma^{-1} \iota (v))
    \cong
    \Hom_{\Der (q)} (\F_{Z(u)}, \F_{Z(w)} [n])
  \end{equation*}
  for some $w \in D$ and $n \in \Z \setminus \{0, 1, 2\}$,
  and thus the statement
  follows directly from \cref{lem:les} in this case.
  If $v \in D$, then the statement follows from \cref{lem:iota0}.
  
  Now suppose we have $v \in T^2 (D)$, then let
  $w := T^{-2}(v)$ and let
  $U$ and $V$ be as in the previous two lemmas.
  We have to show that
  \begin{equation*}
    \Hom_{\Der (q)} (\phi_\gamma^{-1} \iota (u), \phi_\gamma^{-1} \iota (v))
    \cong
    \Hom_{\Der (q)} (\F_{Z(u)}, \F_{Z(w)} [2])
    \cong \{0\}
    .
  \end{equation*}  
  By \cref{lem:les} it suffices to show that the map
  \[\psi \colon \op{coker} \Gamma \left(U; \delta^0_w\right) \rightarrow
  \op{coker} \Gamma \left(V; \delta^0_w\right)\]
  is an epimorphism.
  Now in order for $\op{coker} \Gamma \left(V; \delta^0_w\right)$
  to be non-zero,
  there need to be at least two connected components in
  $V \cap \rho_1(w)$
  which lie in the same connected component of the superset $V \cap \rho_0(w)$.
  Moreover, for $V \cap \rho_1(w)$ to have at least
  two connected components we need to have 
  $w \preceq u$.
  If $u \not\preceq T(w)$,
  then $\psi$
  is an isomorphism by \cref{lem:auxIsoCoker}.
  If $u \preceq T(w)$,
  then the two components of $V \cap \rho_1(w)$
  lie in different components of $V \cap \rho_0(w)$
  as illustrated by \cref{fig:vanishingCoker}
  and thus $\op{coker} \Gamma \left(V; \delta^0_w\right) \cong \{0\}$,
  hence $\psi$ is surjective.
  
  Now suppose we have $v \in T(D)$, then let
  $w := T^{-1}(v)$ and let
  $U$ and $V$ be as in the previous two lemmas.
  We have to show that
  \begin{equation*}
    \Hom_{\Der (q)} (\phi_\gamma^{-1} \iota (u), \phi_\gamma^{-1} \iota (v))
    \cong
    \Hom_{\Der (q)} (\F_{Z(u)}, \F_{Z(w)} [1])
    \cong \{0\}
    .
  \end{equation*}  
  By \cref{lem:les} it suffices to show that the map
  \[\varphi \colon \Gamma \left(U; \F_{Z(w)}\right) \rightarrow
  \Gamma \left(V; \F_{Z(w)}\right)\]
  is an epimorphism and that
  \[\psi \colon \op{coker} \Gamma \left(U; \delta^0_w\right) \rightarrow
  \op{coker} \Gamma \left(V; \delta^0_w\right)\]
  is injective.
  If $u \not\preceq v$,
  then $Z(w)$ is not a closed subset of $V$
  and thus $\Gamma \left(V; \F_{Z(w)}\right) \cong \{0\}$,
  hence $\varphi$ is surjective.
  Moreover, $\psi$
  is an isomorphism by \cref{lem:auxIsoCoker}.
  Finally, we consider the case $u \preceq v$.
  By symmetry we may assume
  without loss of generality
  that the $y$-coordinate of $v$ is at least as large
  as the $y$-coordinate of $T(u)$,
  see also \cref{fig:vanishing}.
  In this case $U \cap \rho_1 (w)$ has at most
  one connected component
  and thus
  $\op{coker} \Gamma \left(U; \delta^0_w\right) \cong \{0\}$,
  hence $\psi$ is injective.
  Moreover,
  $q$ has to intersect at least one of the three colored triangles
  in \cref{fig:vanishing}.
  (It may happen that the blue and the red triangle overlap,
  in which case the green triangle is empty.)
  Here each triangle is closed at the right edge and open
  at the top edge in the sense that for each triangle,
  the vertical edge on the right is part of the triangle,
  but not the horizontal edge at the top.
  If $q$ intersects the blue triangle,
  then $Z(w)$ and $V$ are disjoint,
  hence $\Gamma \left(V; \F_{Z(w)}\right) \cong \{0\}$
  and thus $\varphi$ is surjective.
  If $q$ intersects the green triangle,
  then $V \cap Z(w)$ is not a non-empty closed subset of $V$,
  hence we have $\Gamma \left(V; \F_{Z(w)}\right) \cong \{0\}$
  in this case as well
  and thus $\varphi$ is surjective.
  If $q$ intersects the red triangle,
  then $U \cap Z(w)$ is a closed subset of $U$
  and $V \cap Z(w)$ has at most one connected component,
  hence $\varphi$ is surjective.
\end{proof}

Now we can deduce \cref{prp:iota} from this \enquote{vanishing theorem}.

\begin{figure}[t]
  \centering
  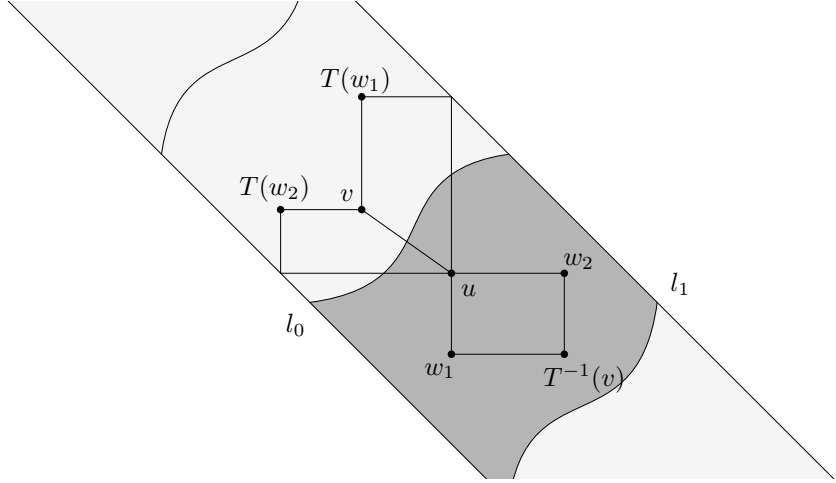
  \caption{The axis-aligned rectangle in $D$ determined by $u$ and $v$.}
  \label{fig:iso_iota}
\end{figure}

\begin{proof}[Proof of \cref{prp:iota}]
  Without loss of generality we assume $u \in D$.
  If \[v \notin (\uparrow u) \cap \op{int} (\downarrow T(u)),\]
  then the statement follows from the previous \cref{lem:vanishing}.
  Now suppose we have
  \[v \in (\uparrow u) \cap \op{int} (\downarrow T(u)).\]
  If $v \in D$, the statement follows from \cref{lem:iota0}.
  For $v \in T(D)$ let $w_1, w_2 \in D$ be as in \cref{fig:iso_iota}.
  Now by construction of $\iota$ we have
  $\iota(u \preceq v) = \partial'$,
  where $\partial'$ is the boundary map of the triangle
  \begin{equation*}
    \iota \left(T^{-1}(v)\right)
    \xrightarrow{\begin{pmatrix} 1 \\ 1 \end{pmatrix}}
    \iota (w_1) \oplus \iota (w_2)
    \xrightarrow{\begin{pmatrix} 1 & -1 \end{pmatrix}}
    \iota (u) \xrightarrow{\partial'}
    \iota (v)[1]
  \end{equation*}
  associated to the short exact sequence
  \begin{equation*}
    0 \rightarrow
    \iota_0 \left(T^{-1}(v)\right)
    \xrightarrow{\begin{pmatrix} 1 \\ 1 \end{pmatrix}}
    \iota_0 (w_1) \oplus \iota_0 (w_2)
    \xrightarrow{\begin{pmatrix} 1 & -1 \end{pmatrix}}
    \iota_0 (u) \rightarrow
    0
    .
  \end{equation*}
  Applying the cohomological functor
  $\Hom_{\Der (q_{\gamma})} (\iota(u), -)$ to this triangle
  we obtain the exact sequence
  \begin{equation*}
    \begin{tikzcd}[row sep=10ex]
      \Hom_{\Der (q_{\gamma})} (\iota(u), \iota(w_1))
      \oplus
      \Hom_{\Der (q_{\gamma})} (\iota(u), \iota(w_2))
      \arrow[d, "(1 \,~ -1)" description]
      \\
      \Hom_{\Der (q_{\gamma})} (\iota(u), \iota(u))
      \arrow[d,
      "{\Hom_{\Der (q_{\gamma})} (\iota(u), \iota(u \preceq v))}" description]
      \\
      \Hom_{\Der (q_{\gamma})} (\iota(u), \iota(v))
      \arrow[d, "{\begin{pmatrix} 1 \\ 1 \end{pmatrix}}"]
      \\
      \Hom_{\Der (q_{\gamma})} (\iota(u), (\iota \circ T)(w_1))
      \oplus
      \Hom_{\Der (q_{\gamma})} (\iota(u), (\iota \circ T)(w_2))
      .
    \end{tikzcd}
  \end{equation*}
  By the previous \cref{lem:vanishing} this exact sequence starts with $\{0\}$
  and ends with $\{0\}$,
  hence $\Hom_{\Der (q_{\gamma})} (\iota(u), \iota(u \preceq v))$
  is an isomorphism.
  Moreover,
  \[\Hom_{\Der (q_{\gamma})} (\iota(u), \iota(u)) =
  \langle \iota(u \preceq u) \rangle =
  \langle \op{id}_{\iota(u)} \rangle \cong \F\]
  by \cref{lem:iota0}
  and thus
  \begin{equation*}
    \Hom_{\Der (q_{\gamma})} (\iota(u), \iota(v)) =
    \langle \iota(u \preceq v) \rangle \cong \F
    .
    \qedhere
  \end{equation*}
\end{proof}


\section{Induced Cohomological Functors on $\M$}
\label{sec:cohoFunctors}

We consider a bounded below chain complex $F$ of sheaves on $q_{\gamma}$
as an object of the derived category $ \Der (q_{\gamma})$.
Then we have the functor
\begin{equation*}
  h_\gamma (F) :=
  \Hom_{\Der (q_{\gamma})} (\iota(-), F) \colon \M^{\circ} \rightarrow \VectF
\end{equation*}
from $\M$ to the category of vector spaces over $\F$.
As ${\iota(u) \cong 0}$ for all ${u \in \partial \M}$
the functor ${h_{\gamma} (F)}$ vanishes on ${\partial \M}$.

\begin{lem}
  \label{lem:boundedAboveSupport}
  The functor
  $h_{\gamma} (F) \colon \M^{\circ} \rightarrow \mathrm{Vect}_{\F}$
  has bounded above support.
\end{lem}

\begin{proof}
  Since $F$ is a bounded below chain complex of sheaves,
  there is an integer $n \in \Z$,
  such that $F^k \cong 0$ for all $k < -n$.
  Thus, we have
  $h_{\gamma} (F)(u) \cong \{0\}$
  for all $u \in \M \setminus T^n(Q)$.
\end{proof}

We will shortly see
that the functor ${h_{\gamma} (F) \colon \M^{\circ} \rightarrow \VectF}$
satisfies the exactness properties of the following definition.

\begin{dfn}
  \label{dfn:cohomological}
  We say that a functor
  $G \colon \M^{\circ} \rightarrow \mathrm{Vect}_{\F}$
  vanishing on $\partial \M$ is \emph{cohomological},
  if for any axis-aligned rectangle
  with one corner lying on $l_1$ and the other corners
  ${u \preceq v \preceq w \in \M}$,
  the long sequence
  \begin{equation*}
    \begin{tikzcd}
      &
      \cdots
      \arrow[r]
      &
      G(T(u))
      \arrow[dll, out=0, in=180, looseness=2, overlay]
      \\
      G(w)
      \arrow[r]
      &
      G(v)
      \arrow[r]
      &
      G(u)
      \arrow[dll, out=0, in=180, looseness=2, overlay]
      \\
      G(T^{-1}(w))
      \arrow[r]
      &
      \cdots
      .
    \end{tikzcd}  
  \end{equation*}
  is exact;
  see also
  \cref{fig:cohomological} or
  \cite[Definition C.1]{2021arXiv210809298B}.
\end{dfn}

\begin{figure}[t]
  \centering
  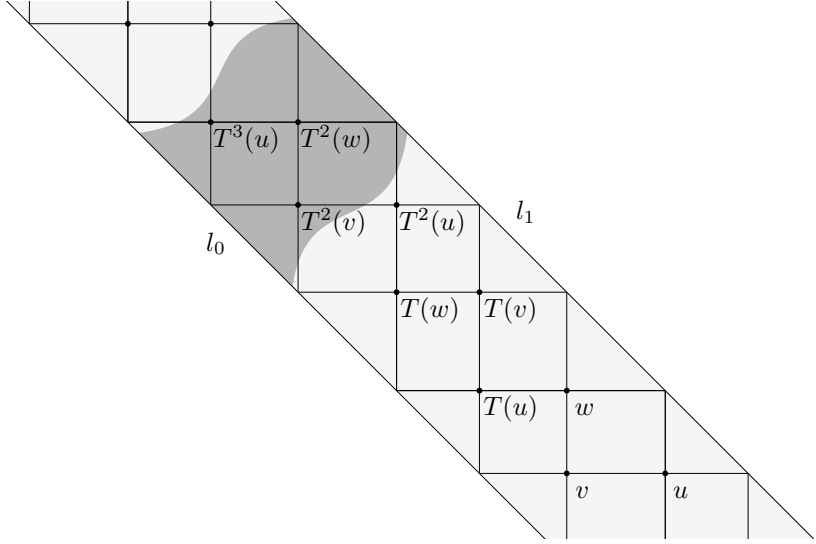
  \caption{
    The linear subposet given by the orbits of $u$, $v$, and $w$.
    The region shaded in dark grey is our fundamental domain $D$.
  }
  \label{fig:cohomological}
\end{figure}

In \cite[Proposition C.2]{2021arXiv210809298B} we also provide
the following useful characterization of cohomological functors.

\begin{prp}
  \label{prp:cohomological}
  A functor $G \colon \M^{\circ} \rightarrow \mathrm{Vect}_{\F}$
  vanishing on $\partial \M$
  is cohomological iff
  for any axis-aligned rectangle
  $u \preceq v_1, v_2 \preceq w \in D$
  as shown in \cref{fig:constrBoundary}
  the long sequence
  \begin{equation}
    \label{eq:charCohomological}
    \begin{tikzcd}
      &
      \cdots
      \arrow[r]
      &
      G(T(u))
      \arrow[dll, out=0, in=180, looseness=2, overlay]
      \\
      G(w)
      \arrow[r]
      &
      G(v_1) \oplus G(v_2)
      \arrow[r, "(1 ~\, -1)"]
      &
      G(u)
      \arrow[dll, out=0, in=180, looseness=2, overlay]
      \\
      G(T^{-1}(w))
      \arrow[r]
      &
      \cdots
    \end{tikzcd}  
  \end{equation}
  is exact.
\end{prp}

\begin{lem}
  \label{lem:cohomological}
  The functor
  ${h_{\gamma} (F) \colon \M^{\circ} \rightarrow \VectF}$
  is cohomological.  
\end{lem}

\begin{proof}
  If we apply the cohomological functor
  $\Hom_{\Der (q_{\gamma})} (-, F) \colon D^+ (q_{\gamma}) \rightarrow \VectF$
  to the distinguished triangle \eqref{eq:triangle}
  for $u \preceq v_1, v_2 \preceq w \in D$ as in \cref{fig:constrBoundary},
  then we obtain the long exact sequence
  \begin{equation*}
    \begin{tikzcd}
      &
      \cdots
      \arrow[r]
      &
      h_{\gamma} (F)(T(u))
      \arrow[dll, out=0, in=180, looseness=2, overlay]
      \\
      h_{\gamma} (F)(w)
      \arrow[r]
      &
      h_{\gamma} (F)(v_1) \oplus h_{\gamma} (F)(v_2)
      \arrow[r, "(1 ~\, -1)"]
      &
      h_{\gamma} (F)(u)
      \arrow[dll, out=0, in=180, looseness=2, overlay]
      \\
      h_{\gamma} (F)(T^{-1}(w))
      \arrow[r]
      &
      \cdots
      .
    \end{tikzcd}  
  \end{equation*}
  Thus, the long sequence \eqref{eq:charCohomological} is exact for
  ${G := h_{\gamma} (F) = \Hom_{\Der (q_{\gamma})} (\iota(-), F)}$
  and hence
  ${h_{\gamma} (F) \colon \M^{\circ} \rightarrow \VectF}$
  is cohomological.
\end{proof}

We also have the following primitive type of a cohomological functor,
see also \cref{fig:contraBlock}.

\begin{dfn}[Contravariant Block]
  \label{dfn:contraBlock}
  For $v \in \M$ we define
  \begin{equation*}
    B_v \colon \M^{\circ} \rightarrow \mathrm{Vect}_{\F},
    u \mapsto
    \begin{cases}
      \F & u \in (\downarrow v) \cap \op{int}\left(\uparrow T^{-1}(v)\right)
      \\
      \{0\} & \text{otherwise},
    \end{cases}
  \end{equation*}
  where $\op{int}\left(\uparrow T^{-1}(v)\right)$ is the interior of the upset
  of $T^{-1}(v)$ in $\M$.
  The internal maps are identities whenever both domain and codomain are $\F$,
  otherwise they are zero.
\end{dfn}

\begin{lem}[Yoneda]
  \label{lem:yoneda}
  Let ${G \colon \M^{\circ} \rightarrow \mathrm{Vect}_{\F}}$
  be a functor
  vanishing on $\partial \M$
  and let ${v \in \op{int} \M := \M \setminus \partial \M}$.
  Then the evaluation at ${1 \in \F = B_v(v)}$
  yields a linear isomorphism
  \begin{equation*}
    \op{Nat}(B_v, G) \cong G(v),
  \end{equation*}
  where $\op{Nat}(B_v, G)$ denotes the vector space
  of natural transformations from $B_v$ to $G$.
\end{lem}

\sloppy
Using the Yoneda \cref{lem:yoneda},
we may rephrase \cref{prp:iota} as follows.

\begin{cor}
  \label{cor:iota}
  The unique natural transformation
  \[
    B_v \rightarrow
    (h_\gamma \circ \iota)(v) = \Hom_{\Der (q_{\gamma})} (\iota(-), \iota(v))
  \]
  sending $1 \in \F = B_v (v)$
  to $\op{id}_{\iota(v)} \in \Hom_{\Der (q_{\gamma})} (\iota(v), \iota(v))$
  is a natural isomorphism.
\end{cor}

\sloppy
The following proposition 
provides a description of $h_{\gamma} (F)$
in terms of local sheaf cohomology.

\begin{prp}
  \label{prp:localCoho}
  Let $u \in D$ and let $n \in \Z$,
  then we have a natural isomorphism
  \begin{equation*}
    h_{\gamma}(F)(T^{-n}(u))
    =
    \Hom_{\Der (q_{\gamma})} ((\iota \circ T^{-n})(u), F)
    \cong
    H^{n}_{q \cap (\uparrow u)} (q \cap \op{int}(\downarrow T(u)); F)
    .
  \end{equation*}
\end{prp}

\begin{proof}
  By \cref{lem:iotaLocalSections} we have
  \begin{align*}
    \Hom_{\Der (q_{\gamma})} ((\iota \circ T^{-n})(u), F)
    & =
      \Hom_{\Der (q_{\gamma})} (\iota_0(u)[-n], F)
    \\
    & =
      \Hom_{\Der (q_{\gamma})} (\iota_0(u), F[n])
    \\
    & \cong
      H^{n}_{q \cap (\uparrow u)} (q \cap \op{int}(\downarrow T(u)); F)
      .
      \qedhere
  \end{align*}
\end{proof}

Finally we note that the assignment
$F \mapsto h_{\gamma} (F) = \Hom_{\Der (q_{\gamma})} (\iota(-), F)$
is functorial in $F$.
Therefore we obtain the functor
$h_{\gamma} \colon D^+(q_{\gamma}) \rightarrow \VectF^{\M^{\circ}}$.
Moreover, we take note of the following.

\begin{lem}
  \label{lem:compatCorefl}
  For each complex of sheaves $F$ of the derived category $D^+ (q_{\gamma})$
  the counit
  $\varepsilon_F \colon R \flat (F) \rightarrow F$
  of the adjunction
  \begin{equation*}
    \begin{tikzcd}
      D^+ (q_{\gamma}, \partial q)
      \arrow[r, ""{name=I}, bend right, hook]
      &
      D^+ (q_{\gamma})
      \arrow[l, "R \flat"'{name=b}, bend right]
      \arrow[phantom, from=I, to=b, "\dashv" rotate=90]
    \end{tikzcd}
  \end{equation*}
  is mapped to a natural isomorphism
  $h_{\gamma} (\varepsilon_F) \colon
  h_{\gamma} (R \flat (F)) \rightarrow h_{\gamma} (F)$.
\end{lem}

\begin{proof}
  Let $u \in \M$,
  then $\iota(u)$ is an object of $D^+ (q_{\gamma}, \partial q)$.
  Thus, the linear map
  \[\Hom_{\Der (q_{\gamma})} (\iota(u), \varepsilon_F) \colon
    \Hom_{\Der (q_{\gamma})} (\iota(u), R \flat (F)) \rightarrow
    \Hom_{\Der (q_{\gamma})} (\iota(u), F)
  \]
  is an isomorphism.
\end{proof}

Let
$h_{\gamma,0} \colon D^+(q_{\gamma}, \partial q) \rightarrow \VectF^{\M^{\circ}}$
be the restriction of $h_{\gamma}$ to $D^+(q_{\gamma}, \partial q)$.
With this
we may rephrase \cref{lem:compatCorefl} more concisely:

\begin{cor}
  \label{cor:vanishingBeckChev}
  The commutative square
  \begin{equation}
    \label{eq:vanishingDualBeckChev}
    \begin{tikzcd}[row sep=6ex]
      D^+(q_{\gamma}, \partial q)
      \arrow[r, hook]
      \arrow[d, "h_{\gamma,0}"']
      &
      D^+(q_{\gamma})
      \arrow[d, "h_{\gamma}"]
      \\
      \VectF^{\M^{\circ}}
      \arrow[r, equal]
      &
      \VectF^{\M^{\circ}}
    \end{tikzcd}
  \end{equation}
  of categories and functors satisfies
  the dual Beck--Chevalley condition
  as in \cref{dfn:BeckChevDual}.
\end{cor}

\begin{cor}
  \label{cor:compsWithDirIms}
  We have the diagram
  \begin{equation*}
    \begin{tikzcd}[row sep=11ex, column sep=7ex]
      &
      D^+(\dot{q})
      \arrow[dl, "R i_!"'{name=dirImProperSupp}]
      \arrow[d , "R i_*"]
      \\
      D^+(q_{\gamma}, \partial q)
      \arrow[d , "h_{\gamma,0}"']
      \arrow[dr, Rightarrow, "h_{\gamma} \circ \varepsilon",
      shorten >=1.5ex, shorten <=1.5ex]
      &
      D^+(q_{\gamma})
      \arrow[d, "h_{\gamma}"]
      \arrow[l, "R \flat"']
      \arrow[to=dirImProperSupp, Leftarrow, shorten >=1.5ex, shorten <=1.5ex]
      \\
      \VectF^{\M^{\circ}}
      \arrow[r, equal]
      &
      \VectF^{\M^{\circ}}      
    \end{tikzcd}
  \end{equation*}
  with all $2$-cells natural isomorphisms,
  where
  $D^+ (\dot{q})$ is the bounded-below derived category of the category
  of sheaves on $\dot{q}$
  with the subspace topology inherited from $q_{\gamma}$.
\end{cor}

\begin{proof}
  The triangular $2$-cell at the top is an isomorphism
  by \eqref{eq:dirImProperSuppAsFlat} and \cref{lem:compose}
  and the square $2$-cell at the bottom is an isomorphism
  by \cref{cor:vanishingBeckChev} or \cref{lem:compatCorefl}.
\end{proof}

In addition to the commutative square \eqref{eq:vanishingDualBeckChev}
we may also consider the commutative square
\begin{equation}
  \label{eq:interiorBeckChev}
  \begin{tikzcd}
    D^+ (\dot{q})
    \arrow[r, "R i_*"]
    \arrow[d, "h_{\gamma} \circ R i_*"']
    &
    D^+ (q_{\gamma})
    \arrow[d, "h_{\gamma}"]
    \\
    \VectF^{\M^{\circ}}
    \arrow[r, equal]
    &
    \VectF^{\M^{\circ}}
    .
  \end{tikzcd}
\end{equation}

\begin{lem}
  \label{lem:interiorBeckChev}
  The commutative square \eqref{eq:interiorBeckChev}
  satisfies the Beck--Chevalley condition as in \cref{dfn:BeckChev}.
\end{lem}

\begin{proof}
  By \cref{lem:dotsAdjEquiv} it suffices to show that the square
  \begin{equation}
    \label{eq:realInteriorBeckChev}
    \begin{tikzcd}
      D^+ (\ddot{q})
      \arrow[r, "R i_{2*}"]
      \arrow[d, "h_{\gamma} \circ R i_{2*}"']
      &
      D^+ (q_{\gamma})
      \arrow[d, "h_{\gamma}"]
      \\
      \VectF^{\M^{\circ}}
      \arrow[r, equal]
      &
      \VectF^{\M^{\circ}}
    \end{tikzcd}
  \end{equation}
  satisfies the Beck--Chevalley condition.
  Now \eqref{eq:realInteriorBeckChev} satisfies the Beck--Chevalley condition
  iff
  the square
  \begin{equation*}
    \begin{tikzcd}
      D^+ (\ddot{q})
      \arrow[r, "R i_{2*}"]
      \arrow[d, "H^n_{Z(u)} (\ddot{q}; -)"'{name=intCoho}]
      &
      D^+ (q_{\gamma})
      \arrow[d, "H^n_{Z(u)} (q_{\gamma}; -)"{name=coho}]
      \\
      \VectF
      \arrow[r, equal]
      &
      \VectF
      \arrow[phantom, from=intCoho, to=coho, "\cong"]
    \end{tikzcd}    
  \end{equation*}
  satisfies the Beck--Chevalley condition for any $u \in D$ and any $n \in \Z$.
  This in turn follows from \cref{cor:BeckChevLocalCoho}.
\end{proof}

Now if $\partial q$ is closed in $q_{\gamma}$,
then $D^+(q_{\gamma}, \partial q)$ and $D^+(\dot{q})$
are equivalent by \cref{lem:qAdjEquivProperSupp}.
So dealing with both of these categories separately is somewhat redundant
in this case.

\begin{lem}
  \label{lem:bothBeckChev}
  If $\partial q$ is closed in $q_{\gamma}$,
  then the square diagram
  \begin{equation}
    \label{eq:bothBeckChev}
    \begin{tikzcd}[row sep=8ex, column sep=7ex]
      D^+ (q_{\gamma})
      \arrow[r, "i^{-1}"]
      \arrow[d, "h_{\gamma}"']
      &
      D^+ (\dot{q})
      \arrow[d, "h_{\gamma} \circ R i_*"]
      \\
      \VectF^{\M^{\circ}}
      \arrow[r, equal]
      \arrow[ru, "h_{\gamma} \circ \eta^i",
      Rightarrow, shorten >=1.5ex, shorten <=1.5ex]
      &
      \VectF^{\M^{\circ}}
    \end{tikzcd}
  \end{equation}
  satisfies both (the ordinary and the dual) Beck--Chevalley conditions,
  where ${\eta^i \colon \op{id} \rightarrow Ri_* \circ i^{-1}}$
  is the unit of the adjunction ${i^{-1} \dashv R i_*}$.
\end{lem}

\begin{proof}
  By \cref{lem:interiorBeckChev} the natural transformation
  ${h_{\gamma} \circ \eta^i \colon
    h_{\gamma} \rightarrow h_{\gamma} \circ Ri_* \circ i^{-1}}$,
  which is the mate of the identity natural transformation
  in the square \eqref{eq:interiorBeckChev},
  is a natural isomorphism,
  hence \eqref{eq:bothBeckChev} satisfies the dual Beck--Chevalley condition.
  Moreover,
  by \cref{lem:adjDirImProperSupp}
  the functor
  ${i^{-1} \colon D^+ (q_{\gamma}) \rightarrow D^+ (\dot{q})}$
  also has the left adjoint
  ${R i_! \colon D^+ (\dot{q}) \rightarrow D^+ (q_{\gamma})}$.
  Now ${h_{\gamma} \circ R i_! = h_{\gamma} \circ \flat \circ R i_*}$
  and moreover,
  the natural transformation
  \begin{equation*}
    h_{\gamma} \circ \varepsilon \circ R i_* \colon
    h_{\gamma} \circ \flat \circ R i_* \Rightarrow
    h_{\gamma} \circ R i_*
  \end{equation*}
  is a natural isomorphism
  by \cref{lem:compatCorefl}.
  Thus, the square diagram \eqref{eq:bothBeckChev}
  satisfies the Beck--Chevalley condition
  iff
  the composition of square diagrams
  \begin{equation}
    \label{eq:compSquareDiagrams}
    \begin{tikzcd}[row sep=8ex, column sep=7ex]
      D^+ (q_{\gamma})
      \arrow[r, "i^{-1}"]
      \arrow[d, "h_{\gamma}"']
      &
      D^+ (\dot{q})
      \arrow[d, "h_{\gamma} \circ R i_*" description]
      \arrow[d, ""{name=dirIm}, phantom]
      \arrow[r, equal]
      &
      D^+ (\dot{q})
      \arrow[d, "h_{\gamma} \circ R i_!"{name=dirImProperSupp}]
      \\
      \VectF^{\M^{\circ}}
      \arrow[r, equal]
      \arrow[ru, "h_{\gamma} \circ \eta^i",
      Rightarrow, shorten >=1.5ex, shorten <=1.5ex]
      &
      \VectF^{\M^{\circ}}
      \arrow[r, equal]
      &
      \VectF^{\M^{\circ}}
      \arrow[from=dirImProperSupp, to=dirIm, Rightarrow,
      shorten >=3.5ex, shorten <=1.5ex,
      "\quad h_{\gamma} \circ \varepsilon \circ R i_*"']
    \end{tikzcd}
  \end{equation}
  satisfies the Beck--Chevalley condition.
  Now for a sheaf $F$ on $q_{\gamma}$
  we have the commutative diagram
  \begin{equation*}
    \begin{tikzcd}[row sep=5ex, column sep=8ex]
      \flat F
      \arrow[r, "(\flat \circ \eta^i)_F"]
      \arrow[d, "\varepsilon_F"']
      &
      i_! i^{-1} F
      \arrow[d, "(\varepsilon \circ i_* \circ i^{-1})_F"]
      \arrow[dl, "\varepsilon^!_F"']
      \\
      F
      \arrow[r, "\eta^i_F"']
      &
      i_* i^{-1} F
      ,
    \end{tikzcd}
  \end{equation*}
  hence the composition of square diagrams \eqref{eq:compSquareDiagrams}
  is the square diagram
  \begin{equation}
    \label{eq:properSuppBeckChev}
    \begin{tikzcd}[row sep=8ex, column sep=7ex]
      D^+ (q_{\gamma})
      \arrow[r, "i^{-1}"]
      \arrow[d, "h_{\gamma}"']
      &
      D^+ (\dot{q})
      \arrow[d, "h_{\gamma} \circ R i_!"]
      \arrow[dl, "h_{\gamma} \circ \varepsilon^!"',
      Rightarrow, shorten >=1.5ex, shorten <=1.5ex]
      \\
      \VectF^{\M^{\circ}}
      \arrow[r, equal]
      &
      \VectF^{\M^{\circ}}
      .
    \end{tikzcd}
  \end{equation}
  Moreover,
  the natural transformation
  ${h_{\gamma} \circ \varepsilon^! \colon
    h_{\gamma} \circ R i_! \circ i^{-1} \rightarrow h_{\gamma}}$
  is just the mate of the identity natural transformation
  of the commutative square
  \begin{equation*}
    \begin{tikzcd}[row sep=6ex, column sep=7ex]
      D^+ (q_{\gamma})
      \arrow[d, "h_{\gamma}"']
      &
      D^+ (\dot{q})
      \arrow[d, "h_{\gamma} \circ R i_!"]
      \arrow[l, "R i_!"']
      \\
      \VectF^{\M^{\circ}}
      \arrow[r, equal]
      &
      \VectF^{\M^{\circ}}
    \end{tikzcd}    
  \end{equation*}
  and thus
  \eqref{eq:properSuppBeckChev} satisfies the Beck--Chevalley condition.
\end{proof}

\begin{prp}
  \label{prp:conservative}
  The functor
  $h_{\gamma,0} \colon D^+(q_{\gamma}, \partial q) \rightarrow \VectF^{\M^{\circ}}$
  is conservative.
  In other words,
  if $\psi \colon G \rightarrow F$ is a homomorphism of
  $D^+(q_{\gamma}, \partial q)$
  with
  $h_{\gamma} (\psi) \colon h_{\gamma} (G) \rightarrow h_{\gamma} (F)$
  a natural isomorphism,
  then $\psi \colon G \rightarrow F$ is an isomorphism as well.
\end{prp}

\begin{proof}
  Suppose $\psi \colon G \rightarrow F$ is a homomorphism of
  $D^+(q_{\gamma}, \partial q)$
  with
  $h_{\gamma} (\psi) \colon h_{\gamma} (G) \rightarrow h_{\gamma} (F)$
  a natural isomorphism and
  let $n \in \Z$ be arbitrary.
  It suffices to show that the induced linear map
  \[H^n (U; \psi) \colon H^n (U; G) \rightarrow H^n (U; F)\]
  is an isomorphism for each open subset $U$
  of some basis $\mathcal{B}$ of $q_\gamma$.
  To this end,
  let
  \[\mathcal{B} := \{q \cap \op{int} (\downarrow (T(u)) \mid u \in D\}.\]
  If $\partial q$ is discrete as a subposet of $\M$,
  then $\mathcal{B}$ is a basis of $q_\gamma$.
  Otherwise, $\mathcal{B} \cup \{q\}$ is a basis of $q_\gamma$
  and both
  $H^n (F; q)$ and
  $H^n \left(G; q\right)$
  vanish.
  So in either case, it suffices to check that $H^n (U; \psi)$
  is an isomorphism for each $U \in \mathcal{B}$.

  \begin{figure}[t]
    \centering
\begingroup%
  \makeatletter%
  \providecommand\color[2][]{%
    \errmessage{(Inkscape) Color is used for the text in Inkscape, but the package 'color.sty' is not loaded}%
    \renewcommand\color[2][]{}%
  }%
  \providecommand\transparent[1]{%
    \errmessage{(Inkscape) Transparency is used (non-zero) for the text in Inkscape, but the package 'transparent.sty' is not loaded}%
    \renewcommand\transparent[1]{}%
  }%
  \providecommand\rotatebox[2]{#2}%
  \newcommand*\fsize{\dimexpr\f@size pt\relax}%
  \newcommand*\lineheight[1]{\fontsize{\fsize}{#1\fsize}\selectfont}%
  \ifx\svgwidth\undefined%
    \setlength{\unitlength}{275.49998474bp}%
    \ifx\svgscale\undefined%
      \relax%
    \else%
      \setlength{\unitlength}{\unitlength * \real{\svgscale}}%
    \fi%
  \else%
    \setlength{\unitlength}{\svgwidth}%
  \fi%
  \global\let\svgwidth\undefined%
  \global\let\svgscale\undefined%
  \makeatother%
  \begin{picture}(1,0.51724141)%
    \lineheight{1}%
    \setlength\tabcolsep{0pt}%
    \put(0,0){\includegraphics[width=\unitlength,page=1]{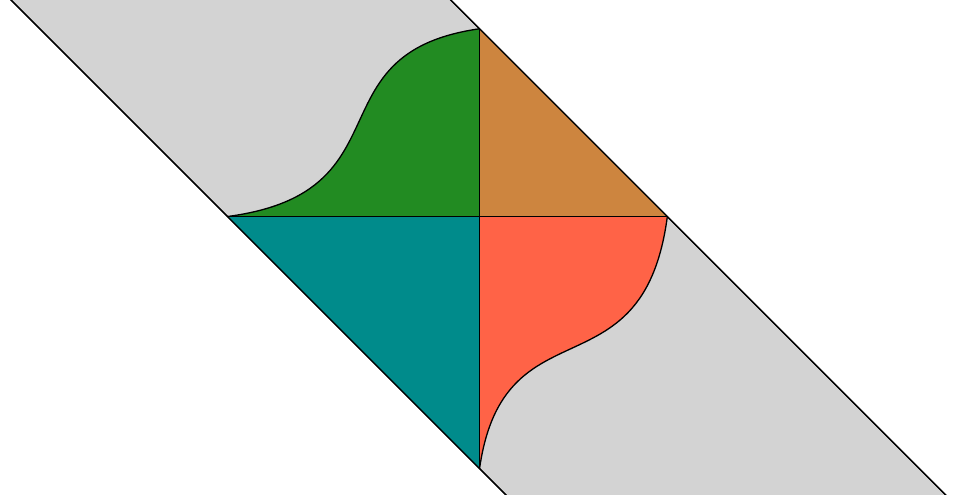}}%
    \put(0.78882473,0.23646263){\makebox(0,0)[t]{\lineheight{1.25}\smash{\begin{tabular}[t]{c}$l_1$\end{tabular}}}}%
    \put(0.28998568,0.19277288){\makebox(0,0)[t]{\lineheight{1.25}\smash{\begin{tabular}[t]{c}$l_0$\end{tabular}}}}%
  \end{picture}%
\endgroup%

    \caption{
      The fundamental domain $D$ partitioned into four regions.
    }
    \label{fig:fourRegions}
  \end{figure}

  Now let $u \in D$ be contained in the region of $D$
  shaded in red in \cref{fig:fourRegions}
  and let ${U := q \cap \op{int} (\downarrow (T(u))}$.
  Then we have
  \[H^n (U; \psi) \cong H^n_q (U; \psi) \cong
    \Hom_{\Der (q_{\gamma})} ((\iota \circ T^{-n})(u), \psi) \cong
    (h_\gamma (\psi) \circ T^{-n})_u 
  \]
  as linear maps
  by \cref{prp:localCoho} and thus $H^n (U; \psi)$
  is an isomorphism.

  Now suppose that $u = (x, y) \in D$ is a point on the
  vertical line segment separating the red from the cyan region
  in \cref{fig:fourRegions},
  presuming the cyan region is non-empty.
  (Otherwise there is no need to treat this case.)
  Then for all $s < x$ the open subsets
  $q \cap \op{int} (\downarrow (T(s, y))$
  are identical.
  So let $U := q \cap \op{int} (\downarrow (T(s, y))$
  for some (and thus any) $s < x$.
  Now for $s < x$ we have the exact sequence
  \begin{equation*}
    H^{n-1} (q \cap ((s, \infty) \times \R); \psi) \rightarrow
    H^n_{q \cap (\uparrow (s, y))} (U; \psi) \rightarrow
    H^n (U; \psi) \rightarrow
    H^{n  } (q \cap ((s, \infty) \times \R); \psi)
  \end{equation*}
  in the category of linear maps by \cite[Equation (2.6.32)]{Kashiwara1990}.
  As both $F$ and $G$
  vanish on $\partial q$,
  taking the direct limit of this sequence over all $s < x$
  we obtain the exact sequence
  \begin{equation*}
    0 \rightarrow
    \varinjlim_{s < x} H^n_{q \cap (\uparrow (s, y))} (U; \psi) \rightarrow
    H^n (U; \psi) \rightarrow
    0
    .
  \end{equation*}
  Moreover, \cref{prp:localCoho} implies
  \begin{equation*}
    H^n_{q \cap (\uparrow (s, y))} (U; \psi) \cong
    \Hom_{\Der (q_{\gamma})} ((\iota \circ T^{-n})(s, y), \psi) \cong
    (h_\gamma (\psi) \circ T^{-n})_{(s, y)}
    ,
  \end{equation*}
  hence $H^n_{q \cap (\uparrow (s, y))} (U; \psi)$
  is an isomorphism for any $s < x$.
  Thus,
  \[\varinjlim_{s < x} H^n_{q \cap (\uparrow (s, y))} (U; \psi) \cong
    H^n (U; \psi)\]
  is an isomorphism as well.

  Completely analogously one can show that
  $H^n (q \cap \op{int} (\downarrow T(u)); \psi)$
  is an isomorphism for any $u \in D$ contained in the region of $D$
  shaded in brown in $\cref{fig:fourRegions}$.

  Now suppose that $u \in D$ is the point
  in the center of the crosshair shown in \cref{fig:fourRegions},
  presuming the green region is non-empty.
  Then for all $w \in \op{int}(\uparrow u)$ we have
  $q \subset \op{int} (\downarrow T(w))$
  as well as the exact sequence
  \begin{equation*}
    H^{n-1} (q \setminus (\uparrow w); \psi) \rightarrow
    H^n_{q \cap (\uparrow w)} (q; \psi) \rightarrow
    H^n (q; \psi) \rightarrow
    H^{n  } (q \setminus (\uparrow w); \psi)
  \end{equation*}
  in the category of linear maps by \cite[Equation (2.6.32)]{Kashiwara1990}.
  As both $F$ and $G$
  vanish on $\partial q$,
  taking the direct limit of this sequence over all $w \in \op{int}(\uparrow u)$
  we obtain the exact sequence
  \begin{equation*}
    0 \rightarrow
    \varinjlim_{\substack{w \in \op{int}(\uparrow u)}}
    H^n_{q \cap (\uparrow w)} (q; \psi)
    \rightarrow
    H^n (q; \psi) \rightarrow
    0
    .
  \end{equation*}
  Moreover, \cref{prp:localCoho} implies
  \begin{equation*}
    H^n_{q \cap (\uparrow w)} (q; \psi) \cong
    \Hom_{\Der (q_{\gamma})} ((\iota \circ T^{-n})(w), \psi) \cong
    (h_\gamma (\psi) \circ T^{-n})_{w}
    ,
  \end{equation*}
  hence $H^n_{q \cap (\uparrow w)} (q; \psi)$
  is an isomorphism
  for all $w \in \op{int}(\uparrow u)$.
  Thus,
  \[\varinjlim_{\substack{w \in \op{int}(\uparrow u)}} H^n_{q \cap (\uparrow w)} (q; \psi)
    \cong
    H^n (q; \psi)\]
  is an isomorphism as well.
\end{proof}

\subsection{Taming Induced Cohomological Functors}
\label{sec:tamingCohoFunctors}

As before let
$\dot{q} = q_{\gamma} \setminus \partial q$
and let ${i \colon \dot{q} \hookrightarrow q_{\gamma}}$
be the corresponding subspace inclusion.

\begin{dfn}
  \label{dfn:tame}
  We say that a bounded below complex $F$
  of sheaves on $q_{\gamma}$ is \emph{tame}
  if ${h_{\gamma} (F) \colon \M^{\circ} \rightarrow \VectF}$
  is pointwise finite-dimensional (pfd).
  Similarly, we say that an object $F$ of $D^+(\dot{q})$
  is \emph{tame}
  if $R i_* F$ is tame.
  We denote the full subcategories of tame complexes
  in the derived categories $D^+ (q_{\gamma})$ and $D^+(\dot{q})$
  by $D^+_t (q_{\gamma})$ and $D^+_t(\dot{q})$ respectively.
  Furthermore,
  we denote the intersection of the two full subcategories
  $D^+_t (q_{\gamma})$ and ${D^+ (q_{\gamma}, \partial q)}$
  by ${D^+_t (q_{\gamma}, \partial q)}$.
\end{dfn}

\begin{lem}
  \label{lem:pfd}
  Let $F$ be an object of $D^+ (q_{\gamma})$.
  If
  ${\dim_{\F} H^n(I; F) < \infty}$
  for any connected open subset ${I \subseteq q_{\gamma}}$
  and any integer $n \in \Z$,
  then $F$ is tame.
\end{lem}

\begin{proof}
  Let $u \in D$ and let $n \in \Z$.
  We show $\dim_{\F} \left(h_{\gamma} (F) \circ T^{-n}\right)(u) < \infty$.
  To this end,
  let $I := q \cap \op{int}(\downarrow T(u))$
  and $J := q \cap (\uparrow u)$.
  By \cref{prp:localCoho} it suffices to show that
  $\dim_{\F} H^n_{J} (I; F) < \infty$.
  We consider the fragment
  \begin{equation*}
    H^{n-1}(I; F) \rightarrow
    H^{n-1}(I \setminus J; F) \rightarrow
    H^n_{J} (I; F) \rightarrow
    H^{n}(I; F) \rightarrow
    H^{n}(I \setminus J; F)
  \end{equation*}
  of the corresponding long exact sequence for local sheaf cohomology,
  see for example
  \mbox{\cite[Equation (2.6.32) and Remark 2.6.10]{Kashiwara1990}}.
  Now $I$ is connected and $I \setminus J$
  is the disjoint union of at most two connected open sets.
  By our assumptions on $F$
  all four cohomology spaces surrounding $H^n_{J} (I; F)$
  in above exact sequence are finite-dimensional.
  As a result, $H^n_{J} (I; F)$ is finite-dimensional as well.
\end{proof}

Now let $F$ be an object of $D^+ (\dot{q})$.

\begin{lem}
  \label{lem:interiorIntervalCoho}
  Let ${I \subseteq q_{\gamma}}$ be a connected open subset.
  Then we have
  ${H^n(I; Ri_* F) \cong H^n(I \setminus \partial q; F)}$
  for all integers ${n \in \Z}$.
\end{lem}

\begin{proof}
  This follows from \cref{lem:compose}
  with
  ${\Gamma(I; -) \colon \mathrm{Sh}(q_{\gamma}) \rightarrow \mathrm{Vect}_{\F}}$
  in place of ${H \colon \mathcal{B} \rightarrow \mathcal{A}}$
  and
  ${i_* \colon \mathrm{Sh}(\dot{q}) \rightarrow \mathrm{Sh}(q_{\gamma})}$
  in place of
  ${G \colon \mathcal{C} \rightarrow \mathcal{B}}$.
\end{proof}

\begin{cor}
  \label{cor:interiorIntervalCoho}
  For any $n \in \Z$ we have
  ${\dim_{\F} H^n(I; F) < \infty}$
  for all connected open subsets ${I \subseteq \dot{q}}$
  iff we have
  ${\dim_{\F} H^n(I; R i_* F) < \infty}$
  for all connected open subsets ${I \subseteq q_{\gamma}}$.
\end{cor}

\begin{lem}
  \label{lem:interiorCharTameness}
  If
  ${\dim_{\F} H^n (I; F) < \infty}$
  for all connected open subsets ${I \subseteq \dot{q}}$
  and all integers ${n \in \Z}$,
  then $F$ is tame.
  If $\partial q$ is closed in $q_{\gamma}$,
  then the converse is true as well.
\end{lem}

\begin{proof}
  The first implication follows from
  \cref{cor:interiorIntervalCoho} and \cref{lem:pfd}.
  Now suppose
  $\partial q$ is closed in $q_{\gamma}$,
  let $I \subseteq \dot{q}$ be a connected open subset,
  let $n \in \Z$,
  and suppose that
  ${h_{\gamma} (R i_* F) \colon \M^{\circ} \rightarrow \VectF}$
  is pfd.
  As $\partial q$ is closed in $q_{\gamma}$ there is some point
  ${u \in D}$ with
  ${\rho_1 (u) = q \setminus I}$
  and ${\rho_0 (u) = q}$.
  By \cref{prp:localCoho}, \eqref{eq:rhoExplicit},
  and \cref{lem:interiorIntervalCoho}
  this implies
  \begin{equation*}
    \left(h_{\gamma}(R i_* F) \circ T^{-n}\right)(u)
    \cong
    H^n (I; R i_* F)
    \cong
    H^n (I; F)
    ,
  \end{equation*}
  hence $H^n (I; F)$ is finite-dimensional.
\end{proof}

\begin{dfn}
  \label{dfn:cont}
  We say that a 
  functor
  $G \colon \M^{\circ} \rightarrow \mathrm{Vect}_{\F}$
  is \emph{sequentially continuous},
  if for any increasing sequence
  $(u_k)_{k=1}^{\infty}$ in $\M$
  converging to $u$
  the natural map
  \begin{equation*}
    G(u) \rightarrow \varprojlim_{k} G(u_k)
  \end{equation*}
  is an isomorphism,
  see also \cite[Definition 2.4]{2021arXiv210809298B}.
\end{dfn}

\begin{prp}
  \label{prp:cont}
  Let $F$ be an object of $D^+(q_{\gamma})$ and suppose that
  $h_{\gamma} (F) \colon \M^{\circ} \rightarrow \mathrm{Vect}_{\F}$
  is pfd.
  Then $h_{\gamma} (F)$ is sequentially continuous.
\end{prp}

\begin{proof}
  Let $(u_k)_{k=1}^{\infty}$ be an increasing sequence in $\M$
  converging to $u \in \M$.
  Without loss of generality we assume that $(u_k)_{k=1}^{\infty}$
  is contained in a single tile $T^{-n} (D)$ for some $n \in \Z$.
  As $h_{\gamma} (F)$ is pfd 
  the projective system
  \begin{equation*}
    \left\{
      H^{n-1}_{q \cap (\uparrow u_k)} (q \cap \op{int}(\downarrow T(u_k)); F)
    \right\}_{k=1}^{\infty}
  \end{equation*}
  satisfies the
  \href{
    https://en.wikipedia.org/wiki/Inverse_limit#Mittag-Leffler_condition
  }{Mittag-Leffler condition}
  by \cref{prp:localCoho}.
  Thus, the natural map
  \begin{equation*}
    H^{n}_{q \cap (\uparrow u)} (q \cap \op{int}(\downarrow T(u)); F)
    \rightarrow
    \varprojlim_k
    H^{n}_{q \cap (\uparrow u_k)} (q \cap \op{int}(\downarrow T(u_k)); F)
  \end{equation*}
  is an isomorphism by \cite[Proposition 2.7.1.(ii)]{Kashiwara1990}.
  In conjunction with \cref{prp:localCoho} this implies
  the sequential continuity of $h_{\gamma} (F)$.  
\end{proof}

As already noted in \cref{sec:intro}
we denote the category
of finite-dimensional vector spaces over $\F$
by $\mathrm{vect}_{\F}$
and
by $\mathcal{J}$ the full subcategory of pfd functors
${\M^{\circ} \rightarrow \vectF}$
that are cohomological, sequentially continuous,
and have bounded above support.
We summarize the results of this subsection so far:

\begin{prp}
  \label{prp:tameJ}
  Let $F$ be an object of $D^+ (q_{\gamma})$
  or of $D^+ (\dot{q})$.
  Then $F$ is tame
  iff
  $h_{\gamma} (F) \colon \M^{\circ} \rightarrow \mathrm{Vect}_{\F}$
  respectively
  $h_{\gamma} (R i_* F) \colon \M^{\circ} \rightarrow \mathrm{Vect}_{\F}$
  is in $\mathcal{J}$.
\end{prp}

\begin{proof}
  This follows directly from Lemmas
  \ref{lem:boundedAboveSupport},
  \ref{lem:cohomological},
  and \cref{prp:cont}.  
\end{proof}

Now by \cref{cor:vanishingBeckChev} the coreflector
$R \flat \colon D^+ (q_{\gamma}) \rightarrow D^+ (q_{\gamma}, \partial q)$
restricts to a coreflector for the full subcategory inclusion
$D^+_t (q_{\gamma}, \partial q) \hookrightarrow D^+_t (q_{\gamma})$.
With this we obtain the following.

\begin{lem}
  \label{lem:replSubcatBeckChev}
  The commutative square
  \begin{equation*}
    \begin{tikzcd}[row sep=6ex]
      D^+_t(q_{\gamma}, \partial q)
      \arrow[r, hook]
      \arrow[d, hook]
      &
      D^+_t(q_{\gamma})
      \arrow[d, hook]
      \\
      D^+(q_{\gamma}, \partial q)
      \arrow[r, hook]
      &
      D^+(q_{\gamma})
    \end{tikzcd}
  \end{equation*}
  of full replete subcategory inclusions
  satisfies the dual Beck--Chevalley condition.
\end{lem}

Now let
\begin{equation*}
  h_{\gamma,t} \colon D^+_t(q_{\gamma}) \rightarrow \mathcal{J}
  \quad \text{and} \quad
  h_{\gamma,0,t} \colon D^+_t(q_{\gamma}, \partial q) \rightarrow \mathcal{J}
\end{equation*}
be the corresponding restrictions of
$h_{\gamma} \colon D^+(q_{\gamma}) \rightarrow \VectF^{\M^{\circ}}$.

\begin{lem}
  The commutative square
  \begin{equation*}
    \begin{tikzcd}[row sep=6ex]
      D^+_t(q_{\gamma}, \partial q)
      \arrow[r, hook]
      \arrow[d, "h_{\gamma,0,t}"']
      &
      D^+_t(q_{\gamma})
      \arrow[d, "h_{\gamma,t}"]
      \\
      \mathcal{J}
      \arrow[r, equal]
      &
      \mathcal{J}
    \end{tikzcd}
  \end{equation*}
  satisfies the dual Beck--Chevalley condition.
\end{lem}

\begin{proof}
  This follows directly from \cref{cor:vanishingBeckChev}.
\end{proof}

\begin{lem}
  \label{lem:tameInteriorAdj}
  The derived adjunction
  \begin{equation*}
    \begin{tikzcd}
      D^+(q_{\gamma})
      \arrow[r, "D^+(i^{-1})"'{name=la}, bend right]
      &
      D^+(\dot{q})
      \arrow[l, "R i_*"'{name=ra}, bend right]
      \arrow[phantom, from=la, to=ra, "\dashv" rotate=90]
    \end{tikzcd}
  \end{equation*}
  restricts to an adjunction
  \begin{equation*}
    \begin{tikzcd}
      D^+_t(q_{\gamma})
      \arrow[r, "D^+(i^{-1})"'{name=la}, bend right]
      &
      D^+_t(\dot{q})
      .
      \arrow[l, "R i_*"'{name=ra}, bend right]
      \arrow[phantom, from=la, to=ra, "\dashv" rotate=90]
    \end{tikzcd}
  \end{equation*}
\end{lem}

\begin{proof}
  By \cref{dfn:tame}
  the derived functor
  $R i_* \colon D^+(\dot{q}) \rightarrow D^+(q_{\gamma})$
  restricts to a functor
  $D^+_t(\dot{q}) \rightarrow D^+_t(q_{\gamma})$.
  Now suppose $F$ is an object of $D^+_t(q_{\gamma})$.
  We have to show that $D^+(i^{-1})(F)$ is an object of $D^+_t(\dot{q})$.
  With some abuse of notation we write $i^{-1} F$ for $D^+(i^{-1})(F)$.
  By \cref{dfn:tame} and \cref{cor:vanishingBeckChev} it suffices to show that
  $R \flat R i_* i^{-1} F$ is an object of $D^+_t(q_{\gamma})$.
  Now by \cref{lem:compose}, \eqref{eq:dirImProperSuppAsFlat},
  and \cref{lem:flatUnitIso} 
  we have
  \begin{equation*}
    R \flat R i_* i^{-1} F \cong
    R(\flat \circ i_* \circ i^{-1})(F) =
    R(i_! \circ i^{-1})(F) \cong
    R \flat F
    ,
  \end{equation*}
  which is in $D^+_t(q_{\gamma})$ by
  \cref{cor:vanishingBeckChev} or \cref{lem:replSubcatBeckChev}.
\end{proof}

\begin{lem}
  The commutative square
  \begin{equation*}
    \begin{tikzcd}
      D^+_t (\dot{q})
      \arrow[r, "R i_*"]
      \arrow[d, "h_{\gamma,t} \circ R i_*"']
      &
      D^+_t (q_{\gamma})
      \arrow[d, "h_{\gamma,t}"]
      \\
      \mathcal{J}
      \arrow[r, equal]
      &
      \mathcal{J}
    \end{tikzcd}
  \end{equation*}
  satisfies the Beck--Chevalley condition.
\end{lem}

\begin{proof}
  This follows directly
  from Lemmas \ref{lem:tameInteriorAdj} and \ref{lem:interiorBeckChev}.
\end{proof}

By Lemmas \ref{lem:replSubcatBeckChev}, \ref{lem:tameInteriorAdj},
and \ref{lem:compose}
we may compose the adjunctions
\begin{equation*}
  \begin{tikzcd}
    D^+_t (q_{\gamma}, \partial q)
    \arrow[r, ""{name=I}, bend right, hook]
    &
    D^+_t (q_{\gamma})
    \arrow[l, "R \flat"'{name=b}, bend right]
    \arrow[phantom, from=I, to=b, "\dashv" rotate=90]
    \arrow[r, "D^+(i^{-1})"'{name=la}, bend right]
    &
    D^+_t(\dot{q})
    \arrow[l, "R i_*"'{name=ra}, bend right]
    \arrow[phantom, from=la, to=ra, "\dashv" rotate=90]
  \end{tikzcd}
\end{equation*}
to obtain the adjunction
\begin{equation}
  \label{eq:vanishingTameInteriorAdj}
  \begin{tikzcd}
    D^+_t (q_{\gamma}, \partial q)
    \arrow[r, "D^+(i^{-1})"'{name=la}, bend right]
    &
    D^+_t(\dot{q})
    .
    \arrow[l, "R i_!"'{name=ra}, bend right]
    \arrow[phantom, from=la, to=ra, "\dashv" rotate=90]
  \end{tikzcd}
\end{equation}

\begin{lem}
  \label{lem:vanishingTameInteriorAdj}
  If $\partial q$ is closed in $q_{\gamma}$,
  then the adjunction \eqref{eq:vanishingTameInteriorAdj}
  is an adjoint equivalence.
\end{lem}

\begin{proof}
  This follows directly from \cref{lem:qAdjEquivProperSupp}.
\end{proof}

\begin{lem}
  If $\partial q$ is closed in $q_{\gamma}$,
  then the square diagram
  \begin{equation}
    \label{eq:bothBeckChev}
    \begin{tikzcd}[row sep=8ex, column sep=7ex]
      D^+_t (q_{\gamma})
      \arrow[r, "i^{-1}"]
      \arrow[d, "h_{\gamma,t}"']
      &
      D^+_t (\dot{q})
      \arrow[d, "h_{\gamma,t} \circ R i_*"]
      \\
      \mathcal{J}
      \arrow[r, equal]
      \arrow[ru, "h_{\gamma} \circ \eta^i",
      Rightarrow, shorten >=1.5ex, shorten <=1.5ex]
      &
      \mathcal{J}
    \end{tikzcd}
  \end{equation}
  satisfies both (the ordinary and the dual) Beck--Chevalley conditions.
\end{lem}

\begin{proof}
  This follows directly from \cref{lem:bothBeckChev}.
\end{proof}

We also note that $\mathcal{J}$
is an additive $\F$-linear subcategory of $\VectF^{\M^{\circ}}$.
We end this section with a proof that
$D^+_t (q_{\gamma}, \partial q)$
is a triangulated subcategory of
$D^+ (q_{\gamma})$.
To this end,
we show the following auxiliary lemma.

\begin{lem}
  \label{lem:weakSerreJ}
  The category of pfd sequentially continuous functors
  $\M^{\circ} \rightarrow \VectF$
  is a weak Serre subcategory of $\VectF^{\M^{\circ}}$,
  i.e.
  for any exact sequence
  \begin{equation*}
    F_1 \rightarrow F_2 \rightarrow F_3 \rightarrow F_4 \rightarrow F_5
  \end{equation*}
  with
  ${F_i \colon \M^{\circ} \rightarrow \VectF}$
  pfd sequentially continuous for ${i=1,2,4,5}$
  the functor
  ${F_3 \colon \M^{\circ} \rightarrow \VectF}$
  is pfd sequentially continuous as well.
\end{lem}

\begin{proof}
  The statement that
  $F_3 \colon \M^{\circ} \rightarrow \VectF$
  is pfd follows point-wise from the analogous property
  of the full subcategory inclusion $\vectF \hookrightarrow \VectF$
  of finite-dimensional vector spaces in $\VectF$.
  Now let
  $(u_k)_{k=1}^{\infty}$
  be an increasing sequence in $\M$
  converging to $u$.
  As inverse limits of finite-dimensional vector spaces are exact,
  both rows of the commutative diagram
  \begin{equation}
    \label{eq:seqContLatter}
    \begin{tikzcd}[column sep=3.2ex]
      F_1 (u)
      \arrow[r]
      \arrow[d]
      &
      F_2 (u)
      \arrow[r]
      \arrow[d]
      &
      F_3 (u)
      \arrow[r]
      \arrow[d]
      &
      F_4 (u)
      \arrow[r]
      \arrow[d]
      &
      F_5 (u)
      \arrow[d]
      \\
      \varprojlim_k F_1(u_k)
      \arrow[r]
      &
      \varprojlim_k F_2(u_k)
      \arrow[r]
      &
      \varprojlim_k F_3(u_k)
      \arrow[r]
      &
      \varprojlim_k F_4(u_k)
      \arrow[r]
      &
      \varprojlim_k F_5(u_k)
    \end{tikzcd}
  \end{equation}
  are exact.
  Moreover,
  as the functors
  ${F_i \colon \M^{\circ} \rightarrow \VectF}$ for ${i=1,2,4,5}$
  are sequentially continuous,
  all four non-center vertical maps
  of \eqref{eq:seqContLatter} are isomorphisms.
  With this it follows from the five lemma that
  ${F_3 (u) \rightarrow \varprojlim_{k} F_3 (u_k)}$
  is an isomorphism as well.
\end{proof}

\begin{cor}
  \label{cor:tameTriaSubcat}
  The category
  $D^+_t (q_{\gamma})$
  is a triangulated subcategory of $D^+ (q_{\gamma})$.
\end{cor}

\begin{proof}
  Suppose
  \begin{equation*}
    F \rightarrow G \rightarrow H \rightarrow F[1]
  \end{equation*}
  is a distinguished triangle with $F$ and $G$ in $D^+_t (q_{\gamma})$.
  We have to show that
  ${h_{\gamma}(H) \colon \M^{\circ} \rightarrow \VectF}$
  is a functor in $\mathcal{J}$.
  By \cref{lem:cohomological} the functor
  $h_{\gamma}(H)$ is cohomological.
  Thus, it remains to show $h_{\gamma}(H)$ is pfd and sequentially continuous.
  As ${h_{\gamma} \colon D^+ (q_{\gamma}) \rightarrow \VectF^{\M^{\circ}}}$
  is cohomological,
  we obtain the exact sequence
  \begin{equation*}
    h_{\gamma}(F) \rightarrow
    h_{\gamma}(G) \rightarrow
    h_{\gamma}(H) \rightarrow
    h_{\gamma}(F[1]) \rightarrow
    h_{\gamma}(G[1])
  \end{equation*}
  with $h_{\gamma}(F)$, $h_{\gamma}(G)$, $h_{\gamma}(F[1])$,
  and $h_{\gamma}(G[1])$
  functors in $\mathcal{J}$ by assumption.
  In conjunction with \cref{lem:weakSerreJ}
  we obtain that
  $h_{\gamma}(H) \colon \M^{\circ} \rightarrow \VectF$
  is pfd and sequentially continuous as well.
\end{proof}

\begin{cor}
  \label{cor:tameVanishingTriaSubcat}
  The category
  $D^+_t (q_{\gamma}, \partial q)$
  is a triangulated subcategory of $D^+ (q_{\gamma})$.  
\end{cor}

\begin{proof}
  This follows
  from Corollaries \ref{cor:vanishingCoreflector} and \ref{cor:tameTriaSubcat}.
\end{proof}

\begin{cor}
  \label{cor:tameInteriorTriaSubcat}
  The category
  $D^+_t (\dot{q})$
  is a triangulated subcategory of $D^+ (\dot{q})$.    
\end{cor}

\begin{proof}
  The derived functor
  $R i_* \colon D^+(\dot{q}) \rightarrow D^+(q_{\gamma})$
  is triangulated,
  hence the result follows from
  \cref{dfn:tame} and \cref{cor:tameTriaSubcat}.
\end{proof}

\subsection{Alternative Construction of Induced Cohomological Functors}
\label{sec:altConstr}

Let $G$ be a bounded below complex of flabby sheaves on $q_{\gamma}$.
In the following we provide an alternative construction
of $h_{\gamma} (G) \colon \M^{\circ} \rightarrow \VectF$,
which will be useful later.
To this end,
we define the functor
\begin{equation*}
  \tilde{F}' \colon D \rightarrow C^+(\VectF)^{\circ},\,
  u \mapsto \Gamma_{q \cap (\uparrow u)}(q \cap \op{int}(\downarrow u); G)
  ,
\end{equation*}
where the internal maps are induced by inclusions.
Post-composing
${\tilde{F}' \colon D \rightarrow C^+(\VectF)^{\circ}}$
with the opposite graded cohomology functor
\[{H^{\bullet \circ} \colon
    C^+(\VectF)^{\circ} \rightarrow \left(\VectF^{\Z}\right)^{\circ},\,
    C \mapsto H^{\bullet}(C)}\]
we obtain the functor
\begin{equation*}
  F' \colon D \rightarrow \left(\VectF^{\Z}\right)^{\circ},\,
  u \mapsto H^{\bullet}_{q \cap (\uparrow u)}(q \cap \op{int}(\downarrow u); G)
\end{equation*}
as $G$ is a complex of flabby sheaves.
Now for $n \in \Z$ and $u \in D$ we have
an isomorphism
\begin{equation}
  \label{eq:phi^n}
  \varphi_u^n \colon h_{\gamma}(G)(T^{-n}(u))
  \xrightarrow{\cong}
  H^n_{q \cap (\uparrow u)}(q \cap \op{int}(\downarrow u); G)
  =
  (F'(u))^n
\end{equation}
by \cref{prp:localCoho},
which is natural in $u$.
Now in order to use this to connect our original construction
of $h_{\gamma}(G) \colon \M^{\circ} \rightarrow \VectF$
with this alternative approach,
let
${h_{\gamma}^{\# \circ}(G) \colon \M \rightarrow \left(\VectF^{\Z}\right)^{\circ}}$
be the transform of
${\left(h_{\gamma}(G)\right)^{\circ} \colon \M \rightarrow \VectF^{\circ}}$
under the $2$-adjunction from \cite[Lemma A.6]{2021arXiv210809298B}.
Then the family of isomorphisms $\{\varphi_u^n\}_{u \in D, n \in \Z}$
assembles to a
natural isomorphism
\begin{equation*}
  \varphi' \colon
  F'
  \xrightarrow{\cong}
  h_{\gamma}^{\# \circ}(G) |_D
  .
\end{equation*}
Now the missing ingredient,
in order to obtain a strictly stable functor
${F \colon \M \rightarrow \left(\VectF^{\Z}\right)^{\circ}}$
from
${F' \colon D \rightarrow \left(\VectF^{\Z}\right)^{\circ}}$
is a natural transformation
as in the diagram
\begin{equation}
  \label{eq:partialTrafo}
  \begin{tikzcd}
    R_D
    \arrow[rrr, "\op{pr}_1"]
    \arrow[ddd, "\op{pr}_2"']
    &[-23pt] & &[-18pt]
    D
    \arrow[ddd, "F'"]
    \\[-17pt]
    & &
    {}
    \arrow[ld, Rightarrow, "{\partial'}"]
    \\
    &
    {}
    \\[-20pt]
    T(D)
    \arrow[rrr, "\Sigma \circ F' \circ T^{-1}"']
    & & &
    \left(\VectF^{\Z}\right)^{\circ}
  \end{tikzcd}
\end{equation}
according to \cite[Proposition A.14]{2021arXiv210809298B},
where
\[R_D :=
  \{
  (w, \hat{u}) \in D \times T(D) \mid
  w \preceq \hat{u} \preceq T(w)
  \}\]
as in \cite[Definition A.8]{2021arXiv210809298B}.
To this end, let $(w, \hat{u}) \in R_D$, let $u := T^{-1}(\hat{u})$,
and let $v_1, v_2 \in [u, w]$
be the lower left respectively the upper right vertex of $[u, w]$
as shown in \cref{fig:constrBoundary}.
If we now instantiate \cref{cor:sesFlabbyLocalSections}
with
\begin{align*}
  X_1 & := q \cap \op{int}(\downarrow v_1),
  & A_1 & := q \setminus (\uparrow v_1),
  \\
  X_2 & := q \cap \op{int}(\downarrow v_2),
        \quad \text{and}
  & A_2 & := q \setminus (\uparrow v_2),
\end{align*}
then we obtain the short exact sequence
\begin{equation}
  \label{eq:sesCochainsAlternative}
  \begin{tikzcd}
    0
    \arrow[r]
    &
    \tilde{F}'(w)
    \arrow[r]
    &
    \tilde{F}'(v_1)
    \oplus
    \tilde{F}'(v_2)
    \arrow[r, "(1 ~ -1)"]
    &[1.5ex]
    \tilde{F}'(u) = (\tilde{F}' \circ T^{-1})(\hat{u})
    \arrow[r]
    &
    0
  \end{tikzcd}
\end{equation}
of cochain complexes in $\VectF$.
By the zig-zag lemma,
this short exact sequence yields a differential
\begin{equation*}
  \delta'_{(w, \hat{u})} \colon
  (\Sigma \circ F' \circ T^{-1})(\hat{u})
  \rightarrow
  F'(w).
\end{equation*}
Now let
\begin{equation*}
  \partial'_{(w, \hat{u})} :=
  \big(\delta'_{(w, \hat{u})}\big)^{\circ} \colon
  F'(w)
  \rightarrow
  (\Sigma \circ F' \circ T^{-1})(\hat{u})
\end{equation*}
be the corresponding homomorphism
in the opposite category $\left(\VectF^{\Z}\right)^\circ$
for all $(w, \hat{u}) \in R_D$.
Then $\partial'$ is a natural transformation as in \eqref{eq:partialTrafo}.
As it turns out,
the diagram
\begin{equation*}
  \begin{tikzcd}[column sep=15ex, row sep=6ex]
    F' \circ \op{pr}_1
    \arrow[d, "\partial'"', Rightarrow]
    \arrow[r, "\varphi' \circ \op{pr}_1", Rightarrow]
    &
    h_{\gamma}^{\# \circ}(G) \circ \op{pr}_1
    \arrow[d, "{\partial(h_{\gamma}^{\# \circ}(G), D)}", Rightarrow]
    \\
    \Sigma \circ F' \circ T^{-1} \circ \op{pr}_2
    \arrow[r, "\Sigma \circ \varphi' \circ T^{-1} \circ \op{pr}_2"', Rightarrow]
    &
    h_{\gamma}^{\# \circ}(G) \circ \op{pr}_2
  \end{tikzcd}
\end{equation*}
of functors and natural transformations commutes,
where $\partial(h_{\gamma}^{\# \circ}(G), D)$
is defined as in \cite[Definition A.9]{2021arXiv210809298B}.
By \cite[Proposition A.14]{2021arXiv210809298B}
this determines a unique strictly stable functor
$F \colon \M \rightarrow \left(\VectF^{\Z}\right)^{\circ}$
as well as a unique strictly stable natural transformation
$\varphi \colon F \rightarrow h_{\gamma}^{\# \circ}(G)$.
Moreover, as $\varphi$ is strictly stable
and as its restriction
$\varphi |_D = \varphi'$
to the fundamental domain $D$ is a natural isomorphism,
$\varphi \colon F \rightarrow h_{\gamma}^{\# \circ}(G)$
is a natural isomorphism as well.
With this we obtain the natural isomorphism
\begin{equation*}
  \op{ev}^0 \circ \varphi^{\circ} \colon
  h_{\gamma} (G) \rightarrow
  \op{ev}^0 \circ F^{\circ}
  ,
\end{equation*}
where
$\op{ev}^0 \colon \VectF^{\Z} \rightarrow \VectF,\,
M^{\bullet} \mapsto M^0$
is the evaluation at $0$ as in \cite[Lemma A.6]{2021arXiv210809298B}.

\subsection{Connection to Derived Level Set Persistence and RISC}
\label{sec:derivedLevelSet}

The following form of derived level set persistence
has been introduced by \cite{MR3259939,MR3873181}.
Let $f \colon X \rightarrow \R$ continuous function.
Then we may consider the derived pushforward
$R f_* \F_X$ as an object of the derived category $D^+(\R)$
of $\F$-linear sheaves on $\R$.
This construction can be made into a contravariant functor
in the following way.
For a commutative triangle
\begin{equation}
  \label{eq:homOverIR}
  \begin{tikzcd}
    X
    \arrow[rr, bend left, "\varphi"]
    \arrow[dr, "f"']
    &
    &
    Y
    \arrow[dl, "g"]
    \\
    &
    \R
  \end{tikzcd}
\end{equation}
of topological spaces,
we consider the unit
\begin{equation*}
  \eta^{\varphi}_{\F_Y} \colon
  \F_Y \rightarrow
  R \varphi_* \varphi^{-1} \F_Y \simeq R \varphi_* \F_X
\end{equation*}
of the derived adjunction $\varphi^{-1} \dashv R \varphi_*$
at $\F_Y$.
Applying the functor $R g_*$ to $\eta^{\varphi}_{\F_X}$
we obtain the homomorphism
\begin{equation*}
  R \varphi_* \F_{\varphi} \colon
  R g_* \F_Y \xrightarrow{(R g_* \circ \eta^{\varphi})_{\F_Y}}
  R g_* R \varphi_* \F_X \simeq R (g \circ \varphi)_* \F_X = R f_* \F_X
  ,
\end{equation*}
see also \cite[(2.7.4)]{Kashiwara1990}.
This way we obtain the functor
\begin{equation*}
  R (-)_* \F_{(-)} \colon
  (\mathrm{Top} / \R)^{\circ} \rightarrow D^+(\R),\,
  (f \colon X \rightarrow \R) \mapsto R f_* \F_X
\end{equation*}
from the opposite category of
the category of topological spaces over the reals $\mathrm{Top} / \R$
to the derived category $D^+(\R)$.
As a note of caution,
we point out that this notation is not being used consistently
across the literature
as \cite{2019arXiv190709759B}
use the same notation for the functor
\begin{equation}
  \label{eq:dirImCoho}
  \mathrm{Top} / \R \rightarrow D^+(\R),\,
  (f \colon X \rightarrow \R) \mapsto \bigoplus_{n=0}^{\infty} R^n f_* \F_X [-i]
  .
\end{equation}
Below we provide the \cref{exm:hood},
which also shows that the functor \eqref{eq:dirImCoho}
and the functor we denote as $R (-)_* \F_{(-)}$ are not naturally isomorphic.

We now assume we are in the setting of \cite[Section 2]{2021arXiv210809298B}.
More specifically,
we assume $l_0$ and $l_1$ intersect the $x$-axis
in $-\pi$ and $\pi$ respectively
and that ${Q = \downarrow \op{Im} \blacktriangle}$,
where
\[{\blacktriangle = \Delta \circ \arctan \colon
    \overline{\R} = [-\infty, \infty] \rightarrow \M,\,
    t \mapsto (\arctan t, \arctan t)
    .
  }
\]
Then the restriction
$\blacktriangle |_{\R} \colon \R \rightarrow q_{\gamma}$
yields an embedding of $\R$ onto $\dot{q} \subset q_{\gamma}$
as shown in \cref{fig:embeddingReals}.
\begin{figure}[t]
  \centering
\begingroup%
  \makeatletter%
  \providecommand\color[2][]{%
    \errmessage{(Inkscape) Color is used for the text in Inkscape, but the package 'color.sty' is not loaded}%
    \renewcommand\color[2][]{}%
  }%
  \providecommand\transparent[1]{%
    \errmessage{(Inkscape) Transparency is used (non-zero) for the text in Inkscape, but the package 'transparent.sty' is not loaded}%
    \renewcommand\transparent[1]{}%
  }%
  \providecommand\rotatebox[2]{#2}%
  \newcommand*\fsize{\dimexpr\f@size pt\relax}%
  \newcommand*\lineheight[1]{\fontsize{\fsize}{#1\fsize}\selectfont}%
  \ifx\svgwidth\undefined%
    \setlength{\unitlength}{306bp}%
    \ifx\svgscale\undefined%
      \relax%
    \else%
      \setlength{\unitlength}{\unitlength * \real{\svgscale}}%
    \fi%
  \else%
    \setlength{\unitlength}{\svgwidth}%
  \fi%
  \global\let\svgwidth\undefined%
  \global\let\svgscale\undefined%
  \makeatother%
  \begin{picture}(1,0.41666667)%
    \lineheight{1}%
    \setlength\tabcolsep{0pt}%
    \put(0,0){\includegraphics[width=\unitlength,page=1]{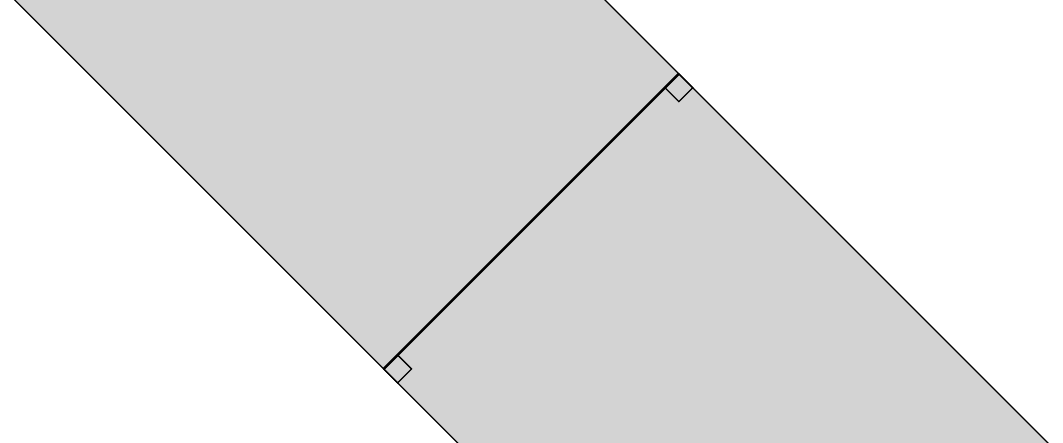}}%
    \put(0.41203701,0.26851863){\makebox(0,0)[t]{\lineheight{1.25}\smash{\begin{tabular}[t]{c}$\M$\end{tabular}}}}%
    \put(0.7777777,0.23148137){\makebox(0,0)[t]{\lineheight{1.25}\smash{\begin{tabular}[t]{c}$ \pi $\end{tabular}}}}%
    \put(0.2222223,0.17129632){\makebox(0,0)[t]{\lineheight{1.25}\smash{\begin{tabular}[t]{c}$-\pi~~$\end{tabular}}}}%
    \put(0,0){\includegraphics[width=\unitlength,page=2]{embedding-reals.pdf}}%
    \put(0.54629632,0.17129632){\makebox(0,0)[t]{\lineheight{1.25}\smash{\begin{tabular}[t]{c}$\op{Im} \blacktriangle$\end{tabular}}}}%
    \put(0.72288015,0.30026789){\makebox(0,0)[t]{\lineheight{1.25}\smash{\begin{tabular}[t]{c}$l_1$\end{tabular}}}}%
    \put(0.2855848,0.09867451){\makebox(0,0)[t]{\lineheight{1.25}\smash{\begin{tabular}[t]{c}$l_0$\end{tabular}}}}%
  \end{picture}%
\endgroup%

  \caption{
    The strip $\M$ and
    the image of the embedding
    $\blacktriangle \colon \overline{\R} \rightarrow \M$.
  }
  \label{fig:embeddingReals}
\end{figure}
The induced tessellation of for this particular choice of $Q \subset \M$
is shown in \cref{fig:tessellationReals}.
Moreover,
we have ${q_{\gamma} = \op{Im} \blacktriangle \cong \overline{\R}}$.
Thus, by post-composing the functor
$R (\blacktriangle |_{\R})_* \colon D^+(\R) \rightarrow D^+(q_{\gamma})$
with 
$h_{\gamma}$
we obtain the functor
\begin{equation*}
  h_{\R} \colon
  D^+(\R) \xrightarrow{R (\blacktriangle |_{\R})_*}
  D^+(q_{\gamma}) \xrightarrow{h_{\gamma}}
  \VectF^{\M^{\circ}}
  .
\end{equation*}
Furthermore,
by composing
${R (-)_* \F_{(-)} \colon (\mathrm{Top} / \R)^{\circ} \rightarrow D^+(\R)}$
and
${h_{\R} \colon D^+(\R) \rightarrow \VectF^{\M^{\circ}}}$
we obtain a functor
\begin{equation*}
  (\mathrm{Top} / \R)^{\circ} \xrightarrow{R (-)_* \F_{(-)}}
  D^+ (\R) \xrightarrow{h_{\R}}
  \VectF^{\M^{\circ}}
  .
\end{equation*}
Now in \cite[Section 2]{2021arXiv210809298B}
we have already provided the functor
$h \colon (\mathrm{Top} / \R)^{\circ} \rightarrow \VectF^{\M^{\circ}}$.
In the following we show that these two functors
are naturally isomorphic
when restricted to functions on locally contractible spaces.
More specifically,
let $\mathrm{lcContr}$ denote the full subcategory
of locally contractible topological spaces.
Before we construct a natural isomorphism $\zeta$
as in the diagram
\begin{equation*}
  \begin{tikzcd}[row sep=12ex, column sep=7ex]
    (\mathrm{lcContr} / \R)^{\circ}
    \arrow[d , "R (-)_* \F_{(-)}"']
    \arrow[dr, "h"{name=h0}]
    \\
    D^+ (\R)
    \arrow[r, "h_{\R}"']
    \arrow[to=h0, Leftarrow, "\zeta", shorten >=1.5ex, shorten <=1.5ex]
    &
    \VectF^{\M^{\circ}}
    ,
  \end{tikzcd}
\end{equation*}
we illustrate the behavior of all three of these functors
with an example.

\begin{exm}
  \label{exm:hood}
  Let $C_4$ be the cyclic graph on four vertices
  $\{1, 2, 3, 4\}$.
  Moreover,
  let $5 * C_4$ be the abstract simplicial cone over $C_4$
  with $5$ as the tip of the cone,
  let $A$ be the full subcomplex of $5 * C_4$
  spanned by the four vertices $\{1,2,3,5\}$,
  let $X := |A|$ and $Y := |5 * C_4|$
  be the corresponding geometric realizations,
  and let $\varphi \colon X \hookrightarrow Y$
  be the corresponding inclusion.
  Furthermore,
  let $a < b < c$,
  let
  $g \colon Y \rightarrow \R$
  be the unique simplexwise linear function with
  \begin{align*}
    g(1) = g(3) = g(4) & = a, \\
    g(2) & = b, \\
    \text{and} \quad
    g(5) & = c,
  \end{align*}
  and let $f := g |_X \colon X \rightarrow \R$
  be the restriction of $g \colon Y \rightarrow \R$
  to $X$.
  Then the diagram \eqref{eq:homOverIR} commutes;
  see also
  \cref{fig:hood}.
  \begin{figure}[t]
    \centering
    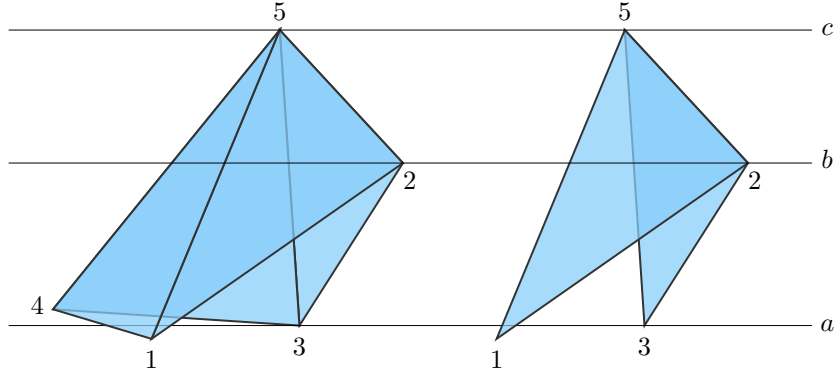
    \caption{
      The geometric realization
      $Y := |5 * C_4|$ on the left hand side
      and
      $X := |A|$
      on the right hand side,
      making $g \colon Y \rightarrow \R$
      and $f \colon X \rightarrow \R$
      the corresponding height functions
      of the depicted embeddings into $\R^3$.
    }
    \label{fig:hood}
  \end{figure}
  Now let
  \begin{equation*}
    \bar{a} := \arctan a, \quad \bar{b} := \arctan b, \quad
    \text{and} \quad
    \bar{c} := \arctan c.
  \end{equation*}
  Then we have
  \begin{equation}
    \label{eq:exampleRISCdecomp}
    \begin{split}
      h(g) & \cong
      B_{(\bar{c}, \bar{a})}
      \oplus
      B_{\left(\pi-\bar{b}, \bar{c}-2\pi\right)}
      \\
      \text{and} \quad
      h(f) & \cong
      B_{(\bar{c}, \bar{a})}
      \oplus
      B_{\left(\pi-\bar{b}, \bar{a}\right)}      
    \end{split}
  \end{equation}
  or, in other words
  \begin{align*}
    (\beta^0 \circ h)(g)
    =
    \mathbf{1}_{
    \left\{(\bar{c}, \bar{a}), \left(\pi-\bar{b}, \bar{c}-2\pi\right)\right\}
    }
    \quad \text{and} \quad
    (\beta^0 \circ h)(f)
    =
    \mathbf{1}_{
    \left\{(\bar{c}, \bar{a}), \left(\pi-\bar{b}, \bar{a}\right)\right\}
    }
    ,
  \end{align*}
  where
  ${\mathbf{1}_{
      \left\{(\bar{c}, \bar{a}), \left(\pi-\bar{b}, \bar{a}\right)\right\}
    } \colon \op{int} \M \rightarrow \N_0}$
  is the indicator function of the subset
  ${\left\{(\bar{c}, \bar{a}), \left(\pi-\bar{b}, \bar{a}\right)\right\} \subset
    \op{int} \M}$;
  see also \cref{fig:hoodDiagram}.
  \begin{figure}[t]
    \centering
    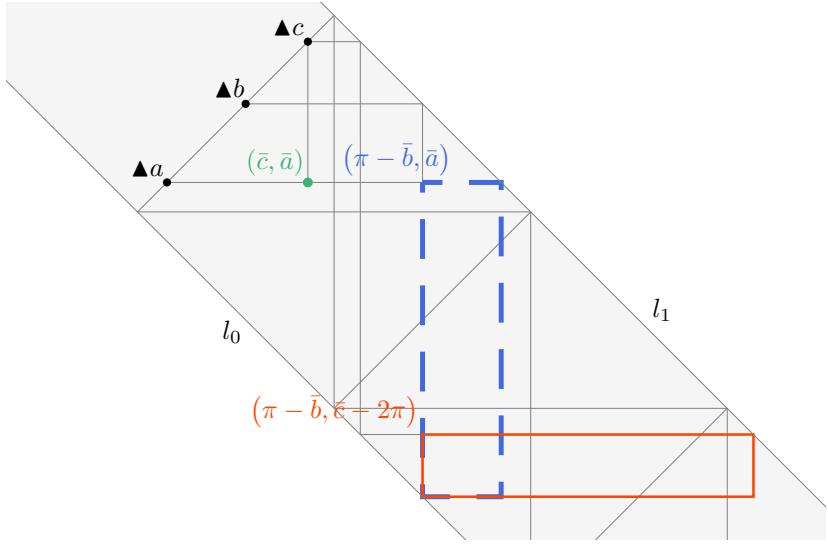
    \caption{
      The indecomposables
      $B_{\left(\pi-\bar{b}, \bar{a}\right)}$
      and
      $B_{\left(\pi-\bar{b}, \bar{c}-2\pi\right)}$.
    }
    \label{fig:hoodDiagram}
  \end{figure}
  We consider the natural transformation
  \begin{equation*}
    \eta := 
    B_{\left(\pi-\bar{b}, \bar{c}-2\pi\right) \preceq \left(\pi-\bar{b}, \bar{a}\right)}
    \colon
    B_{\left(\pi-\bar{b}, \bar{c}-2\pi\right)}
    \rightarrow
    B_{\left(\pi-\bar{b}, \bar{a}\right)}
    .
  \end{equation*}
  As is shown in \cref{fig:hoodDiagram} the supports of
  $B_{(\pi-\bar{b}, \bar{a})}$ and
  $B_{(\pi-\bar{b}, \bar{c}-2\pi)}$
  have a non-empty intersection.
  For any $u$ in this intersection,
  $\eta_u$
  maps
  $B_{(\pi-\bar{b}, \bar{c}-2\pi)} (u) = \F$
  identically onto
  $B_{(\pi-\bar{b}, \bar{a})} (u) = \F$.
  As it turns out, we have the commutative diagram
  \begin{equation}
    \label{eq:hoodSquareRISC}
    \begin{tikzcd}[row sep=5ex]
      B_{(\bar{c}, \bar{a})}
      \oplus
      B_{\left(\pi-\bar{b}, \bar{c}-2\pi\right)}
      \arrow[r, "\cong"]
      \arrow[d, "B_{(\bar{c}, \bar{a})} \oplus \eta"']
      &
      h(g)
      \arrow[d, "h(\varphi)"]
      \\
      B_{(\bar{c}, \bar{a})}
      \oplus
      B_{\left(\pi-\bar{b}, \bar{a}\right)}
      \arrow[r, "\cong"]
      &
      h(f)
      ,
    \end{tikzcd}
  \end{equation}
  when the isomorphisms \eqref{eq:exampleRISCdecomp} are chosen appropriately.
  Now
  \begin{equation}
    \label{eq:exampleSheavesDecomp}
    \begin{split}
      R g_* \F_Y
      & \cong
      \F_{[a,c]}
      \oplus
      \F_{[b,c)} [-1]
      \cong
      \blacktriangle^{-1} \iota(\bar{c}, \bar{a})
      \oplus
      \blacktriangle^{-1} \iota\left(\pi-\bar{b}, \bar{c}-2\pi\right)
      \\
      \text{and} \quad
      R f_* \F_X
      & \cong
      \F_{[a,c]}
      \oplus
      \F_{[a,b)}
      \cong
      \blacktriangle^{-1} \iota(\bar{c}, \bar{a})
      \oplus
      \blacktriangle^{-1} \iota\left(\pi-\bar{b}, \bar{a}\right)
      .      
    \end{split}
  \end{equation}
  Moreover,
  considering the short exact sequence
  \begin{equation*}
    0 \rightarrow
    \F_{[a, b)} \rightarrow
    \F_{[a,c)} \rightarrow
    \F_{[b, c)} \rightarrow
    0
  \end{equation*}
  of sheaves
  we obtain the distinguished triangle
  \begin{equation*}
    \F_{[a, b)} \rightarrow
    \F_{[a,c)} \rightarrow
    \F_{[b, c)} \xrightarrow{\partial}
    \F_{[a, b)} [1]
    .
  \end{equation*}
  Furthermore,
  we have commutative diagrams
  \begin{equation*}
    \begin{tikzcd}[row sep=8ex]
      \blacktriangle^{-1} \iota\left(\pi-\bar{b}, \bar{c}-2\pi\right)
      \arrow[r, "\cong"]
      \arrow[
      d,
      "{\blacktriangle^{-1}
        \iota\left(
          \left(\pi-\bar{b}, \bar{c}-2\pi\right)
          \preceq
          \left(\pi-\bar{b}, \bar{a}\right)
        \right)
      }" description
      ]
      &
      \F_{[b, c)} [-1]
      \arrow[d, "{\partial[-1]}"]
      \\
      \blacktriangle^{-1} \iota\left(\pi-\bar{b}, \bar{a}\right)
      \arrow[r, "\cong"]
      &
      \F_{[a,b)}
    \end{tikzcd}
  \end{equation*}
  and
  \begin{equation}
    \label{eq:hoodSquareSheaves}
    \begin{tikzcd}[row sep=8ex]
      \F_{[a,c]}
      \oplus
      \F_{[b,c)} [-1]
      \arrow[r, "\cong"]
      \arrow[d, "{\F_{[a,c]} \oplus \partial[-1]}"']
      &
      R g_* \F_Y
      \arrow[d, "R \varphi_* \F_{\varphi}"]
      \\
      \F_{[a,c]}
      \oplus
      \F_{[a,b)}
      \arrow[r, "\cong"]
      &
      R f_* \F_X
      ,
    \end{tikzcd}
  \end{equation}
  when the isomorphisms \eqref{eq:exampleSheavesDecomp}
  are chosen appropriately.
  Now the functor
  ${h_{\R} \colon D^+(\R) \rightarrow \VectF^{\M^{\circ}}}$
  maps the commutative square \eqref{eq:hoodSquareSheaves}
  to a square isomorphic to \eqref{eq:hoodSquareRISC}.
  Moreover, \cref{fig:hoodDiagram} shows that the intersection
  of the supports of
  $B_{\left(\pi-\bar{b}, \bar{a}\right)}$
  and
  $B_{\left(\pi-\bar{b}, \bar{c}-2\pi\right)}$
  is disjoint from the regions of $\M$
  covered by the south faces of the Mayer--Vietoris pyramids
  \cite{Carlsson:2009:ZPH:1542362.1542408}
  in the tessellation \cref{fig:tessellationReals}.
  This way we obtain a geometric explanation
  of the phenomenon described in \cite[Remark 4.9]{2019arXiv190709759B}.
\end{exm}

Now let ${f \colon X \rightarrow \R}$
be a continuous function with $X$ a locally contractible topological space.
Moreover,
let $C^{\bullet}$ be the presheaf of singular cochains
with coefficients in $\F$ on $X$
and let
$\epsilon \colon \F_X \rightarrow C^{\bullet}$
be the embedding of $\F_X$ as the subpresheaf
of $0$-cocycles of $C^{\bullet}$.
By \cite{2016arXiv160206674S} there is a complex $\mathcal{F}$
of flabby sheaves on $X$
together with a quasi-isomorphism of complexes of presheaves
${\tilde{\psi} \colon C^{\bullet} \rightarrow \mathcal{F}}$
such that ${\tilde{\psi} \circ \epsilon \colon \F_X \rightarrow \mathcal{F}}$
is a quasi-isomorphism of complexes of sheaves.
For ${u \in D}$ we consider the commutative diagram
\begin{equation}
  \label{eq:relCochainsLocalSections}
  \begin{tikzcd}
    0
    \arrow[d]
    &
    0
    \arrow[d]
    \\
    (C^{\bullet} \circ f^{-1} \circ \blacktriangle^{-1})
    (\op{int}(\downarrow T(u)), \M \setminus (\uparrow u))
    \arrow[r, dashed]
    \arrow[d]
    &
    \Gamma_{(\blacktriangle \circ f)^{-1}(\uparrow u)}
    ((\blacktriangle \circ f)^{-1}(\op{int}(\downarrow T(u))); \mathcal{F})
    \arrow[d]
    \\
    (C^{\bullet} \circ f^{-1} \circ \blacktriangle^{-1})
    (\op{int}(\downarrow T(u)))
    \arrow[r, "\tilde{\psi}_{(\blacktriangle \circ f)^{-1}(\op{int}(\downarrow T(u)))}"]
    \arrow[d]
    &
    (\mathcal{F} \circ f^{-1} \circ \blacktriangle^{-1})
    (\op{int}(\downarrow T(u)))
    \arrow[d]
    \\
    (C^{\bullet} \circ f^{-1} \circ \blacktriangle^{-1})
    (\M \setminus (\uparrow u))
    \arrow[r, "\tilde{\psi}_{(\blacktriangle \circ f)^{-1}(\M \setminus (\uparrow u))}"]
    \arrow[d]
    &
    (\mathcal{F} \circ f^{-1} \circ \blacktriangle^{-1})
    (\M \setminus (\uparrow u))
    \arrow[d]
    \\
    0
    &
    \,
    0
    .
  \end{tikzcd}
\end{equation}
By definition of the relative singular cochain complex
$(C^{\bullet} \circ f^{-1} \circ \blacktriangle^{-1})
(\op{int}(\downarrow T(u)), \M \setminus (\uparrow u))$
and the set of local sections
$\Gamma_{(\blacktriangle \circ f)^{-1}(\uparrow u)}
((\blacktriangle \circ f)^{-1}(\op{int}(\downarrow T(u))); \mathcal{F})$
both columns are exact.
In particular, the horizontal dashed arrow exists as indicated.
Moreover,
the lower two horizontal arrows
are quasi-isomorphisms of cochain complexes
as $\tilde{\psi} \colon C^{\bullet} \rightarrow \mathcal{F}$
is a quasi-isomorphism of complexes of presheaves.
Thus, the dashed arrow is a quasi-isomorphism of cochain complexes as well.
Taking the $\Z$-graded cohomology of the cochain complex
$(C^{\bullet} \circ f^{-1} \circ \blacktriangle^{-1})
(\op{int}(\downarrow T(u)), \M \setminus (\uparrow u))$
we obtain
\begin{equation}
  \label{eq:cohoRISC}
  (H^{\bullet} \circ f^{-1} \circ \blacktriangle^{-1})
  (\op{int}(\downarrow T(u)), \M \setminus (\uparrow u))
  =
  h^{\#}(f)(u)
  ,
\end{equation}
where $h(f) \colon \M^{\circ} \rightarrow \VectF$
is the relative interlevel set cohomology (RISC)
of $f \colon X \rightarrow \R$
as defined in \cite[Section 2]{2021arXiv210809298B}
and $h^{\#}(f) \colon \M^{\circ} \rightarrow \VectF^{\Z}$
is the transform of $h(f)$ under the $2$-adjunction
from \cite[Lemma A.6]{2021arXiv210809298B}.
(Strictly speaking, we should only use the term RISC,
when $f \colon X \rightarrow \R$ is $\F$-tame,
so we may still adapt this notion to more general settings.)
Moreover,
we have
\begin{equation*}
  \begin{split}
    \Gamma_{(\blacktriangle \circ f)^{-1}(\uparrow u)}
    ((\blacktriangle \circ f)^{-1}(\op{int}(\downarrow T(u))); \mathcal{F})
    & =
    \Gamma_{q \cap (\uparrow u)}
    (q \cap \op{int}(\downarrow T(u)); (\blacktriangle \circ f)_* \mathcal{F})
    \\
    & =
    \Gamma_{q \cap (\uparrow u)}
    (q \cap \op{int}(\downarrow T(u)); G)
    ,
  \end{split}
\end{equation*}
where $G := (\blacktriangle \circ f)_* \mathcal{F}$.
As both $\mathcal{F}$ and $G$ are complexes of flabby sheaves,
this implies that
\begin{equation}
  \label{eq:localSectionsPushforward}
  H^{\bullet}_{(\blacktriangle \circ f)^{-1}(\uparrow u)}
  ((\blacktriangle \circ f)^{-1}(\op{int}(\downarrow T(u))); \mathcal{F})
  \cong
  H^{\bullet}_{q \cap (\uparrow u)}
  (q \cap \op{int}(\downarrow T(u)); G)
  =
  F'(u)
  ,
\end{equation}
where
\begin{equation*}
  F' \colon D \rightarrow \left(\VectF^{\Z}\right)^{\circ},\,
  u \mapsto H^{\bullet}_{q \cap (\uparrow u)}(q \cap \op{int}(\downarrow u); G)
\end{equation*}
as in the previous \cref{sec:altConstr}.
Altogether,
\eqref{eq:cohoRISC},
the dashed arrow in \eqref{eq:relCochainsLocalSections},
and \eqref{eq:localSectionsPushforward} yield a natural isomorphism
\begin{equation*}
  \psi' \colon
  F'
  \xrightarrow{\cong}
  h^{\# \circ}(f) |_D
  .
\end{equation*}
Now in order
to apply \cite[Proposition A.14]{2021arXiv210809298B}
to extend $\psi'$
to a strictly stable natural isomorphism
\begin{equation*}
  \psi \colon
  F
  \rightarrow
  h^{\# \circ}(f)
  ,
\end{equation*}
where
${F \colon \M \rightarrow \left(\VectF^{\Z}\right)^{\circ}}$
is defined as in the previous \cref{sec:altConstr},
we have to show that the diagram
\begin{equation}
  \label{eq:psiCompatStrictlyStable}
  \begin{tikzcd}[column sep=15ex, row sep=6ex]
    F \circ \op{pr}_1
    \arrow[d, "\partial'"', Rightarrow]
    \arrow[r, "\psi' \circ \op{pr}_1", Rightarrow]
    &
    h_{0}^{\# \circ}(f) \circ \op{pr}_1
    \arrow[d, "{\partial(h_{0}^{\# \circ}(f), D)}", Rightarrow]
    \\
    F \circ \op{pr}_2
    \arrow[r, "\Sigma \circ \psi' \circ T^{-1} \circ \op{pr}_2"', Rightarrow]
    &
    h_{0}^{\# \circ}(f) \circ \op{pr}_2
  \end{tikzcd}
\end{equation}
of functors and natural transformations commutes,
where $\partial(h_{0}^{\# \circ}(f), D)$
is defined as in \cite[Definition A.9]{2021arXiv210809298B}.
To this end,
it suffices to check the commutativity of \eqref{eq:psiCompatStrictlyStable}
for all pairs $(w, \hat{u}) \in R_D$
with the same $x$- or the same $y$-coordinate
and both within the square $\left(-\frac{\pi}{2}, \frac{\pi}{2}\right)^2$.
We treat the case, where $w$ and $\hat{u}$ have the same $y$-coordinate,
the other case is similar.
In this case there are real numbers $a < b < c$
with
\begin{equation*}
  \hat{u} = (\arctan a, \arctan b)
  \quad \text{and} \quad
  w = (\arctan c, \arctan b)
  .
\end{equation*}
Then we have
\begin{align*}
  \left(h_{0}^{\#}(f) \circ \op{pr}_1\right)(w)
  & =
    (H^{\bullet} \circ f^{-1})(\R, \R \setminus [a, b])
  \\
  \text{and} \quad
  \left(h_{0}^{\#}(f) \circ \op{pr}_2\right)(\hat{u})
  & =
    (H^{\bullet-1} \circ f^{-1})((a, b))
    .
\end{align*}
Moreover,
\begin{equation}
  \label{eq:mvsDiffCompatStrictlyStable}
  \partial(h_{0}^{\# \circ}(f), D)^{\circ}_{(w, \hat{u})} \colon
  (H^{\bullet-1} \circ f^{-1})((a, b)) \rightarrow
  (H^{\bullet} \circ f^{-1})(\R, \R \setminus [a, b])
\end{equation}
is the differential
of the Mayer--Vietoris sequence associated
to the square-shaped sublattice
\begin{equation*}
  \begin{tikzcd}
    f^{-1} (\R, \R \setminus [b, c])
    &
    f^{-1} ((-\infty, b), (-\infty, b))
    \arrow[l, hook']
    \\
    f^{-1} ((a, \infty), (c, \infty))
    \arrow[u, hook']
    &
    f^{-1} ((a, b), \emptyset)
    \arrow[l, hook']
    \arrow[u, hook']
  \end{tikzcd}
\end{equation*}
of pairs of open subspaces of $X$.
Now considering the right hand side of
\eqref{eq:psiCompatStrictlyStable}
we have
\begin{align*}
  \left(F \circ \op{pr}_1\right)(w)
  & \cong
    H^{\bullet}_{f^{-1}([b,c])} (X; \mathcal{F})
  \\
  \text{and} \quad
  \left(F \circ \op{pr}_2\right)(\hat{u})
  & \cong
    H^{\bullet-1} (f^{-1}((a, b)); \mathcal{F})
    .
\end{align*}
Under these isomorphisms the map
${\delta'_{(w, \hat{u})} \colon
  \left(F' \circ \op{pr}_2\right)(\hat{u}) \rightarrow
  \left(F' \circ \op{pr}_1\right)(w)}$
corresponds to the differential
associated to the short exact sequence
\begin{equation}
  \label{eq:sesSheavesBeforeExcision}
  0 \rightarrow
  \Gamma_{f^{-1}([b, c])} (X; \mathcal{F}) \rightarrow
  \Gamma_{f^{-1}((-\infty, c])} (f^{-1}((a, \infty)); \mathcal{F}) \rightarrow
  \Gamma (f^{-1}((a, b)); \mathcal{F}) \rightarrow
  0
\end{equation}
of cochain complexes in $\VectF$.
Now in order to show the commutativity of \eqref{eq:psiCompatStrictlyStable}
at ${(w, \hat{u}) \in R_D}$
we show that the Mayer--Vietoris differential
\eqref{eq:mvsDiffCompatStrictlyStable}
can be realized as the differential associated to a short exact sequence
of cochain complexes as well.
To this end,
we consider the sublattice
\begin{equation}
  \label{eq:sublatticePairsIR}
  \begin{tikzcd}[column sep=2.8ex]
    (\R, \R \setminus [b, c])
    &
    (\R \setminus [b, c], \R \setminus [b, c])
    \arrow[l, hook']
    &
    ((-\infty, b), (-\infty, b))
    \arrow[l, hook']
    \\
    ((a, \infty), (a, b) \cup (c, \infty))
    \arrow[u, hook']
    &
    ((a, b) \cup (c, \infty), (a, b) \cup (c, \infty))
    \arrow[l, hook']
    \arrow[u, hook']
    &
    ((a, b), (a, b))
    \arrow[l, hook']
    \arrow[u, hook']
    \\
    ((a, \infty), (c, \infty))
    \arrow[u, hook']
    &
    ((a, b) \cup (c, \infty), (c, \infty))
    \arrow[l, hook']
    \arrow[u, hook']
    &
    ((a, b), \emptyset)
    .
    \arrow[l, hook']
    \arrow[u, hook']
  \end{tikzcd}
\end{equation}
Now all inclusions of \eqref{eq:sublatticePairsIR}
other than those of the lower left square
induce isomorphisms in singular cohomology by excision.
Thus,
the Mayer--Vietoris sequence associated to the outer square
of \eqref{eq:sublatticePairsIR}
is isomorphic
to the Mayer--Vietoris sequence associated to the lower left square
of \eqref{eq:sublatticePairsIR}.
Moreover,
the Mayer--Vietoris sequence associated to the lower left square
of \eqref{eq:sublatticePairsIR}
is the same as the long exact sequence associated to the triple
${f^{-1} ((a, \infty), (a, b) \cup (c, \infty), (c, \infty))}$,
which is the long exact sequence associated to the short exact sequence
\begin{equation}
  \label{eq:sesAfterExcision}
  \begin{tikzcd}
    0
    \arrow[d]
    \\
    (C^{\bullet} \circ f^{-1}) ((a, \infty), (a, b) \cup (c, \infty))
    \arrow[d]
    \\
    (C^{\bullet} \circ f^{-1}) ((a, \infty), (c, \infty))
    \arrow[d]
    \\
    (C^{\bullet} \circ f^{-1}) ((a, b) \cup (c, \infty), (c, \infty))
    \arrow[d]
    \\
    0    
  \end{tikzcd}
\end{equation}
of cochain complexes in $\VectF$.
Furthermore,
the short exact sequence \eqref{eq:sesSheavesBeforeExcision}
is isomorphic to the short exact sequence
\begin{equation}
  \label{eq:sesSheavesAfterExcision}
  \begin{tikzcd}
    0
    \arrow[d]
    \\
    \Gamma_{f^{-1}([b, c])}(f^{-1}((a, \infty)); \mathcal{F})
    \arrow[d]
    \\
    \Gamma_{f^{-1}((a, c])}(f^{-1}((a, \infty)); \mathcal{F})
    \arrow[d]
    \\
    \Gamma_{f^{-1}((a, c])}(f^{-1}((a, b) \cup (c, \infty)); \mathcal{F})
    \arrow[d]
    \\
    0    
  \end{tikzcd}
\end{equation}
by the sheaf condition.
Now the presheaf homomorphism
$\tilde{\psi} \colon C^{\bullet} \rightarrow \mathcal{F}$
induces a homomorphism
\begin{equation}
  \label{eq:sesPresheafHom}
  \begin{tikzcd}
    0
    \arrow[d]
    &
    0
    \arrow[d]
    \\
    (C^{\bullet} \circ f^{-1}) ((a, \infty), (a, b) \cup (c, \infty))
    \arrow[d]
    \arrow[r]
    &
    \Gamma_{f^{-1}([b, c])}(f^{-1}((a, \infty)); \mathcal{F})
    \arrow[d]
    \\
    (C^{\bullet} \circ f^{-1}) ((a, \infty), (c, \infty))
    \arrow[d]
    \arrow[r]
    &
    \Gamma_{f^{-1}((a, c])}(f^{-1}((a, \infty)); \mathcal{F})
    \arrow[d]
    \\
    (C^{\bullet} \circ f^{-1}) ((a, b) \cup (c, \infty), (c, \infty))
    \arrow[d]
    \arrow[r]
    &
    \Gamma_{f^{-1}((a, c])}(f^{-1}((a, b) \cup (c, \infty)); \mathcal{F})
    \arrow[d]
    \\
    0
    &
    0
  \end{tikzcd}
\end{equation}
of short exact sequences.
Moreover,
the horizontal arrow at the top of \eqref{eq:sesPresheafHom}
induces the homomorphism
${\psi'_w \colon F(w) \rightarrow h_{0}^{\# \circ}(f)(w)}$
under the aforementioned isomorphisms.
Similarly the horizontal arrow at the bottom of \eqref{eq:sesPresheafHom}
induces the homomorphism
${(\Sigma \circ \psi' \circ T^{-1})_{\hat{u}} \colon
  F(\hat{u}) \rightarrow h_{0}^{\# \circ}(f)(\hat{u})}$.
Thus, the diagram
\begin{equation}
  \label{eq:psiCompatStrictlyStableConcrete}
  \begin{tikzcd}[column sep=15ex, row sep=6ex]
    F(w)
    \arrow[d, "\partial'_{(w, \hat{u})}"']
    \arrow[r, "\psi'_w"]
    &
    h_{0}^{\# \circ}(f)(w)
    \arrow[d, "{\partial(h_{0}^{\# \circ}(f), D)_{(w, \hat{u})}}"]
    \\
    F(\hat{u})
    \arrow[r, "(\Sigma \circ \psi' \circ T^{-1})_{\hat{u}}"']
    &
    h_{0}^{\# \circ}(f)(\hat{u})
  \end{tikzcd}
\end{equation}
commutes.
By a similar argument \eqref{eq:psiCompatStrictlyStableConcrete}
commutes for any pair $(w, \hat{u}) \in R_D$
with identical $x$-coordinates and both points within the square
$\left(-\frac{\pi}{2}, \frac{\pi}{2}\right)^2$.
We just need to keep track of an additional sign,
as in this case the long exact sequence of the associated triple
\enquote{passes through the second direct summand
  of the Mayer--Vietoris sequence}.
Now for an arbitrary pair $(w, \hat{u}) \in R_D$
we can always find a pair $(w', \hat{u}') \in R_D$
with the same $x$- or the same $y$-coordinate
and both within the square $\left(-\frac{\pi}{2}, \frac{\pi}{2}\right)^2$
and $w \preceq w' \preceq \hat{u}' \preceq \hat{u}$,
hence the diagram \eqref{eq:psiCompatStrictlyStable} commutes.
As a result,
${\psi' \colon F' \xrightarrow{\cong} h^{\# \circ}(f) |_D}$
extends to a unique strictly stable natural transformation
${\psi \colon F \rightarrow h^{\# \circ}(f)}$
by \cite[Proposition A.14]{2021arXiv210809298B}.
Moreover,
as ${\psi' \colon F' \xrightarrow{\cong} h^{\# \circ}(f) |_D}$
is a natural isomorphism,
${\psi \colon F \rightarrow h^{\# \circ}(f)}$
is a natural isomorphism as well.
In conjunction with the previous \cref{sec:altConstr}
we obtain the span
\begin{equation*}
  h^{\# \circ}(f)
  \xLongleftarrow{\psi}
  F
  \xLongrightarrow{\varphi}
  h_{\gamma}^{\# \circ}(G)
\end{equation*}
of strictly stable functors and
and strictly stable natural isomorphisms.
Applying the oppositization $2$-functor
to this span and using the $2$-adjunction from
\cite[Lemma A.6]{2021arXiv210809298B}
we obtain a natural isomorphism
\begin{equation}
  \label{eq:RISCtoSellaPushforward}
  h(f) \cong h_{\gamma}(G)
  .
\end{equation}
Moreover,
as ${\tilde{\psi} \circ \epsilon \colon \F_X \rightarrow \mathcal{F}}$
is a quasi-isomorphism of complexes of sheaves,
we have a quasi-isomorphism
\begin{equation*}
  G
  =
  (\blacktriangle \circ f)_* \mathcal{F}
  \simeq
  R (\blacktriangle \circ f)_* \F_X
  \simeq
  R \blacktriangle_* R f_* \F_X
\end{equation*}
and hence
\begin{equation}
  \label{eq:SellaDerivedLevelSets}
  h_{\gamma} (G) \cong h_{\gamma} (R \blacktriangle_* R f_* \F_X)
  \cong h_{\R} (R f_* \F_X)
  .
\end{equation}
Combining \eqref{eq:RISCtoSellaPushforward}
and \eqref{eq:SellaDerivedLevelSets}
we obtain the natural isomorphism
\[\zeta_f \colon
  h (f) \xrightarrow{\cong} h_{\R} (R f_* \F_X)
  ,
\]
which is natural in $f \colon X \rightarrow \R$.
In summary we obtain the following.

\begin{prp}
  \label{prp:h0Iso}
  There is a natural isomorphism $\zeta$
  as in the diagram
  \begin{equation*}
    \begin{tikzcd}[row sep=12ex, column sep=7ex]
      (\mathrm{lcContr} / \R)^{\circ}
      \arrow[d , "R (-)_* \F_{(-)}"']
      \arrow[dr, "h"{name=h0}]
      \\
      D^+ (\R)
      \arrow[r, "h_{\R}"']
      \arrow[to=h0, Leftarrow, "\zeta", shorten >=1.5ex, shorten <=1.5ex]
      &
      \VectF^{\M^{\circ}}
      .
    \end{tikzcd}
  \end{equation*}
\end{prp}

Now in \cite[Definition 2.2]{2021arXiv210809298B}
we have also defined the following notion.

\begin{dfn}
  \label{dfn:FtameFunctions}
  We say that a function $f \colon X \rightarrow \R$
  is \emph{$\F$-tame} if the singular cohomology
  $H^n (f^{-1}(I); \F)$
  is finite-dimensional
  for all open intervals $I \subseteq \R$ and any integer $n \in \Z$.
  Moreover,
  we denote the full subcategory of all $\F$-tame functions
  on locally contractible spaces by
  $(\mathrm{lcContr} / \R)_t$.
\end{dfn}

Analogously we may define the following notion for objects of $D^+(\R)$.

\begin{dfn}
  \label{dfn:tameComplexOfSheavesOnIR}
  We say that an object $F$ of $D^+(\R)$ is \emph{tame}
  if $H^n(I; F)$ is finite-dimensional for any open interval $I \subseteq \R$
  and any integer $n \in \Z$.
  Moreover,
  we denote the full subcategory of tame objects in $D^+(\R)$ by $D^+_t (\R)$.
\end{dfn}

By \cref{lem:interiorCharTameness} and \cref{prp:tameJ} the functor
${h_{\R} \colon D^+(\R) \rightarrow \VectF^{\M^{\circ}}}$
restricts to a functor
${h_{\R,t} \colon D^+_t (\R) \rightarrow \mathcal{J}}$.
In conjunction with \cref{prp:h0Iso} we obtain the diagram
\begin{equation*}
  \begin{tikzcd}[row sep=12ex, column sep=7ex]
    (\mathrm{lcContr} / \R)_t^{\circ}
    \arrow[d , "R (-)_* \F_{(-)}"']
    \arrow[dr, "h"{name=h0}]
    \\
    D^+_t (\R)
    \arrow[r, "h_{\R,t}"']
    \arrow[to=h0, Leftarrow, "\zeta", shorten >=1.5ex, shorten <=1.5ex]
    &
    \mathcal{J}
    .
  \end{tikzcd}
\end{equation*}

\section{Projective Covers of $\mathcal{J}$-Presentable Functors}
\label{sec:projCover}

Let $F \colon \M^{\circ} \rightarrow \vectF$
be a pfd sequentially continuous functor.

\begin{lem}
  \label{lem:maxSub}
  Suppose $G \hookrightarrow F$
  is a proper subfunctor of $F \colon \M^{\circ} \rightarrow \vectF$.
  Then, there is a point $w \in \M$ and a commutative triangle
  of the form
  \begin{equation*}
    \begin{tikzcd}
      G
      \arrow[ r, hook]
      \arrow[dr, "0"']
      &
      F
      \arrow[d , two heads]
      \\
      &
      S_w
      .
    \end{tikzcd}
  \end{equation*}
\end{lem}

\begin{cor}
  \label{cor:maxSub}
  Any proper subfunctor of
  $F \colon \M^{\circ} \rightarrow \vectF$
  is contained point-wise
  in a maximal subfunctor of $F$.
\end{cor}

\begin{proof}[Proof of \cref{lem:maxSub}]
  Since the point-wise inclusion
  $G \hookrightarrow F$
  is proper,
  there is a point $u := (x, y) \in \M$
  such that $G(u)$ is a proper subspace of $F(u)$.
  Thus, there is a linear form
  $\alpha \in (F(u))^*$ with
  $\alpha |_{G(u)} = 0$.
  As
  $F \colon \M^{\circ} \rightarrow \vectF$
  is pfd and sequentially continuous,
  the dual covariant functor
  $F^* \colon \M \rightarrow \vectF$
  is sequentially cocontinuous,
  i.e.
  for any increasing sequence
  $(v_k)_{k=1}^{\infty}$ in $\M$
  converging to $v \in \M$
  the natural map
  \begin{equation}
    \varprojlim_{k} F^* (v_k) \rightarrow F^* (v)
  \end{equation}
  is an isomorphism.
  We consider the restriction of
  $F^* \colon \M^{\circ} \rightarrow \vectF$
  to the maximal horizontal line segment $h \subset \M$ through $u$.
  Since $F^*$ vanishes on $l_0$,
  the linear form $\alpha$ has to die at some point $v \in h$.
  Let $\beta := F^* (u \preceq v) (\alpha)$ be the pullback
  of $\alpha$ to $F(v)$.
  As $F^* \colon \M^{\circ} \rightarrow \vectF$
  is sequentially cocontinuous
  we have $\beta \neq 0$
  (the linear form $\alpha$ dies hard so to speak).
  Now let $g \subset \M$ be the maximal vertical line segment
  through $v$.
  As $F^*$ vanishes on $l_1$,
  the linear form $\beta$ has to die at some point $w \in g$.
  Let $\gamma := F^* (v \preceq w) (\beta)$ be the pullback
  of $\beta$ to $G(w)$.
  Again, $\gamma \neq 0$
  as $F^*$ is sequentially cocontinuous.
  Now $\gamma \in F^* (v)$
  determines a point-wise surjective family of maps
  $\rho \colon F \rightarrow S_w$.
  As all internal pullbacks of $\gamma$ are trivial,
  the family of maps $\rho \colon F \rightarrow S_w$
  is a natural transformation.
\end{proof}

Now let $p \colon I \rightarrow \op{int} \M$ be a map of sets with
\begin{equation*}
  \# p^{-1} (u) = \dim_{\F} \mathrm{Nat}(F, S_u)
\end{equation*}
for any $u \in \op{int} \M$.
For $i \in I$ we now also write
$S_i := S_{p(i)} \colon \M^{\circ} \rightarrow \VectF$.
Then $\left(\mathrm{Nat}(F, S_i)\right)_{i \in I}$ is an indexed family
of non-empty sets.
Let $(\alpha_i)_{i \in I}$ be a choice function
of this indexed family of sets such that
$(\alpha_i)_{i \in p^{-1}(u)}$
is a basis of $\mathrm{Nat}(F, S_u)$
for any $u \in \op{int} \M$.
As $F \colon \M^{\circ} \rightarrow \vectF$ is pfd
we have
$\bigoplus_{i \in I} S_i = \prod_{i \in I} S_i$.
Let
\begin{equation*}
  \theta \colon F \rightarrow \bigoplus_{i \in I} S_i
\end{equation*}
be the natural transformation induced by the family
$(\alpha_i)_{i \in I}$,
then $\theta \colon F \rightarrow \bigoplus_{i \in I} S_i$
is an epimorphism since
$(\alpha_i)_{i \in p^{-1}(u)}$ is a basis of $\mathrm{Nat}(F, S_u)$
for any $u \in \op{int} \M$.
Moreover, let
\begin{equation*}
  \op{rad} F \colon \M^{\circ} \rightarrow \vectF
\end{equation*}
be the point-wise intersection
of all maximal subfunctors of $F \colon \M^{\circ} \rightarrow \vectF$.

\begin{lem}
  \label{lem:radical}
  We have the point-wise inclusion
  $\ker \theta \hookrightarrow \op{rad} F$
  of subfunctors.
\end{lem}

\begin{proof}
  By \cref{lem:maxSub} we know
  that any maximal subfunctor $G \hookrightarrow F$ of
  $F \colon \M^{\circ} \rightarrow \vectF$
  yields a short exact sequence
  \begin{equation}
    \label{eq:maxSimple}
    \begin{tikzcd}
      0
      \arrow[r]
      &
      G
      \arrow[r, hook]
      &
      F
      \arrow[r, "\alpha", two heads]
      &
      S_u
      \arrow[r]
      &
      0
    \end{tikzcd}
  \end{equation}
  for some $u \in \op{int} \M$.
  Since
  $(\alpha_i)_{i \in p^{-1}(u)}$
  is a basis of $\mathrm{Nat}(F, S_u)$
  there is a linear combination
  \begin{equation*}
    \alpha = \sum_{i \in p^{-1}(u)} \lambda_i \alpha_i
    .
  \end{equation*}
  Now let $\eta \colon S_u \rightarrow S_u^{\# p^{-1}(u)}$
  be the natural transformation induced by the family
  $\left(\lambda_i \op{id}_{S_u}\right)_{i \in p^{-1}(u)}$,
  then we have the commutative diagram
  \begin{equation*}
    \begin{tikzcd}[row sep=6ex]
      G
      \arrow[r, hook]
      &
      F
      \arrow[r, "\alpha", two heads]
      \arrow[d, "\theta"', two heads]
      &
      S_u
      \arrow[d, "\eta", tail]
      \\
      &
      \bigoplus_{i \in I} S_i
      &
      S_u^{\# p^{-1}(u)}
      \arrow[l, tail]
      .
    \end{tikzcd}
  \end{equation*}
  Since $\eta \colon S_u \rightarrow S_u^{\# p^{-1}(u)}$
  and the horizontal arrow at the bottom
  are both natural monomorphisms
  the kernel of
  $\theta \colon F \rightarrow \bigoplus_{i \in I} S_i$
  is contained in the kernel of
  $\alpha \colon F \rightarrow S_u$,
  which is $G \colon \M^{\circ} \rightarrow \vectF$
  by the exactness of \eqref{eq:maxSimple}.
\end{proof}

\begin{cor}
  \label{cor:essential}
  The epimorphism
  $\theta \colon F \rightarrow \bigoplus_{i \in I} S_i$
  is essential.
\end{cor}

The proof of this corollary is almost identical
to the proof of a similar statement in
\cite[Lemma 3.3.(2)]{MR3431480}.

\begin{proof}
  Let $G \hookrightarrow F$ be a subfunctor of
  $F \colon \M^{\circ} \rightarrow \vectF$ with
  $F = G + \ker \theta$.
  If $G \neq F$,
  then there is a maximal subfunctor
  $G' \hookrightarrow F$ of $F$
  containing $G \colon \M^{\circ} \rightarrow \vectF$
  point-wise by \cref{cor:maxSub}.
  Thus we have
  \[F = G + \ker \theta \hookrightarrow G + \op{rad} F \hookrightarrow G'\]
  by \cref{lem:radical}.
  This is a contradiction and therefore $F = G$.
  It follows that ${\theta \colon F \rightarrow \bigoplus_{i \in I} S_i}$
  is essential.
\end{proof}

Recall that $\mathcal{C}$ is the category of functors
$\M^{\circ} \rightarrow \VectF$
vanishing on $\partial \M$.
By
\cite[Corollary 3.6]{2021arXiv210809298B}
any functor in $\mathcal{J}$
is projective in $\mathcal{C}$.
Thus,
we may view $\mathcal{J}$
as a full replete additive subcategory of projectives in $\mathcal{C}$.
In \cref{sec:presentable} we provide some auxiliary results
in this particular context
that we will need here.
Now suppose
$F \colon \M^{\circ} \rightarrow \VectF$
is some $\mathcal{J}$-presentable functor
in the sense of \cref{dfn:presentable}.
This implies in particular,
that $F \colon \M^{\circ} \rightarrow \VectF$
is pfd, sequentially continuous by the Mittag-Leffler condition,
has bounded above support,
and vanishes on $\partial \M$.
So above constructions apply to $F \colon \M^{\circ} \rightarrow \VectF$.
Now at this point,
we haven't shown \cref{thm:frobenius} yet.
So we may not use the Betti functions of \cref{dfn:betti} yet.
However,
as the $\Ext_{\mathcal{C}}^0$-functor
is naturally isomorphic to the $\mathrm{Nat}$-functor
and since we know that $F \colon \M^{\circ} \rightarrow \VectF$ is pfd,
we can use the $0$-th Betti function
$\beta^0 (F) \colon \op{int} \M \rightarrow \N_0$.
Moreover,
by our choice of $p \colon I \rightarrow \op{int} \M$ above we have
\begin{equation}
  \label{eq:bettiIndexCount}
  \beta^0 (F) (u) = \dim_{\F} \mathrm{Nat}(F, S_u) = \# p^{-1} (u) 
\end{equation}
for any $u \in \op{int} \M$.

\begin{lem}
  \label{lem:admissibleBetti}
  For any $u \in \op{int} \M$ we have
  \begin{equation*}
    \sum_{v \in (\uparrow u) \cap \op{int} (\downarrow T(u))}
    \beta^0 (F) (v) < \infty
    .
  \end{equation*}
\end{lem}

\begin{proof}
  Since ${F \colon \M^{\circ} \rightarrow \vectF}$ is $\mathcal{J}$-presentable,
  there is some epimorphism
  ${\psi \colon P \rightarrow F}$
  with $P$ a functor in $\mathcal{J}$.
  Moreover,
  as ${\mathrm{Nat}(-, S_u)}$ is a left-exact functor
  for any ${u \in \op{int} \M}$
  we have the point-wise inequality
  ${\beta^0(F) \leq \beta^0(P)}$.
  Thus, it suffices to show
  \begin{equation*}
    \sum_{v \in (\uparrow u) \cap \op{int} (\downarrow T(u))}
    \beta^0 (P) (v) < \infty
  \end{equation*}
  for any $u \in \op{int} \M$.
  By \cite[Theorem 3.5]{2021arXiv210809298B}
  the cohomological functor
  $P \colon \M^{\circ} \rightarrow \vectF$
  is naturally isomorphic
  to a direct sum of contravariant blocks
  $B_v$, $v \in \op{int} \M$.
  Moreover, we have
  \begin{equation*}
    \dim_{\F} \mathrm{Nat} (B_u, S_v) =
    \dim_{\F} \mathrm{Nat} S_v(u) =
    \begin{cases}
      1 & u = v
      \\
      0 & u \neq v
    \end{cases}
  \end{equation*}
  by the Yoneda \cref{lem:yoneda},
  hence $\beta^0 (P) (v)$ is the multiplicity of $B_v$ in $P$.
  Now let $u \in \op{int} \M$,
  then we have
  \begin{equation*}
    \dim_{\F} B_v (u) =
    \begin{cases}
      1 & v \in (\uparrow u) \cap \op{int} (\downarrow T(u))
      \\
      0 & \text{otherwise}
    \end{cases}
  \end{equation*}
  and thus,
  \begin{equation*}
    \sum_{v \in (\uparrow u) \cap \op{int} (\downarrow T(u))}
    \beta^0 (P) (v) = \dim_{\F} P(u) < \infty
  \end{equation*}
  since ${P \colon \M^{\circ} \rightarrow \vectF}$ is pfd.
\end{proof}

\begin{dfn}
  \label{dfn:admissibleBetti}
  We say that a function
  $b \colon \op{int} \M \rightarrow \N_0$
  is an \emph{admissible Betti function}
  if
  its support is bounded above and
  \begin{equation*}
    \sum_{v \in (\uparrow u) \cap \op{int} (\downarrow T(u))}
    b (v) < \infty
  \end{equation*}
  for all $u \in \op{int} \M$.
  We denote the commutative monoid of admissible Betti functions by
  $\mathbb{B}$.
\end{dfn}

In particular
$\beta^0 (F) \colon \op{int} \M \rightarrow \N_0$
is an admissible Betti function
by \cref{lem:admissibleBetti}.
We will see in \cref{cor:admissibleBetti} below
that all higher Betti functions are admissible as well.

\begin{lem}
  \label{lem:sumInJ}
  For an admissible Betti function
  ${b \colon \op{int} \M \rightarrow \N_0}$
  the direct sum
  \[
    \bigoplus_{v \in \op{int} \M} B_v^{b(v)} \colon
    \M^{\circ} \rightarrow \VectF
  \]
  is a functor in $\mathcal{J}$.
\end{lem}

\begin{proof}
  As each contravariant block is cohomological the direct sum
  $\bigoplus_{v \in \op{int} \M} B_v^{b(v)}$ is a cohomological functor as well.
  Moreover, any upper bound for the support of
  $b \colon \op{int} \M \rightarrow \N_0$
  is an upper bound for the support of
  $\bigoplus_{v \in \op{int} \M} B_v^{b(v)}$ as well.
  Next we show the direct sum
  $\bigoplus_{v \in \op{int} \M} B_v^{b(v)}$ is pfd.
  To this end, let $u \in \op{int} \M$.
  Since
  \begin{equation*}
    \dim_{\F} B_v (u) =
    \begin{cases}
      1 & v \in (\uparrow u) \cap \op{int} (\downarrow T(u))
      \\
      0 & \text{otherwise}
    \end{cases}
  \end{equation*}
  for any $v \in \op{int} \M$,
  we have
  \begin{equation*}
    \dim_{\F}
    \bigoplus_{v \in \op{int} \M} B_v^{b(v)} (u)
    = \sum_{v \in (\uparrow u) \cap \op{int} (\downarrow T(u))} b(v)
    <
    \infty
    .
  \end{equation*}
  Now that
  $\bigoplus_{v \in \op{int} \M} B_v^{b(v)}$ is pfd
  we have
  $\bigoplus_{v \in \op{int} \M} B_v^{b(v)} = \prod_{v \in \op{int} \M} B_v^{b(v)}$.
  With this sequential continuity follows from the commutativity of limits.
\end{proof}

Writing $B_i := B_{p(i)} \colon \M^{\circ} \rightarrow \VectF$
for any $i \in I$,
we obtain the following corollary.

\begin{cor}
  \label{cor:sumInJ}
  The direct sum
  $\bigoplus_{i \in I} B_i \colon \M^{\circ} \rightarrow \VectF$
  is a functor in $\mathcal{J}$.  
\end{cor}

\begin{proof}
  By \cref{lem:admissibleBetti} the function
  $\beta^0 (F) \colon \op{int} \M \rightarrow \N_0$
  is an admissible Betti function
  and
  $\# p^{-1} (v) = \beta^0 (F)(v)$
  for any $v \in \op{int} \M$ by \eqref{eq:bettiIndexCount}.
  Thus,
  $\bigoplus_{i \in I} B_i \colon \M^{\circ} \rightarrow \VectF$
  is a functor in $\mathcal{J}$
  by \cref{lem:sumInJ}
\end{proof}

We now construct a projective cover
$\varphi \colon \bigoplus_{i \in I} B_i \rightarrow F$.
To this end,
let
$\upsilon_i \colon B_i \rightarrow S_i$
be the unique natural transformation sending
$1 \in \F = B_i (p(i))$ to $1 \in \F = S_i (p(i))$
provided by the Yoneda \cref{lem:yoneda}.
As $\bigoplus_{i \in I} B_i$ is projective
in $\mathcal{C}$
by the Yoneda \cref{lem:yoneda},
there is a \enquote{section}
$\varphi \colon \bigoplus_{i \in I} B_i \rightarrow F$
as in the commutative triangle
\begin{equation}
  \label{eq:sectionProjective}
  \begin{tikzcd}[row sep=6ex, column sep=6ex]
    &
    \bigoplus_{i \in I} B_i
    \arrow[ld, "\varphi"']
    \arrow[d, "\bigoplus_{i \in I} \upsilon_i", two heads]
    \\
    F
    \arrow[r, "\theta"', two heads]
    &
    \bigoplus_{i \in I} S_i
    .
  \end{tikzcd}
\end{equation}

\begin{lem}
  \label{lem:phiEpi}
  The natural transformation
  $\varphi \colon \bigoplus_{i \in I} B_i \rightarrow F$
  is a natural epimorphism.
\end{lem}

\begin{proof}
  Considering the commutative triangle \eqref{eq:sectionProjective}
  and the fact that
  $\theta \colon F \rightarrow \bigoplus_{i \in I} S_i$
  is an essential epimorphism by \cref{cor:essential},
  we see that $\varphi \colon \bigoplus_{i \in I} B_i \rightarrow F$
  is an epimorphism as well.
\end{proof}

\begin{cor}
  \label{cor:essential2}
  The natural epimorphism
  $\bigoplus_{i \in I} \upsilon_i \colon
  \bigoplus_{i \in I} B_i \rightarrow
  \bigoplus_{i \in I} S_i$
  is essential.
\end{cor}

\begin{proof}
  Above we defined
  $\theta \colon F \rightarrow \bigoplus_{i \in I} S_i$
  to be the natural transformation induced by the choice function
  $(\alpha_i)_{i \in I}$ of the indexed family
  $(\mathrm{Nat}(F, S_i))_{i \in I}$.
  As $\bigoplus_{i \in I} B_i$ is pfd by \cref{cor:sumInJ},
  we may consider above constructions in the special case where
  $F = \bigoplus_{i \in I} B_i$.
  Moreover, since $\bigoplus_{i \in I} B_i$ is pfd we also have
  $\bigoplus_{i \in I} B_i = \prod_{i \in I} B_i$
  so we may define \[\alpha_i := \op{pr}_i \circ \upsilon_i\] for $i \in I$,
  where
  $\op{pr}_i \colon \bigoplus_{i' \in I} B_{i'} \rightarrow B_i$
  is the projection to the $i$-th summand.
  For this particular choice for the family $(\alpha_i)_{i \in I}$
  we obtain
  $\theta = \bigoplus_{i \in I} \upsilon_i$,
  hence $\bigoplus_{i \in I} \upsilon_i$ is an essential epimorphism
  by \cref{cor:essential}.
\end{proof}

\begin{prp}
  \label{prp:projCover}
  The natural transformation
  $\varphi \colon \bigoplus_{i \in I} B_i \rightarrow F$
  is a projective cover of
  $F \colon \M^{\circ} \rightarrow \vectF$
  in the category
  of $\mathcal{J}$-presentable functors
  $\mathrm{pres}(\mathcal{J})$.
\end{prp}

\begin{proof}
  We consider the commutative triangle
  \eqref{eq:sectionProjective}.
  By \cref{lem:phiEpi} the natural transformation
  $\varphi \colon \bigoplus_{i \in I} B_i \rightarrow F$
  is an epimorphism.
  Moreover,
  $\bigoplus_{i \in I} \upsilon_i \colon
  \bigoplus_{i \in I} B_i \rightarrow
  \bigoplus_{i \in I} S_i$
  is essential by \cref{cor:essential2},
  hence
  $\varphi \colon \bigoplus_{i \in I} B_i \rightarrow F$
  is essential by \cite[Lemma 3.1]{MR3431480}.
  Furthermore,
  the direct sum
  $\bigoplus_{i \in I} B_i \colon \M^{\circ} \rightarrow \VectF$
  is projective
  in $\mathcal{C}$
  by the Yoneda \cref{lem:yoneda}
  and thus it is also projective in $\mathrm{pres}(\mathcal{J})$
  by Corollaries \ref{cor:sumInJ} and \ref{cor:presProj}.
\end{proof}

In particular the $0$-th Betti function
$\beta^0 (F) \colon \op{int} \M \rightarrow \N_0$
determines the isomorphism class of the domains
of projective covers of the $\mathcal{J}$-presentable functor
$F \colon \op{int} \M \rightarrow \vectF$.

\begin{cor}
  \label{cor:decompProj}
  If $F \colon \M^{\circ} \rightarrow \VectF$
  is projective in $\mathrm{pres}(\mathcal{J})$,
  then $\varphi \colon \bigoplus_{i \in I} B_i \rightarrow F$
  is a natural isomorphism.
\end{cor}

\begin{proof}
  As $F \colon \M^{\circ} \rightarrow \VectF$
  is projective,
  the identity natural transformation
  $\op{id}_F \colon F \rightarrow F$
  is a projective cover as well.
  By the uniqueness of projective covers
  \cite[Corollary 3.5]{MR3431480}
  the natural transformation
  $\varphi \colon \bigoplus_{i \in I} B_i \rightarrow F$
  is a natural isomorphism.
\end{proof}

\begin{cor}
  \label{cor:decompJ}
  If $F \colon \M^{\circ} \rightarrow \VectF$
  is a functor in $\mathcal{J}$,
  then $\varphi \colon \bigoplus_{i \in I} B_i \rightarrow F$
  is a natural isomorphism.
\end{cor}

\begin{proof}
  By \cite[Corollary 3.6]{2021arXiv210809298B}
  and \cref{cor:presProj}
  the functor
  $F \colon \M^{\circ} \rightarrow \VectF$
  is projective itself,
  and thus the result follows with \cref{cor:decompProj}.
\end{proof}

\begin{cor}
  \label{cor:JiffProj}
  The additive category $\mathcal{J}$
  is the subcategory of projectives in $\mathrm{pres}(\mathcal{J})$.
\end{cor}

\begin{proof}
  This follows from \cref{prp:projCover}, \cref{cor:sumInJ},
  and \cref{cor:Jproj}.
  Alternatively,
  we may also conclude this
  from Corollaries \ref{cor:decompProj} and \ref{cor:decompJ}.
\end{proof}

\section{
  Equivalence of $D^+_t (q_{\gamma}, \partial q)$
  and $\mathcal{J}$
}
\label{sec:equiv}

Let $p \colon I \rightarrow \op{int} \M$
be some map of sets
such that the assignment
\begin{equation*}
  \op{int} \M \rightarrow \N_0, \,
  u \mapsto \# p^{-1}(u)
\end{equation*}
is an admissible Betti function.
Again we write
$B_i := B_{p(i)} \colon \M^{\circ} \rightarrow \vectF$
for any $i \in I$.

\begin{prp}
  \label{prp:dirSumProd}
  The direct sum
  $\bigoplus_{i \in I} (\iota \circ p)(i)$
  together with the projections
  \[\op{pr}_j \colon
    \bigoplus_{i \in I} (\iota \circ p)(i) \rightarrow (\iota \circ p)(j)\]
  for $j \in I$
  satisfies the universal property of the product in
  $D^+ (q_{\gamma}, \partial q)$.
\end{prp}

Our proof of this proposition is involved enough
so that we defer it to its own \cref{sec:dirSumProd}.

\begin{remark}
  If $\partial q$ is not closed in $q_{\gamma}$,
  then it may well happen that the derived product
  $R \prod_{i \in I} (\iota \circ p)(i)$,
  which is the product in $D^+ (q_{\gamma})$,
  is not in $D^+ (q_{\gamma}, \partial q)$.
  Only after applying the coreflection
  $R \flat$ to the derived product $R \prod_{i \in I} (\iota \circ p)(i)$
  we obtain the product in $D^+ (q_{\gamma}, \partial q)$,
  which is conveniently isomorphic to the direct sum
  $\bigoplus_{i \in I} (\iota \circ p)(i)$
  by \cref{prp:dirSumProd}.
\end{remark}

\begin{cor}
  \label{cor:dirSumPres}
  The direct sum
  $\bigoplus_{i \in I} (\iota \circ p)(i)$
  is preserved by
  $h_{\gamma} \colon D^+ (q_{\gamma}) \rightarrow \VectF^{\M^{\circ}}$.
\end{cor}

\begin{proof}
  As $\iota(u)$ is in $D^+ (q_{\gamma}, \partial q)$
  for any $u \in \op{int} \M$,
  the functor
  \[h_{\gamma} \colon D^+ (q_{\gamma}) \rightarrow \VectF^{\M^{\circ}}, \,
    F \mapsto \Hom_{\Der (q_{\gamma})} (\iota(-), F)
  \]
  maps the direct sum
  $\bigoplus_{i \in I} (\iota \circ p)(i)$
  to the direct product
  $\prod_{i \in I} (h_{\gamma} \circ \iota \circ p)(i) \colon
  \M^{\circ} \rightarrow \VectF$
  by \cref{prp:dirSumProd}.
  Moreover, we have
  $(h_{\gamma} \circ \iota \circ p)(i) \cong B_i$
  by \cref{cor:iota},
  hence
  \[\prod_{i \in I} (h_{\gamma} \circ \iota \circ p)(i) \cong
    \prod_{i \in I} B_i.
  \]
  Furthermore,
  the direct sum
  $\bigoplus_{i \in I} B_i$ is pfd by \cref{cor:sumInJ}
  and thus
  $\prod_{i \in I} B_i = \bigoplus_{i \in I} B_i$.
  Invoking \cref{cor:iota} once more we obtain
  \[\bigoplus_{i \in I} B_i \cong
    \bigoplus_{i \in I} (h_{\gamma} \circ \iota \circ p)(i).
    \qedhere
  \]
\end{proof}

\begin{cor}
  The direct sum
  $\bigoplus_{i \in I} (\iota \circ p)(i)$
  is tame.
\end{cor}

\begin{proof}
  By \cref{cor:dirSumPres} the direct sum
  $\bigoplus_{i \in I} (\iota \circ p)(i)$
  is mapped to
  $\bigoplus_{i \in I} (h_{\gamma} \circ \iota \circ p)(i)$.
  Moreover, we have
  $(h_{\gamma} \circ \iota \circ p)(i) \cong B_i$
  by \cref{cor:iota},
  hence
  \[\bigoplus_{i \in I} (h_{\gamma} \circ \iota \circ p)(i) \cong
    \bigoplus_{i \in I} B_i
    ,
  \]
  which is in $\mathcal{J}$ by \cref{lem:sumInJ}.
\end{proof}

Now suppose $F$ is an object of $D^+_t (q_{\gamma}, \partial q)$
and let $p \colon I \rightarrow \op{int} \M$
be some map of sets such that
\begin{equation*}
  \# p^{-1} (u) = \left(\beta^0 \circ h_{\gamma}\right)(F)(u) =
  \dim_{\F} \mathrm{Nat} (h_{\gamma}(F), S_u)
\end{equation*}
for any $u \in \op{int} \M$.
By \cref{cor:decompJ} there is a natural isomorphism
$\varphi \colon \bigoplus_{i \in I} B_i \rightarrow h_{\gamma}(F)$.
For each $j \in I$
let $\psi_j \colon (\iota \circ p)(j) \rightarrow F$
be the image of $1 \in \F = B_j (p(j))$
under the composition of linear maps
\begin{equation*}
  \F = B_j (p(j)) \longrightarrow
  \bigoplus_{i \in I} B_i (p(j)) \xrightarrow{\varphi_{p(j)}}
  h_{\gamma} (F) (p(j)) =
  \Hom_{D^+ (q_{\gamma})} ((\iota \circ p)(j), F)
  .
\end{equation*}
Now the family
$(\psi_i)_{i \in I}$
is a choice function for the indexed family
\[\left(\Hom_{D^+ (q_{\gamma})} ((\iota \circ p)(i), F)\right)_{i \in I}.\]
As such,
the family $(\psi_i)_{i \in I}$ induces a homomorphism
\[\psi \colon \bigoplus_{i \in I} (\iota \circ p)(i) \rightarrow F\]
in the derived category $D^+ (q_{\gamma})$.
Now let $G := \bigoplus_{i \in I} (\iota \circ p)(i)$,
we aim to show that $\psi \colon G \rightarrow F$ is an isomorphism.
We use the following lemma as a stepping stone towards this goal.

\begin{lem}
  \label{lem:inducedIso}
  The natural transformation
  $h_{\gamma} (\psi) \colon h_{\gamma} (G) \rightarrow h_{\gamma} (F)$
  is a natural isomorphism.
\end{lem}

\begin{proof}
  We consider the commutative diagram
  \begin{equation*}
    \begin{tikzcd}[row sep=5ex]
      \bigoplus_{i \in I} (h_\gamma \circ \iota \circ p)(i)
      \arrow[r]
      &
      h_\gamma (G)
      \arrow[d, "h_\gamma (\psi)"]
      \\
      \bigoplus_{i \in I} B_{i}
      \arrow[u]
      \arrow[r, "\varphi"']
      &
      h_\gamma (F)
    \end{tikzcd}
  \end{equation*}
  of functors ${\M^{\circ} \rightarrow \VectF}$
  and natural transformations.
  By \cref{cor:dirSumPres} the horizontal arrow at the top
  is a natural isomorphism.
  Moreover, the vertical arrow on the left
  is a natural isomorphism by \cref{cor:iota}.
  Furthermore,
  the natural transformation
  ${\varphi \colon \bigoplus_{i \in I} B_{i} \rightarrow h_{\gamma} (F)}$
  is a natural isomorphism
  by \cref{cor:decompJ}.
  As a result
  ${h_{\gamma}(\psi) \colon h_{\gamma}(G) \rightarrow h_{\gamma}(F)}$
  is an isomorphism as well.
\end{proof}

In conjunction with \cref{prp:conservative} we obtain the following.

\begin{cor}
  \label{cor:decompDerCat}
  The homomorphism
  $\psi \colon G \rightarrow F$
  is an isomorphism.
\end{cor}

\begin{thm}
  \label{thm:equiv}
  The functor
  $h_{\gamma,0,t} \colon D^+_t (q_{\gamma}, \partial q) \rightarrow \mathcal{J}$
  is an equivalence of $\F$-linear categories.
\end{thm}

\begin{proof}
  First we show that $h_{\gamma,0,t}$ is essentially surjective.
  To this end, let $F \colon \M^{\circ} \rightarrow \vectF$
  be a functor in $\mathcal{J}$.
  By \cref{cor:decompJ} the functor $F$ is naturally isomorphic
  to a direct sum
  $\bigoplus_{i \in I} B_{p(i)} \colon \M^{\circ} \rightarrow \vectF$
  for some map of sets $p \colon I \rightarrow \op{int} \M$
  with $\# p^{-1}(u) = \beta^0(F)(u)$ for any $u \in \op{int} \M$.
  Now
  ${\bigoplus_{i \in I} B_{p(i)} \cong
    \bigoplus_{i \in I} (h_{\gamma} \circ \iota \circ p)(i)}$
  by \cref{cor:iota}.
  Moreover,
  ${h_{\gamma} \left(\bigoplus_{i \in I} (\iota \circ p)(i)\right) \cong
    \bigoplus_{i \in I} (h_{\gamma} \circ \iota \circ p)(i)}$
  by \cref{cor:dirSumPres},
  hence
  ${h_{\gamma,0,t} \colon D^+_t (q_{\gamma}, \partial q) \rightarrow \mathcal{J}}$
  is essentially surjective.
  
  Next we show ${h_{\gamma,0,t}}$ is fully faithful.
  To this end,
  let $F$ and $G$ be objects of $D^+_t (q_{\gamma}, \partial q)$,
  then we have
  \begin{equation*}
    F \cong \bigoplus_{i \in I} (\iota \circ p_1)(i)
    \quad \text{and} \quad
    G \cong \bigoplus_{j \in J} (\iota \circ p_2)(j)
  \end{equation*}
  for some maps of sets
  $p_1 \colon I \rightarrow \op{int} \M$ and
  $p_2 \colon J \rightarrow \op{int} \M$
  with
  \begin{equation}
    \label{eq:sumsFandG}
    \# p_1^{-1}(u) = (\beta^0 \circ h_{\gamma})(F)(u)
    \quad \text{and} \quad
    \# p_2^{-1}(u) = (\beta^0 \circ h_{\gamma})(G)(u)
  \end{equation}
  for any ${u \in \op{int} \M}$
  by \cref{cor:decompDerCat}.
  Moreover,
  by \cref{prp:dirSumProd} the direct sum
  ${\bigoplus_{j \in J} (\iota \circ p_2)(j)}$
  satisfies the universal property
  of the corresponding product in ${D^+_t (q_{\gamma}, \partial q)}$.
  Furthermore, the functor
  $h_{\gamma} \colon D^+ (q_{\gamma}) \rightarrow \VectF^{\M^{\circ}}$
  preserves the direct sums in \eqref{eq:sumsFandG}
  by \cref{cor:dirSumPres}.
  As any pfd direct sum of functors is a direct product as well
  we obtain the commutative diagram
  \begin{equation*}
    \begin{tikzcd}[row sep=7ex]
      \Hom_{D^+(q_{\gamma})} (F, G)
      \arrow[r, "\cong"]
      \arrow[d, "h_{\gamma}"']
      &
      \prod_{i \in I, j \in J}
      \Hom_{D^+(q_{\gamma})} ((\iota \circ p)(i), (\iota \circ p)(j))
      \arrow[d, "\prod_{i \in I, j \in J} h_{\gamma}"]
      \\
      \mathrm{Nat} (h_{\gamma}(F), h_{\gamma}(G))
      \arrow[r, "\cong"]
      &
      \prod_{i \in I, j \in J}
      \mathrm{Nat}
      ((h_{\gamma} \circ \iota \circ p)(i), (h_{\gamma} \circ \iota \circ p)(j))
    \end{tikzcd}
  \end{equation*}
  with both horizontal maps isomorphisms.
  Thus,
  it suffices to show
  that the restriction of
  ${h_{\gamma} \colon D^+ (q_{\gamma}) \rightarrow \VectF^{\M^{\circ}}}$
  to the full subcategory on the image of $\iota$
  is fully faithful.
  To this end, we note that for a homomorphism
  ${\psi \colon \iota(u) \rightarrow \iota(v)}$,
  where ${u, v \in \op{int} \M}$,
  the natural transformation
  \[h_\gamma (\psi) \colon
    (h_\gamma \circ \iota)(u) \rightarrow (h_\gamma \circ \iota)(v)\]
  sends the identity
  \[\op{id}_{\iota(u)} \in (h_\gamma \circ \iota)(u)(u) =
    \Hom_{\Der (q_{\gamma})} (\iota(u), \iota(u))
  \]
  to
  \[\psi \in (h_\gamma \circ \iota)(v)(u) =
    \Hom_{\Der (q_{\gamma})} (\iota(u), \iota(v))
    .
  \]
  By \cref{cor:iota} and the Yoneda \cref{lem:yoneda}
  this describes a one-to-one correspondence between
  $\Hom_{\Der (q_{\gamma})} (\iota(u), \iota(v))$
  and $\op{Nat} ((h_\gamma \circ \iota)(u), (h_\gamma \circ \iota)(v))$.
\end{proof}

\begin{cor}
  \label{cor:equivBdClosed}
  If $\partial q$ is closed in $q_{\gamma}$,
  then the composition of functors
  \begin{equation*}
    D^+_t (\dot{q}) \xrightarrow{R i_*}
    D^+_t (q_{\gamma}) \xrightarrow{h_{\gamma,t}}
    \mathcal{J}
  \end{equation*}
  yields an equivalence of categories.
\end{cor}

\begin{proof}
  This follows in conjunction with
  \cref{cor:compsWithDirIms} and \cref{lem:vanishingTameInteriorAdj}.
\end{proof}

\subsection{Partial Faithfulness}
\label{sec:partFaithful}

To complete our proof of \cref{thm:equiv}
we still need to provide a proof of \cref{prp:dirSumProd}.
Now in order to prove \cref{prp:dirSumProd}
we will first show with \cref{prp:partFaithful} below that the functor
$h_{\gamma,0} \colon
D^+ (q_{\gamma}, \partial q) \rightarrow
\VectF^{\M^{\circ}}$
is at least partially full and partially faithful
in the sense that the family of homomorphisms
with codomain $(\iota \circ T^{-n})(t)$
is in one-to-one correspondence with corresponding
natural transformations with codomain
$(h_{\gamma} \circ \iota \circ T^{-n})(t)$
for any $n \in \Z$ and $t \in q \setminus \partial q$.
To this end,
let $t \in q \setminus \partial q$ and let $n \in \Z$.
For any $u \in T^{-n}(D)$ we define the following subsets of $q$:
\begin{align*}
  I(u) & :=
         q \cap \op{int} \left(\downarrow T^{n+1}(u)\right) =
         q \setminus (\rho_1 \circ T^n)(u)
         ,
  \\
  C(u) & :=
         q \setminus \left(\uparrow T^n(u)\right) =
         q \setminus (\rho_0 \circ T^n)(u)
         ,
  \\
  \text{and} \quad
  Z(u) & :=
         I(u) \setminus C(u) =
         (\rho_0 \circ T^n)(u) \setminus (\rho_1 \circ T^n)(u) 
         .
\end{align*}
Moreover, we define the functors
\begin{align*}
  \F_I [-n] & \colon T^{-n} (D) \rightarrow C^+ (q_{\gamma}), \,
              u \mapsto \F_{I(u)} [-n]
              ,
  \\
  \F_C [-n] & \colon T^{-n} (D) \rightarrow C^+ (q_{\gamma}), \,
              u \mapsto \F_{C(u)} [-n]
              ,
  \\
  \text{and} \quad
  \F_Z [-n] & \colon T^{-n} (D) \rightarrow C^+ (q_{\gamma}), \,
              u \mapsto \F_{Z(u)} [-n]
              ,
\end{align*}
where $C^+ (q_{\gamma}) := C^+ (\mathrm{Sh}(q_{\gamma}))$
is the category of bounded below cochain complexes of sheaves on $q_{\gamma}$.
We note that
\begin{equation}
  \label{eq:iotaRestr}
  \F_Z [-n] = \iota |_{T^{-n}(D)}
  .
\end{equation}
By \cite[Proposition 2.3.6.(v)]{Kashiwara1990}
we have the short exact sequence
\begin{equation*}
  0 \rightarrow
  \F_C [-n] \rightarrow
  \F_I [-n] \rightarrow
  \F_Z [-n] \rightarrow
  0
\end{equation*}
of functors $T^{-n}(D) \rightarrow C^+ (q_{\gamma})$.
Thus, we obtain the sequence
\begin{equation}
  \label{eq:pwTriangle}
  \F_C [-n] \rightarrow
  \F_I [-n] \rightarrow
  \F_Z [-n] \rightarrow
  \F_C [1-n]
\end{equation}
of functors ${T^{-n}(D) \rightarrow D^+ (q_{\gamma})}$,
which is point-wise a distinguished triangle in ${D^+ (q_{\gamma})}$,
see also \mbox{\cite[Equation (2.6.33)]{Kashiwara1990}}.
Now let $q_0$ be the unique point of intersection of $q$ and $l_0$.
Similarly let $q_1$ be the unique point of intersection of $q$ and $l_1$.
We note that $t \in q$ lies on the vertical line through $q_0$
iff $q_0 \preceq t$.
Similarly $t$ lies on the horizontal line through $q_1$
iff $q_1 \preceq t$.
For $u \in T^{-n}(D)$ we write $C_0(u) \subseteq C(u)$
for the connected component of $C(u)$ containing $q_0$ if it exists
and otherwise we set $C_0(u) = \emptyset$.
We define $C_1(u) \subseteq C(u)$ analogously as well as the functors
\begin{align*}
  \\
  \F_{C_0} [-n] & \colon T^{-n} (D) \rightarrow C^+ (q_{\gamma}), \,
                  u \mapsto \F_{C_0 (u)} [-n]
  \\
  \text{and} \quad
  \F_{C_1} [-n] & \colon T^{-n} (D) \rightarrow C^+ (q_{\gamma}), \,
                  u \mapsto \F_{C_1 (u)} [-n]
                  .
\end{align*}
Now let $m = n, n-1$ and let
\[P 
  \colon
  (T^{-n}(D))^{\circ} \rightarrow \vectF, \,
  u \mapsto
  \Hom_{D^+(q_{\gamma})}\left(\F_{I(u)} [-n], (\iota \circ T^{-n})(t)\right)
  .
\]
Then $P$ is isomorphic to the functor
\begin{equation*}
  P' \colon (T^{-n}(D))^{\circ} \rightarrow \mathrm{vect}_{\F},
  u \mapsto
  \begin{cases}
    \F & u \in \op{int} \left(\uparrow T^{-n-1}(t)\right)
    \\
    \{0\} & \text{otherwise},
  \end{cases}
\end{equation*}
whose internal maps
are identities whenever both domain and codomain are $\F$
and zero otherwise.
For any functor
$F \colon (T^{-n}(D))^{\circ} \rightarrow \VectF$
we have
\[\mathrm{Nat} (F, P') \cong
  \left(
    \varinjlim_{u \in \op{int} \left(\uparrow T^{-n-1}(t)\right)} F(u)
  \right)^*.
\]
As direct limits are exact as well as the dual space functor,
the functor ${\mathrm{Nat} (-, P') \cong \mathrm{Nat} (-, P)}$
is exact as well.
Now suppose $F$ is an object of $D^+ (q_{\gamma})$.
In the following we use 
$\Hom
(\F_{C_0} [-m], F)$
as a short hand for the functor
\begin{equation*}
  (T^{-n}(D))^{\circ} \rightarrow \VectF, \,
  u \mapsto
  \Hom_{D^+(q_{\gamma})}\left(\F_{C_0(u)} [-m], F\right)  
\end{equation*}
and similarly
for $\F_{C_1} [-m]$, $\F_{I} [-n]$, $\F_{C} [-n]$, or $\F_{Z} [-n]$
in place of $\F_{C_0} [-m]$.

\begin{lem}
  \label{lem:boundaryStalks}
  We have
  \begin{equation*}
    \mathrm{Nat}
    \left(
      \Hom
      (\F_{C_0} [-m], F),
      P
    \right)
    \cong
    \begin{cases}
      \displaystyle \varinjlim_{q_0 \in U} H^m(U; F)
      &
      q_0 \preceq t
      \\
      \{0\}
      &
      \text{otherwise.}
    \end{cases}
  \end{equation*}
  Similarly we have
  \begin{equation*}
    \mathrm{Nat}
    \left(
      \Hom
      (\F_{C_1} [-m], F),
      P
    \right)
    \cong
    \begin{cases}
      \displaystyle \varinjlim_{q_1 \in U} H^m(U; F)
      &
      q_1 \preceq t
      \\
      \{0\}
      &
      \text{otherwise.}
    \end{cases}
  \end{equation*}  
\end{lem}

\begin{cor}
  \label{cor:complVanishing}
  For any object $F$ of $D^+ (q_{\gamma}, \partial q)$ we have
  $\mathrm{Nat}
  \left(
    \Hom
    (\F_{C} [-m], F),
    P
  \right)
  \cong
  \{0\}
  .$
\end{cor}

\begin{proof}
  We have
  \begin{align*}
    \mathrm{Nat}
    \left(
      \Hom
      (\F_{C} [-m], F),
      P
    \right)
    & =
      \mathrm{Nat}
      \left(
      \Hom
      (\F_{C_0} [-m] \oplus \F_{C_1} [-m], F),
      P
      \right)
    \\
    & =
      \mathrm{Nat}
      \left(
      \Hom
      (\F_{C_0} [-m] , F)
      \oplus
      \Hom
      (\F_{C_1} [-m] , F),
      P
      \right)
    \\
    & =
      \mathrm{Nat}
      \left(
      \Hom
      (\F_{C_0} [-m] , F),
      P
      \right)
      \oplus
      \mathrm{Nat}
      \left(
      \Hom
      (\F_{C_1} [-m] , F),
      P
      \right)
    \\
    & = \{0\} \oplus \{0\}
  \end{align*}
  by \cref{lem:boundaryStalks}.
\end{proof}

\begin{lem}
  \label{lem:HomIso}
  The functor
  $\Hom (\F_I [-n], -)$ induces an isomorphism
  \begin{equation*}
    \Hom_{D^+ (q_{\gamma})} (F, (\iota \circ T^{-n})(t))
    \xrightarrow{~\cong~}
    \mathrm{Nat} (\Hom (\F_I [-n], F), P)
    .
  \end{equation*}
\end{lem}

\begin{proof}
  The set of open neighbourhoods of $t \in q_{\gamma}$ of the form $I(u)$
  for some $u \in T^{-n}(D)$ forms a neighbourhood basis for $t$.
  Thus, the induced map
  \begin{equation*}
    \varinjlim_{\substack{u \in T^{-n}(D) \\ t \in I(u)}} H^n (I(u); F)
    \xrightarrow{~\cong~}
    \varinjlim_{t \in U} H^n (U; F)
  \end{equation*}
  is an isomorphism.
  With this we consider the commutative diagram
  \begin{equation*}
    \begin{tikzcd}[column sep=16ex]
      \Hom_{D^+ (q_{\gamma})} (F, (\iota \circ T^{-n})(t))
      \arrow[r, "{\Hom (\F_I [-n], -)}"]
      \arrow[d]
      &
      \mathrm{Nat} (\Hom (\F_I [-n], F), P)
      \arrow[d]
      \\
      \left(
        \displaystyle
        \varinjlim_{t \in U} H^n (U; F)
      \right)^*
      \arrow[r, "\cong"]
      &
      \left(
        \displaystyle
        \varinjlim_{\substack{u \in T^{-n}(D) \\ t \in I(u)}} H^n (I(u); F)
      \right)^*
      .
    \end{tikzcd}
  \end{equation*}
  Here the vertical map on the left hand side
  takes any homomorphism
  $\psi \colon F \rightarrow (\iota \circ T^{-n})(t)$
  to the family of maps
  $\{H^n (U; \psi) \colon H^n(U; F) \rightarrow \F \mid t \in U\}$
  and then to the naturally induced map of type
  $\varinjlim_{t \in U} H^n (U; F) \rightarrow \F$.
  By \cref{lem:mapDualStalk} this map is an isomorphism.
  The vertical map on the right hand side
  takes any natural transformation
  $\eta \colon \Hom (\F_I [-n], F) \rightarrow P$
  to the family of maps
  $\{\eta_u \colon H^n(I(u); F) \rightarrow \F \mid t \in I(u)\}$
  and then to the naturally induced map of type
  \begin{equation*}
    \varinjlim_{\substack{u \in T^{-n}(D) \\ t \in I(u)}} H^n (I(u); F) \rightarrow \F
    .
  \end{equation*}
  As $P(u) \cong \{0\}$ for all $u \in T^{-n}(D)$ with $t \notin I(u)$,
  this vertical map on the right hand side is an isomorphism.
  As a result,
  we obtain that the horizontal map at the top
  is an isomorphism as well.
\end{proof}

Now let
\[B := \Hom(\F_Z [-n], (\iota \circ T^{-n})(t)) \colon
  (T^{-n}(D))^{\circ} \rightarrow \vectF,
\]
then we have
\[B =
  (h_{\gamma} \circ \iota \circ T^{-n})(t) |_{T^{-n}(D)} \cong
  B_{T^{-n}(t)} |_{T^{-n}(D)}
\]
by \eqref{eq:iotaRestr} and \cref{cor:iota}.
We consider
\[\mathrm{Nat} (-, P) =
  \mathrm{Nat} (-, \Hom(\F_I [-n], (\iota \circ T^{-n})(t)))\]
as a functor from the category of functors
$(T^{-n}(D))^{\circ} \rightarrow \VectF$
vanishing on $\partial \M \cap T^{-n}(D)$
to the category of vector spaces over $\F$.

\begin{lem}
  \label{lem:univElem}
  The natural transformation
  \[B = \Hom(\F_Z [-n], (\iota \circ T^{-n})(t)) \rightarrow
    \Hom(\F_I [-n], (\iota \circ T^{-n})(t)) = P\]
  is a universal element of the functor
  $\mathrm{Nat} (-, P)$.
  In other words,
  for any natural transformation
  $\eta \colon F \rightarrow P$
  with $F$ vanishing on $\partial \M \cap T^{-n}(D)$
  there is a unique natural transformation
  $\eta' \colon F \rightarrow B$ such that the diagram
  \begin{equation*}
    \begin{tikzcd}
      &
      B
      \arrow[d]
      \\
      F
      \arrow[ru, "\eta'"]
      \arrow[r , "\eta"']
      &
      P
    \end{tikzcd}
  \end{equation*}
  commutes.
\end{lem}

\begin{proof}
  As $P \cong P'$ we have the short exact sequence
  \begin{equation*}
    0 \rightarrow
    B \rightarrow
    P \rightarrow
    \op{Ran}_{\partial \M} P |_{\partial \M} \rightarrow
    0,
  \end{equation*}
  where
  $\op{Ran}_{\partial \M} P |_{\partial \M}$
  is the right Kan extension of the restriction of $P$
  to $\partial \M \cap T^{-n}(D)$ along the inclusion
  $\partial \M \cap T^{-n}(D) \hookrightarrow T^{-n}(D)$.
  As $\op{Ran}_{\partial \M} F |_{\partial \M} \cong 0$
  for any functor $F \colon (T^{-n}(D))^{\circ} \rightarrow \VectF$
  vanishing on $\partial \M \cap T^{-n}(D)$
  the result follows.
\end{proof}

\begin{lem}
  \label{lem:restrIso}
  For any functor
  $F \colon \M^{\circ} \rightarrow \VectF$
  the restriction map
  \begin{equation*}
    \mathrm{Nat} (F, B_{T^{-n}(t)}) \cong
    \mathrm{Nat} (F, (h_{\gamma} \circ \iota \circ T^{-n})(t))
    \rightarrow
    \mathrm{Nat} (F |_{T^{-n}(D)}, B)
  \end{equation*}
  is an isomorphism.
\end{lem}

\begin{prp}
  \label{prp:partFaithful}
  For any object $F$ of $D^+ (q_{\gamma}, \partial q)$
  the natural map
  \begin{equation*}
    \Hom_{D^+ (q_{\gamma})} (F, (\iota \circ T^{-n})(t))
    \xrightarrow{~h_{\gamma}~}
    \mathrm{Nat}(h_{\gamma}(F), (h_{\gamma} \circ \iota \circ T^{-n})(t))
  \end{equation*}
  is an isomorphism.
\end{prp}

\begin{proof}
  We consider the commutative diagram
  \begin{equation*}
    \begin{tikzcd}[row sep=7ex]
      \Hom_{D^+ (q_{\gamma})} (F, (\iota \circ T^{-n})(t))
      \arrow[ d, "h_{\gamma}"']
      \arrow[rd, "{\Hom(\F_I [-n], -)}"]
      \\
      \mathrm{Nat}(h_{\gamma}(F), (h_{\gamma} \circ \iota \circ T^{-n})(t))
      \arrow[d]
      &
      \mathrm{Nat}(\Hom(\F_I [-n], F), P)
      \arrow[d]
      \\
      \mathrm{Nat}(\Hom(\F_Z [-n], F), B)
      \arrow[r]
      &
      \mathrm{Nat}(\Hom(\F_I [-n], F), P)
      .
    \end{tikzcd}
  \end{equation*}
  By \cref{lem:restrIso}
  the vertical map on the lower left hand side
  is an isomorphism.
  Moreover,
  by \cref{lem:univElem}
  the horizontal map at the bottom is an isomorphism.
  Furthermore,
  the diagonal map on the upper right hand side
  induced by the functor ${\Hom(\F_I [-n], -)}$
  is an isomorphism by \cref{lem:HomIso}.
  Thus,
  it suffices to show that the vertical map on the right hand side
  \begin{equation*}
    \mathrm{Nat}(\Hom(\F_I [-n], F), P)
    \rightarrow
    \mathrm{Nat}(\Hom(\F_I [-n], F), P)
  \end{equation*}
  is an isomorphism.
  To this end,
  we post-compose the point-wise distinguished triangle \eqref{eq:pwTriangle}
  of functors
  with the cohomological functor
  $\Hom_{D^+(q_{\gamma})} (-, F)$
  to obtain the exact sequence
  \begin{equation*}
    \Hom(\F_C [1-n], F) \rightarrow
    \Hom(\F_Z [ -n], F) \rightarrow
    \Hom(\F_I [ -n], F) \rightarrow
    \Hom(\F_C [ -n], F)
  \end{equation*}
  of functors $(T^{-n}(D))^{\circ} \rightarrow \VectF$.
  Now we apply the exact functor $\mathrm{Nat} (-, P)$
  to this exact sequence to obtain the exact sequence
  \begin{equation}
    \label{eq:NatHomExact}
    \begin{tikzcd}
      \mathrm{Nat} (\Hom(\F_C [ -n], F), P)
      \arrow[d]
      \\
      \mathrm{Nat} (\Hom(\F_I [ -n], F), P)
      \arrow[d]
      \\
      \mathrm{Nat} (\Hom(\F_Z [ -n], F), P)
      \arrow[d]
      \\
      \mathrm{Nat} (\Hom(\F_C [1-n], F), P)
    \end{tikzcd}
  \end{equation}
  in turn.
  By \cref{cor:complVanishing} we have
  \begin{equation*}
    \mathrm{Nat} (\Hom(\F_C [ -n], F), P) \cong
    \{0\} \cong
    \mathrm{Nat} (\Hom(\F_C [1-n], F), P)
  \end{equation*}
  and thus the vertical map in the center of \eqref{eq:NatHomExact}
  has to be an isomorphism.
\end{proof}

\subsection{Products Vanishing on $\partial q$}
\label{sec:dirSumProd}

In the following we provide a proof of \cref{prp:dirSumProd}
by considering the induced maps on cohomology stalks
of the direct sum and the coreflection of the derived product
into $D^+ (q_{\gamma}, \partial q)$.
To this end,
let $t \in q \setminus \partial q$ and let $n \in \Z$
as in \cref{sec:partFaithful}.

\begin{lem}
  \label{lem:dirSumFlip}
  For a family of objects
  $\{F_i \mid i \in I\}$ in $D^+ (q_{\gamma}, \partial q)$
  the naturally induced map
  \begin{equation*}
    \mathrm{Nat} \left(
      h_{\gamma} \left(
        \bigoplus_{i \in I} F_i
      \right),
      (h_{\gamma} \circ \iota \circ T^{-n})(t)
    \right)
    \longrightarrow
    \mathrm{Nat} \left(
      \bigoplus_{i \in I} 
      h_{\gamma} \left(
        F_i
      \right),
      (h_{\gamma} \circ \iota \circ T^{-n})(t)
    \right)    
  \end{equation*}
  is an isomorphism.
\end{lem}

\begin{proof}
  We consider the commutative diagram
  \begin{equation*}
    \begin{tikzcd}[column sep=9ex]
      \Hom_{D^+(q_{\gamma})} (\bigoplus_{i \in I} F_i, (\iota \circ T^{-n})(t))
      \arrow[r, "h_{\gamma}"]
      \arrow[dd]
      &
      \mathrm{Nat} \left(
        h_{\gamma} \left(
          \bigoplus_{i \in I} F_i
        \right),
        (h_{\gamma} \circ \iota \circ T^{-n})(t)
      \right)
      \arrow[d]
      \\
      &
      \mathrm{Nat} \left(
        \bigoplus_{i \in I} 
        h_{\gamma} \left(
          F_i
        \right),
        (h_{\gamma} \circ \iota \circ T^{-n})(t)
      \right)
      \arrow[d]
      \\
      \prod_{i \in I}
      \Hom_{D^+(q_{\gamma})} (F_i, (\iota \circ T^{-n})(t))
      \arrow[r, "\prod_{i \in I} h_{\gamma}"']
      &
      \prod_{i \in I}
      \mathrm{Nat} \left(
        h_{\gamma} \left(
          F_i
        \right),
        (h_{\gamma} \circ \iota \circ T^{-n})(t)
      \right)
      .
    \end{tikzcd}
  \end{equation*}
  By \cref{prp:partFaithful} the two horizontal maps are isomorphisms.
  Moreover,
  the vertical map on the left hand side
  and the vertical map on the lower right hand side
  are both isomorphisms by the universal property of the direct sum.
  Thus,
  the vertical map on the upper right hand side is an isomorphism as well.
\end{proof}

Now let $p \colon I \rightarrow \op{int} \M$
be some map of sets
such that the assignment
\begin{equation*}
  \op{int} \M \rightarrow \N_0, \,
  u \mapsto \# p^{-1}(u)
\end{equation*}
is an admissible Betti function.
Again we write
$B_i := B_{p(i)} \colon \M^{\circ} \rightarrow \vectF$
for any $i \in I$.

\begin{proof}[Proof of \cref{prp:dirSumProd}]
  Let
  \begin{equation*}
    \kappa \colon
    \bigoplus_{i \in I} (\iota \circ p)(i) \rightarrow
    R \prod_{i \in I} (\iota \circ p)(i)
  \end{equation*}
  be the naturally induced homomorphism
  from the direct sum to the derived product.
  Then there is a unique homomorphism
  \begin{equation*}
    \kappa^{\#} \colon
    \bigoplus_{i \in I} (\iota \circ p)(i) \rightarrow
    R \flat R \prod_{i \in I} (\iota \circ p)(i)
  \end{equation*}
  such that the diagram
  \begin{equation}
    \label{eq:kappaSharp}
    \begin{tikzcd}[row sep=7ex]
      &
      R \flat R \prod_{i \in I} (\iota \circ p)(i)
      \arrow[d, "\varepsilon_{R \prod_{i \in I} (\iota \circ p)(i)}"]
      \\
      \bigoplus_{i \in I} (\iota \circ p)(i)
      \arrow[ur, "\kappa^{\#}"]
      \arrow[ r, "\kappa"']
      &
      R \prod_{i \in I} (\iota \circ p)(i)
    \end{tikzcd}
  \end{equation}
  commutes.
  We show that the induced map on cohomology stalks
  \begin{equation*}
    \varinjlim_{t \in U} H^n (U; \kappa^{\#}) \colon
    \varinjlim_{t \in U} H^n \left(
      U; \bigoplus_{i \in I} (\iota \circ p)(i)
    \right)
    \longrightarrow
    \varinjlim_{t \in U} H^n \left(
      U; R \flat R \prod_{i \in I} (\iota \circ p)(i)
    \right)    
  \end{equation*}
  is an isomorphism.
  By \cref{lem:mapDualStalk} we may as well show
  that the induced map
  \begin{equation*}
    \begin{tikzcd}[row sep=10ex]
      \Hom_{D^+(X)} \left(
        R \flat R \prod_{i \in I} (\iota \circ p)(i),
        (\iota \circ T^{-n})(t)
      \right)
      \arrow[d,
      "{\Hom_{D^+(X)} \left(\kappa^{\#}, (\iota \circ T^{-n})(t)\right)}"
      description]
      \\
      \Hom_{D^+(X)} \left(
        \bigoplus_{i \in I} (\iota \circ p)(i),
        (\iota \circ T^{-n})(t)
      \right)
    \end{tikzcd}
  \end{equation*}
  is an isomorphism.
  \cref{prp:partFaithful} in turn implies
  that
  it suffices to show that the induced map
  \begin{equation*}
    \begin{tikzcd}[row sep=10ex]
      \mathrm{Nat} \left(
        h_{\gamma} \left(
          R \flat R \prod_{i \in I} (\iota \circ p)(i)
        \right),
        (h_{\gamma} \circ \iota \circ T^{-n})(t)
      \right)
      \arrow[d,
      "{\mathrm{Nat} \left(
          h_{\gamma} \left(\kappa^{\#}\right),
          (h_{\gamma} \circ \iota \circ T^{-n})(t)
        \right)}"
      description]
      \\
      \mathrm{Nat} \left(
        h_{\gamma} \left(
          \bigoplus_{i \in I} (\iota \circ p)(i)
        \right),
        (h_{\gamma} \circ \iota \circ T^{-n})(t)
      \right)
    \end{tikzcd}    
  \end{equation*}
  is an isomorphism.
  To this end,
  we apply the functor
  \begin{equation*}
    \mathrm{Nat} (h_{\gamma} (-), (h_{\gamma} \circ \iota \circ T^{-n})(t))
    \colon
    D^+ (q_{\gamma}) \rightarrow \VectF
  \end{equation*}
  to the commutative triangle \eqref{eq:kappaSharp}
  to obtain the commutative triangle
  \begin{equation*}
    \begin{tikzcd}[row sep=10ex, column sep=-10ex]
      \mathrm{Nat} \left(
        h_{\gamma} \left(
          R \prod_{i \in I} (\iota \circ p)(i)
        \right), (h_{\gamma} \circ \iota \circ T^{-n})(t)
      \right)
      \arrow[ddr, bend right,
      "{\mathrm{Nat} \left(
          h_{\gamma} \left(
            \kappa
          \right), (h_{\gamma} \circ \iota \circ T^{-n})(t)
        \right)}"']
      \arrow[dr,
      "{\mathrm{Nat} \left(
          h_{\gamma} \left(
            \varepsilon_{R \prod_{i \in I} (\iota \circ p)(i)}
          \right), (h_{\gamma} \circ \iota \circ T^{-n})(t)
        \right)}"]
      \\
      &
      \mathrm{Nat} \left(
        h_{\gamma} \left(
          R \flat R \prod_{i \in I} (\iota \circ p)(i)
        \right), (h_{\gamma} \circ \iota \circ T^{-n})(t)
      \right)
      \arrow[d,
      "{\mathrm{Nat} \left(
          h_{\gamma} \left(
            \kappa^{\#}
          \right), (h_{\gamma} \circ \iota \circ T^{-n})(t)
        \right)}"]
      \\
      &
      \mathrm{Nat} \left(
        h_{\gamma} \left(
          \bigoplus_{i \in I} (\iota \circ p)(i)
        \right), (h_{\gamma} \circ \iota \circ T^{-n})(t)
      \right)
      .
    \end{tikzcd}
  \end{equation*}
  By \cref{lem:compatCorefl}
  the upper left diagonal map
  ${\mathrm{Nat} \left(
      h_{\gamma} \left(
        \varepsilon_{R \prod_{i \in I} (\iota \circ p)(i)}
      \right), (h_{\gamma} \circ \iota \circ T^{-n})(t)
    \right)}$
  is an isomorphism.
  Thus,
  it suffices to show that the curved map on the left hand side
  ${\mathrm{Nat} \left(
      h_{\gamma} \left(
        \kappa
      \right), (h_{\gamma} \circ \iota \circ T^{-n})(t)
    \right)}$
  is an isomorphism.
  To this end,
  we consider the commutative diagram
  \begin{equation*}
    \begin{tikzcd}[column sep=9ex, row sep=5ex]
      h_{\gamma} \left(
        \bigoplus_{i \in I} (\iota \circ p)(i)
      \right)
      \arrow[r, "h_{\gamma} (\kappa)"]
      &
      h_{\gamma} \left(
        R \prod_{i \in I} (\iota \circ p)(i)
      \right)
      \arrow[d]
      \\
      \bigoplus_{i \in I} 
      (h_{\gamma} \circ \iota \circ p)(i)
      \arrow[u]
      \arrow[r]
      &
      \prod_{i \in I} 
      (h_{\gamma} \circ \iota \circ p)(i)
      \\
      \bigoplus_{i \in I} 
      B_{p(i)}
      \arrow[u]
      \arrow[r, equal]
      &
      \prod_{i \in I} 
      B_{p(i)}
      .
      \arrow[u]
    \end{tikzcd}
  \end{equation*}
  Here the equality at the bottom follows from \cref{lem:sumInJ}.
  Moreover,
  the two vertical arrows in the second row are natural isomorphisms
  by \cref{cor:iota}.
  By the universal property of the derived product
  the vertical arrow on the upper right hand side
  is a natural isomorphism as well.
  Furthermore,
  if we apply the functor
  \begin{equation}
    \label{eq:NatFunc}
    \mathrm{Nat} (-, (h_{\gamma} \circ \iota \circ T^{-n})(t))
    \colon
    \VectF^{\M^{\circ}} \rightarrow \VectF
  \end{equation}
  to the vertical map on the upper left hand side,
  then we obtain an isomorphism by \cref{lem:dirSumFlip}.
  Thus applying the functor \eqref{eq:NatFunc}
  to the horizontal arrow at the top
  yields an isomorphism as well
  and this implies the result.
\end{proof}

\section{Abelianization of $D^+_t (q_{\gamma}, \partial q)$}
\label{sec:abelianizationSheaves}

Recall that $\mathcal{C}$ is the category of functors
$\M^{\circ} \rightarrow \VectF$
vanishing on $\partial \M$.

\begin{prp}
  \label{prp:abelianizationSheaves}
  The category
  of $\mathcal{J}$-presentable functors
  $\mathrm{pres}(\mathcal{J})$
  is an abelian subcategory of $\mathcal{C}$.
  Moreover,
  the cohomological functor
  \[{D^+_t (q_{\gamma}, \partial q) \xrightarrow{~h_{\gamma}~}
      \mathcal{J} \hookrightarrow \mathrm{pres}(\mathcal{J})}\]
  is the abelianization of $D^+_t (q_{\gamma}, \partial q)$.
  In other words,
  the precomposition functor
  $(-) \circ h_{\gamma,0,t}$
  from the category of exact functors
  $\mathrm{pres}(\mathcal{J}) \rightarrow \mathcal{A}$
  to the category of cohomological functors
  $D^+_t (q_{\gamma}, \partial q) \rightarrow \mathcal{A}$  
  is an equivalence of categories
  for any abelian category $\mathcal{A}$.
\end{prp}

\begin{proof}
  By \cref{cor:tameVanishingTriaSubcat}
  the category $D^+_t (q_{\gamma}, \partial q)$
  is indeed triangulated.
  Now let
  $F \colon D^+_t (q_{\gamma}, \partial q) \rightarrow \mathcal{C}$
  be the composition of the two functors
  $h_{\gamma,0,t} \colon D^+_t (q_{\gamma}, \partial q) \rightarrow \mathcal{J}$
  and the inclusion functor
  $\mathcal{J} \hookrightarrow \mathcal{C}$.
  Then $\mathcal{J}$ is the essential image of
  $F \colon D^+_t (q_{\gamma}, \partial q) \rightarrow \mathcal{C}$
  by \cref{thm:equiv}.
  With this the result follows directly from \cref{prp:abelianization}.
\end{proof}

\begin{cor}
  If $\partial q$ is closed in $q_{\gamma}$,
  then the composition of functors
  \begin{equation*}
    D^+_t (\dot{q}) \xrightarrow{R i_*}
    D^+_t (q_{\gamma}) \xrightarrow{h_{\gamma}}
    \mathcal{J} \hookrightarrow
    \mathrm{pres}(\mathcal{J})
  \end{equation*}
  is the abelianization of $D^+_t (\dot{q})$.
\end{cor}

\begin{proof}
  This follows in conjunction with
  \cref{cor:compsWithDirIms}, \cref{lem:vanishingTameInteriorAdj},
  and \cref{cor:tameInteriorTriaSubcat}.
\end{proof}

\begin{cor}
  \label{cor:abelianizationReals}
  The composition of functors
  \begin{equation*}
    D^+_t (\R) \xrightarrow{h_{\R}}
    \mathcal{J} \hookrightarrow
    \mathrm{pres}(\mathcal{J})
  \end{equation*}
  is the abelianization of $D^+_t (\R)$.
\end{cor}

\begin{proof}[Proof of \cref{thm:frobenius}]
  By \cref{prp:abelianizationSheaves} the category
  of $\mathcal{J}$-presentable functors
  $\mathrm{pres}(\mathcal{J})$
  is an abelian subcategory of $\mathcal{C}$,
  which is an abelian subcategory of
  $\VectF^{\M^{\circ}}$.
  Moreover, \cref{prp:abelianizationSheaves} further implies
  that
  $\mathrm{pres}(\mathcal{J})$
  is the abelianization
  of the triangulated category
  $D^+_t (q_{\gamma}, \partial q)$.
  As such, $\mathrm{pres}(\mathcal{J})$ is Frobenius
  by \mbox{\cite[Section 4.2]{MR2355771}}.
  Furthermore,
  a $\mathcal{J}$-presentable functor is projective
  in $\mathrm{pres}(\mathcal{J})$ iff it is in $\mathcal{J}$
  by \cref{cor:JiffProj}.
  Thus,
  $\mathcal{J}$ is the subcategory of projectives
  in $\mathrm{pres}(\mathcal{J})$.
\end{proof}

Moreover, \cref{thm:frobenius} has the following corollary.

\begin{cor}
  The triangulated category
  $D^+_t (q_{\gamma}, \partial q)$
  has split idempotents.
\end{cor}

\begin{proof}
  By \mbox{\cite[Section 4.3]{MR2355771}}
  and \cref{prp:abelianizationSheaves}
  the subcategory of projectives in $\mathrm{pres}(\mathcal{J})$
  has split idempotents,
  which is $\mathcal{J}$ by \cref{thm:frobenius}.
  Furthermore,
  $\mathcal{J}$ and $D^+_t (q_{\gamma}, \partial q)$
  are equivalent as $\F$-linear categories
  by \cref{thm:equiv}.
\end{proof}

We end this section with a proof of \cref{prp:minProjRes}.

\begin{lem}
  \label{lem:bettiComparison}
  Let
  $\varphi \colon P \rightarrow F$
  be a projective cover in $\mathrm{pres}(\mathcal{J})$
  and let
  $\kappa \colon K \rightarrow P$
  be its kernel
  so we have the short exact sequence
  \begin{equation*}
    0 \rightarrow
    K \xrightarrow{\kappa}
    P \xrightarrow{\varphi}
    F \rightarrow
    0 .
  \end{equation*}
  Then $\beta^{n+1} (F) = \beta^n (K)$
  for all $n \in \N_0$.
\end{lem}

\begin{proof}
  Let $u \in \op{int} \M$ be a point
  inside the interior of $\M$.
  As $\mathrm{Nat}(-, S_u)$ is left-exact
  as a functor on $\mathcal{C}$
  and since
  $P \colon \M^{\circ} \rightarrow \VectF$
  is projective in $\mathcal{C}$
  by \cref{thm:frobenius} and
  \cite[Corollary 3.6]{2021arXiv210809298B}
  we obtain the long exact sequence
  \begin{equation}
    \label{eq:bettiComparison}
    \!\!\!\!\!\!\!\!\!\!\!\!\!\!\!\!\!\!\!\!\!
    \begin{tikzcd}[row sep=6ex, column sep=2.5ex]
      &
      \Ext_{\mathcal{C}}^2(F, S_u)
      \arrow[r]
      &
      \{0\}
      \arrow[r]
      &
      \cdots
      .
      \\
      &
      \Ext_{\mathcal{C}}^1(F, S_u)
      \arrow[r]
      &
      \{0\}
      \arrow[r]
      &
      \Ext_{\mathcal{C}}^2(K, S_u)
      \arrow[ull, out=0, in=180, looseness=1.4, overlay]      
      \\
      0
      \arrow[r]
      &
      \mathrm{Nat} (F, S_u)
      \arrow[r]
      &
      \mathrm{Nat} (P, S_u)
      \arrow[r]
      &
      \mathrm{Nat} (K, S_u)
      \arrow[ull, out=0, in=180, looseness=1.4, overlay]
    \end{tikzcd}
  \end{equation}
  Moreover,
  by \cref{prp:projCover}
  we have
  \begin{equation*}
    \dim \mathrm{Nat} (F, S_u) =
    \beta^0(F)(u) =
    \beta^0(P)(u) =
    \dim \mathrm{Nat} (P, S_u)
    ,
  \end{equation*}
  hence the naturally induced map
  $\mathrm{Nat} (F, S_u) \rightarrow \mathrm{Nat} (P, S_u)$
  on the lower left hand side of \eqref{eq:bettiComparison}
  is an isomorphism.
  In conjunction with \eqref{eq:bettiComparison}
  we obtain natural isomorphisms
  \begin{equation*}
    \Ext_{\mathcal{C}}^n(K, S_u) \rightarrow \Ext_{\mathcal{C}}^{n+1}(F, S_u)
  \end{equation*}
  for all $n \in \N_0$.
\end{proof}

\begin{proof}[Proof of \cref{prp:minProjRes}]
  By \cref{prp:projCover} and \cref{thm:frobenius}
  any $\mathcal{J}$-presentable functor
  has a minimal projective resolution
  \begin{equation*}
    \dots \rightarrow
    P_n \rightarrow
    \dots \rightarrow
    P_2 \rightarrow
    P_1 \rightarrow
    P_0 \rightarrow
    F \rightarrow 0
  \end{equation*}
  by functors $P_n \colon \M^{\circ} \rightarrow \VectF$
  in $\mathcal{J}$.
  The statement about the multiplicities of indecomposables
  and the Betti functions
  follows by induction from
  \cref{lem:bettiComparison} and \cref{prp:projCover}.
\end{proof}

Moreover, \cref{prp:minProjRes} has the following corollary.

\begin{cor}
  \label{cor:admissibleBetti}
  Let
  $F \colon \M^{\circ} \rightarrow \VectF$
  be a $\mathcal{J}$-presentable functor.
  Then
  $\beta^n (F) \colon \op{int} \M \rightarrow \N_0$
  is an admissible Betti function
  for any $n \in \N_0$.
\end{cor}

\begin{proof}
  Let $n \in \N_0$ and let $u \in \op{int} \M$.
  Then
  $\beta^n (F)(u)$
  is the multiplicity of the indecomposable
  $B_u \colon \op{int} \M \rightarrow \VectF$
  in some projective
  $P \colon \M^{\circ} \rightarrow \VectF$
  by \cref{prp:minProjRes},
  which is the same as $\beta^0 (P)(u)$.
  Thus, $\beta^n (F) = \beta^0 (P)$,
  which is an admissible Betti function by \cref{lem:admissibleBetti}.
\end{proof}

\section{Categorification of Persistence Diagrams}
\label{sec:categorification}

We show that the sequence of Betti functions
is actually determined by the first four of these.
To this end,
the following notion will be useful.

\begin{dfn}
  \label{dfn:equivProjRes}
  Let
  \begin{equation}
    \label{eq:projRes}
    \dots \xrightarrow{\delta_{n+1}}
    P_n \xrightarrow{\delta_n}
    \dots \xrightarrow{\delta_3}
    P_2 \xrightarrow{\delta_2}
    P_1 \xrightarrow{\delta_1}
    P_0 \xrightarrow{\epsilon}
    F \rightarrow
    0
  \end{equation}
  be a projective resolution
  of $F \colon \M^{\circ} \rightarrow \VectF$
  in $\mathrm{pres}(\mathcal{J})$.
  We say that the projective resolution \eqref{eq:projRes}
  is \emph{equivariant},
  if
  \begin{equation*}
    P_{n+3} = P_n \circ T
    \quad \text{and} \quad
    \delta_{n+3} = - \delta_n \circ T
    \quad \text{for all $n \in \N_0$.}
  \end{equation*}
\end{dfn}

\begin{lem}
  \label{lem:equivProjRes}
  Any $\mathcal{J}$-presentable functor
  $F \colon \M^{\circ} \rightarrow \VectF$
  has an equivariant projective resolution.
\end{lem}

\begin{proof}
  By \cref{dfn:presentable,thm:equiv} we may choose a presentation
  \begin{equation*}
    h_{\gamma} (Y) \xrightarrow{h_{\gamma}(\delta)}
    h_{\gamma} (X) \rightarrow F \rightarrow 0
  \end{equation*}
  with $X$ and $Y$ objects of $D^+_t (q_{\gamma}, \partial q)$.
  Then we may form homotopy pullback squares
  \begin{equation*}
    \begin{tikzcd}
      {X[-1]}
      \arrow[r, "\theta"]
      \arrow[d]
      &
      Z
      \arrow[r]
      \arrow[d, "\kappa"]
      &
      0
      \arrow[d]
      \\
      0
      \arrow[r]
      &
      Y
      \arrow[r, "\delta"']
      &
      X
      .
    \end{tikzcd}
  \end{equation*}
  As
  $h_{\gamma} (X[-1]) \cong h_{\gamma} (X) \circ T$
  we may set
  \begin{align*}
    P_{3n} & := h_{\gamma} (X) \circ T^n,
    &
    P_{3n+1} & := h_{\gamma} (Y) \circ T^n,
    &
    P_{3n+2} & := h_{\gamma} (Z) \circ T^n,
    \\
    \delta_{3n+1} & := (-1)^n h_{\gamma} (\delta) \circ T^n,
    &
    \delta_{3n+2} & := (-1)^n h_{\gamma} (\kappa) \circ T^n,
    &
    \text{and} ~\,
    \delta_{3n+3} & := (-1)^n h_{\gamma} (\theta) \circ T^n
  \end{align*}
  for all $n \in \N_0$.
  Moreover,
  as $h_{\gamma} \colon D^+(q_{\gamma}) \rightarrow \VectF^{\M^{\circ}}$
  is cohomological
  and as
  $\mathrm{pres}(\mathcal{J})$ is an abelian subcategory of $\VectF^{\M^{\circ}}$
  by \cref{thm:frobenius}
  (or by \cref{prp:abelianizationSheaves})
  the sequence
  \eqref{eq:projRes}
  is indeed exact.
  \cref{thm:frobenius} further implies
  that
  $P_n \colon \M^{\circ} \rightarrow \VectF$
  is projective for each $n \in \N_0$.
\end{proof}

Now let
$F \colon \M^{\circ} \rightarrow \VectF$
be a
$\mathcal{J}$-presentable functor.
We obtain the following corollary.

\begin{cor}
  \label{cor:bettiEquiv}
  We have
  $\beta^{n+3} (F) = \beta^n (F) \circ T$
  for all $n \geq 1$.
\end{cor}

As $F \colon \M^{\circ} \rightarrow \VectF$
has bounded above support
as a $\mathcal{J}$-presentable functor,
we obtain yet another corollary.

\begin{cor}
  \label{cor:bettiLocFinite}
  For any $u \in \op{int} \M$
  we have
  $\beta^n(F)(u) = 0$
  for almost all $n \in \N_0$.
\end{cor}

With this the following is a sound definition.

\begin{dfn}
  \label{dfn:eulerFn}
  We define the \emph{Euler function} of
  ${F \colon \M^{\circ} \rightarrow \vectF}$
  to be
  \begin{equation*}
    \chi(F) := \sum_{n \in \N_0} (-1)^n \beta^n(F) \colon
    \op{int} \M \rightarrow \Z, \,
    u \mapsto \sum_{n \in \N_0} (-1)^n \dim \Ext_{\mathcal{C}}^n(F, S_u)
    .
  \end{equation*}
\end{dfn}

Now \cref{cor:bettiEquiv} further implies
that for any bounded region in $\M$
the restriction of $\beta^n(F) \colon \M \rightarrow \N_0$
to this region vanishes
for almost all $n \in \N_0$.
Thus, the point-wise absolute value
\begin{equation*}
  |\chi(F)| \colon \op{int} \M \rightarrow \N_0,\,
  u \mapsto |\chi(F)(u)|
\end{equation*}
is an admissible Betti function.

\begin{dfn}
  \label{dfn:admissibleEuler}
  We say that a function
  $\mu \colon \op{int} \M \rightarrow \Z$
  is an \emph{admissible Euler function}
  if the point-wise absolute value
  \[|\mu| \colon \op{int} \M \rightarrow \N_0, \,
    u \mapsto |\mu(u)|\]
  is an admissible Betti function.
  Moreover,
  we denote the abelian group of admissible Euler functions by $G(\mathbb{B})$.
  Then
  the inclusion
  $\mathbb{B} \hookrightarrow G(\mathbb{B})$
  satisfies the universal property of the Grothendieck group
  of the commutative monoid $\mathbb{B}$.
\end{dfn}

In particular
$\chi(F) \colon \op{int} \M \rightarrow \Z$
is an admissible Euler function.
Next we show that
$\chi \colon \mathrm{Ob}(\mathrm{pres}(\mathcal{J})) \rightarrow G(\mathbb{B})$
is an additive invariant of $\mathrm{pres}(\mathcal{J})$.
To this end,
we choose an equivariant projective resolution
\begin{equation*}
  \dots \xrightarrow{\delta_{n+1}}
  P_n \xrightarrow{\delta_n}
  \dots \xrightarrow{\delta_3}
  P_2 \xrightarrow{\delta_2}
  P_1 \xrightarrow{\delta_1}
  P_0 \xrightarrow{\epsilon}
  F \rightarrow
  0
  .
\end{equation*}
Then we obtain the following counterpart
to Euler's polyhedron formula.

\begin{lem}
  \label{lem:eulerFormula}
  We have
  $\chi(F) = \sum_{n \in \N_0} (-1)^n \beta^0 (P_n)$.
\end{lem}

\begin{proof}
  Let $u \in \op{int} \M$.
  Then the cochain complex
  $\mathrm{Nat} (P_{\bullet}, S_u)$
  is finite
  as $P_n(u) \cong \{0\}$ for almost all $n \in \N_0$,
  and thus the Euler characteristic of
  $\mathrm{Nat} (P_{\bullet}, S_u)$ can be computed
  in terms of its cohomology
  $\Ext_{\mathcal{C}}^{\bullet} (F, S_u)$.
  As a result, we obtain
  \begin{align*}
    \sum_{n \in \N_0} (-1)^n \beta^0 (P_n)(u)
    & = \sum_{n \in \N_0} (-1)^n \dim \mathrm{Nat} (P_n, S_u)
    \\
    & = \sum_{n \in \N_0} (-1)^n \dim \Ext_{\mathcal{C}}^n(F, S_u)
    \\
    & = \chi(F)(u)
      .
      \qedhere
  \end{align*}
\end{proof}

Now in order to harness this lemma
to show the additivity of
$\chi \colon \mathrm{Ob}(\mathrm{pres}(\mathcal{J})) \rightarrow G(\mathbb{B})$,
we make the following observation about
$\beta^0 \colon \mathrm{Ob}(\mathrm{pres}(\mathcal{J})) \rightarrow \mathbb{B}$.

\begin{lem}
  \label{lem:bettiAdditive}
  The restriction of
  $\beta^0 \colon
  \mathrm{Ob}(\mathrm{pres}(\mathcal{J})) \rightarrow
  \mathbb{B}$
  to functors in $\mathcal{J}$
  is an additive invariant.
\end{lem}

\begin{proof}
  As all functors in $\mathcal{J}$ are projective
  by \cref{cor:JiffProj} or \cref{thm:frobenius},
  all short exact sequences in $\mathcal{J}$ split.
  Thus, the result follows from the additivity of the functor
  ${\mathrm{Nat}(-, S_u) \colon \mathrm{pres}(\mathcal{J}) \rightarrow \VectF}$
  for any ${u \in \op{int} \M}$.
\end{proof}

Now suppose
\begin{equation*}
  0 \rightarrow F \rightarrow G \rightarrow H \rightarrow 0
\end{equation*}
is a short exact sequence of $\mathcal{J}$-presentable functors.
Moreover, suppose we have equivariant projective resolutions
\begin{align*}
  \dots \rightarrow
  P_n \rightarrow
  \dots \rightarrow
  P_2 \rightarrow
  P_1 \rightarrow
  P_0 \rightarrow
  F \rightarrow 0 &
  \\
  \text{and} \quad
  \dots \rightarrow
  R_n \rightarrow
  \dots \rightarrow
  R_2 \rightarrow
  R_1 \rightarrow
  R_0 \rightarrow
  H \rightarrow 0 &.
\end{align*}

\begin{lem}[Horseshoe Lemma]
  \label{lem:horseshoe}
  There exist an equivariant projective resolution
  \begin{equation*}
    \dots \rightarrow
    Q_n \rightarrow
    \dots \rightarrow
    Q_2 \rightarrow
    Q_1 \rightarrow
    Q_0 \rightarrow
    G \rightarrow 0
  \end{equation*}
  and a short exact sequence of chain complexes
  \begin{equation*}
    0 \rightarrow
    P_{\bullet} \rightarrow
    Q_{\bullet} \rightarrow
    R_{\bullet} \rightarrow
    0
  \end{equation*}
  such that the diagram
  \begin{equation*}
    \begin{tikzcd}
      P_0
      \arrow[r]
      \arrow[d]
      &
      Q_0
      \arrow[r]
      \arrow[d]
      &
      R_0
      \arrow[d]
      \\
      F
      \arrow[r]
      &
      G
      \arrow[r]
      &
      H
    \end{tikzcd}
  \end{equation*}
  commutes.
\end{lem}

\begin{proof}
  If we apply the construction
  from the proof of the ordinary horseshoe lemma,
  then we obtain an equivariant chain complex.
\end{proof}

\begin{prp}
  \label{prp:eulerAdditive}
  The function
  $\chi \colon
  \mathrm{Ob}(\mathrm{pres}(\mathcal{J})) \rightarrow
  G(\mathbb{B})$
  is an additive invariant of $\mathrm{pres}(\mathcal{J})$.
\end{prp}

\begin{proof}
  This result follows directly from Lemmas
  \ref{lem:eulerFormula}, \ref{lem:horseshoe}, and \ref{lem:bettiAdditive}.
\end{proof}

In particular we obtain the group homomorphism
\begin{equation*}
  [\chi] \colon K_0(\mathrm{pres}(\mathcal{J})) \rightarrow G(\mathbb{B}), \,
  [F] \mapsto \chi(F)
  .
\end{equation*}
Next we show that $[\chi]$ is an isomorphism.
Now $\mathcal{J}$ being the subcategory of projectives
in $\mathrm{pres}(\mathcal{J})$ by \cref{thm:frobenius},
it is a Quillen exact category as well.
Thus, we may also consider the Grothendieck group
$K_0 (\mathcal{J})$ of $\mathcal{J}$.
Moreover,
the restriction of
${\beta^0 \colon \mathrm{Ob}(\mathrm{pres}(\mathcal{J})) \rightarrow \mathbb{B}}$
to functors in $\mathcal{J}$
is an additive invariant by \cref{lem:bettiAdditive},
hence we obtain the group homomorphism
\begin{equation*}
  [\beta^0] \colon K_0 (\mathcal{J}) \rightarrow G(\mathbb{B}), \,
  [P] \mapsto \beta^0(P)
  .
\end{equation*}
Furthermore,
since $\beta^n(P) = 0$ for any $n \geq 1$ and any projective $P$,
we obtain the commutative triangle
\begin{equation}
  \label{eq:diagGrothendieck}
  \begin{tikzcd}[row sep=6ex]
    K_0(\mathcal{J})
    \arrow[r]
    \arrow[rd, "{[\beta^0]}"']
    &
    K_0(\mathrm{pres}(\mathcal{J}))
    \arrow[d, "{[\chi]}"]
    \\
    &
    G(\mathbb{B})
  \end{tikzcd}
\end{equation}
of abelian groups.
Thus, in order to show that
${[\chi] \colon K_0(\mathrm{pres}(\mathcal{J})) \rightarrow G(\mathbb{B})}$
is an isomorphism it suffices to show that
${[\beta^0] \colon K_0 (\mathcal{J}) \rightarrow G(\mathbb{B})}$
is an isomorphism
and that the upper vertical map in \eqref{eq:diagGrothendieck}
induced by the subcategory inclusion
${\mathcal{J} \hookrightarrow \mathrm{pres}(\mathcal{J})}$
is an epimorphism.

\begin{lem}
  \label{lem:bettiIso}
  The group homomorphism
  $[\beta^0] \colon K_0 (\mathcal{J}) \rightarrow G(\mathbb{B})$
  is an isomorphism.
\end{lem}

\begin{proof}
  By \cref{lem:sumInJ}
  the group homomorphism
  $[\beta^0] \colon K_0 (\mathcal{J}) \rightarrow G(\mathbb{B})$
  is surjective .
  Now suppose ${P \colon \M^{\circ} \rightarrow \VectF}$
  is a functor in $\mathcal{J}$
  such that
  ${\beta^0(P) = [\beta^0]([P]) = 0}$.
  Then the zero natural transformation
  ${0 \rightarrow P}$
  is a projective cover of $P$ by \cref{prp:minProjRes}
  and thus ${P \cong 0}$.
\end{proof}

\begin{lem}
  \label{lem:grothendieckIso}
  The group homomorphism
  ${K_0 (\mathcal{J}) \rightarrow K_0(\mathrm{pres}(\mathcal{J}))}$
  induced by the full subcategory inclusion
  ${\mathcal{J} \hookrightarrow \mathrm{pres}(\mathcal{J})}$
  is an isomorphism.
\end{lem}

\begin{proof}
  By \cref{lem:bettiIso}
  the homomorphism
  ${[\beta^0] \colon K_0 (\mathcal{J}) \rightarrow G(\mathbb{B})}$
  is injective.
  In conjunction with the commutativity of \eqref{eq:diagGrothendieck}
  this implies the injectivity of
  ${K_0 (\mathcal{J}) \rightarrow K_0(\mathrm{pres}(\mathcal{J}))}$.
  Now let
  $F \colon \M^{\circ} \rightarrow \VectF$
  be a $\mathcal{J}$-presentable functor
  and let
  \begin{equation*}
    \dots \xrightarrow{\delta_{n+1}}
    P_n \xrightarrow{\delta_n}
    \dots \xrightarrow{\delta_3}
    P_2 \xrightarrow{\delta_2}
    P_1 \xrightarrow{\delta_1}
    P_0 \xrightarrow{\epsilon}
    F \rightarrow
    0
  \end{equation*}
  be an equivariant projective resolution of $F$.
  Then both
  \begin{equation*}
    \bigoplus_{n=0}^{\infty} P_{2n}
    \quad \text{and} \quad
    \bigoplus_{n=0}^{\infty} P_{2n+1}
  \end{equation*}
  are functors in $\mathcal{J}$
  and thus it suffices to show that
  \begin{equation*}
    \left[\bigoplus_{n=0}^{\infty} P_{2n}\right]
    =
    [F] +
    \left[\bigoplus_{n=0}^{\infty} P_{2n+1}\right]
  \end{equation*}
  as elements of $K_0(\mathrm{pres}(\mathcal{J}))$.
  To this end,
  we consider the exact sequence
  \begin{equation}
    \label{eq:seqProjSums}
    \begin{tikzcd}
      \displaystyle \bigoplus_{n=0}^{\infty} P_{2n+3}
      \arrow[r, "\varphi_3"]
      &
      \displaystyle \bigoplus_{n=0}^{\infty} P_{2n+2}
      \arrow[r, "\varphi_2"]
      &
      \displaystyle \bigoplus_{n=0}^{\infty} P_{2n+1}
      \arrow[r, "\varphi_1"]
      &
      \displaystyle \bigoplus_{n=0}^{\infty} P_{2n}
      ,
    \end{tikzcd}
  \end{equation}
  where
  \begin{equation*}
    \varphi_3 := \bigoplus_{n=0}^{\infty} \delta_{2n+3},
    \quad
    \varphi_2 := \bigoplus_{n=0}^{\infty} \delta_{2n+2},
    \quad \text{and} ~~
    \varphi_1 := \bigoplus_{n=0}^{\infty} \delta_{2n+1}.
  \end{equation*}
  First we note that
  \begin{equation}
    \label{eq:cokerPlusF}
    \op{coker} \varphi_1 \cong
    \op{coker} \delta_1 \oplus \op{coker} \varphi_3 \cong
    F \oplus \op{coker} \varphi_3
    .
  \end{equation}
  Moreover,
  by the exactness of \eqref{eq:seqProjSums}
  we have $\op{coker} \varphi_3 \cong \op{Im} \varphi_2 = \ker \varphi_1$.
  In conjunction with \eqref{eq:cokerPlusF} we obtain
  $\op{coker} \varphi_1 \cong F \oplus \ker \varphi_1$
  and hence
  \begin{equation*}
    [\op{coker} \varphi_1] = [F] + [\ker \varphi_1]
    .
  \end{equation*}
  From this equation in turn we obtain
  \begin{align*}
    \left[\bigoplus_{n=0}^{\infty} P_{2n}\right]
    & = [\op{Im} \varphi_1] + [\op{coker} \varphi_1]
    \\
    & = [\op{Im} \varphi_1] + [F] + [\ker \varphi_1]
    \\
    & = [F] + [\op{ker} \varphi_1] + [\op{Im} \varphi_1]
    \\
    & = [F] + \left[\bigoplus_{n=0}^{\infty} P_{2n+1}\right].
      \qedhere
  \end{align*}
\end{proof}

\begin{proof}[Proof of \cref{thm:eulerIso}]
  This follows directly from
  Lemmas \ref{lem:bettiIso} and \ref{lem:grothendieckIso}
  and the commutativity of the triangle \eqref{eq:diagGrothendieck}.
\end{proof}

\bibliographystyle{alpha}
\bibliography{
  bib/cohen-steiner-edelsbrunner-harer-2007.bib,
  bib/keller-scherotzke-2016.bib,
  bib/lesnick-wright-2019.bib,
  bib/happel-1988.bib,
  bib/hiraoka-2020.bib,
  bib/blanchette-bruestle-hanson-2021.bib,
  bib/botnan-oppermann-oudot-2021.bib,
  bib/carlsson-zomorodian-2005.bib,
  bib/carlsson-zomorodian-2009.bib,
  bib/quillen-1973.bib,
  bib/swan-1971.bib,
  bib/patel-2018.bib,
  bib/bauer-botnan-fluhr-2020.bib,
  bib/bauer-botnan-fluhr-2021.bib,
  bib/berkouk-ginot-oudot-2019.bib,
  bib/curry-2014.bib,
  bib/kashiwara-schapira-1990.bib,
  bib/kashiwara-schapira-2006.bib,
  bib/kashiwara-schapira-2018.bib,
  bib/tomDieck-2008.bib,
  bib/carlsson-2009.bib,
  bib/krause-2015.bib,
  bib/krause-2007.bib,
  bib/sella-2016.bib,
  bib/cohen-steiner-edelsbrunner-harer-2009.bib,
  bib/mazorchuk-2012.bib,
  bib/bendich-edelsbrunner-morozov-patel-2013.bib
}

\appendix

\section{General Facts on Derived Adjunctions}
\label{sec:homologicalAlg}

Let
\begin{equation*}
  \begin{tikzcd}
    \mathcal{B}
    \arrow[r, "F"'{name=F}, bend right]
    &
    \mathcal{C}
    \arrow[l, "G"'{name=G}, bend right]
    \arrow[phantom, from=F, to=G, "\dashv" rotate=90]
  \end{tikzcd}
\end{equation*}
be an adjunction of abelian categories.
Moreover, we assume that 
the left adjoint $F$ is exact (or equivalently left-exact)
and that $\mathcal{C}$ has enough injectives.

\begin{lem}
  \label{lem:coreflection}
  For any object $X \in \mathcal{B}$
  the derived unit
  \begin{equation*}
    \eta^{\Der (\mathcal{B})}_X \colon
    X \xrightarrow{\eta_X}
    (G \circ F)(X) \rightarrow (RG \circ F)(X)
  \end{equation*}
  is an isomorphism of the derived category
  $\Der (\mathcal{B})$
  iff
  the ordinary unit $\eta_X$ is an isomorphism and
  $F(X)$ is $G$-acyclic, i.e., $R^k \flat (G) = 0$ for all $k \neq 0$.
\end{lem}

\begin{lem}
  \label{lem:compose}
  Let
  ${H \colon \mathcal{B} \rightarrow \mathcal{A}}$
  be a left exact functor
  that has a right derived functor
  ${RH \colon D^+(\mathcal{B}) \rightarrow D^+(\mathcal{A})}$.
  Then the derived functors $RH$ and $RG$ compose as
  ${R (H \circ G) \cong RH \circ RG.}$
\end{lem}

\begin{lem}
  \label{lem:cohomChar}
  If the adjunction $\Der (F) = LF \dashv RG$ is
  \href{https://ncatlab.org/nlab/show/coreflective+subcategory}{coreflective},
  then the essential image of $\Der (F)$ is the full subcategory
  $\Der_{F(\mathcal{B})} (\mathcal{C})$
  of complexes whose cohomology objects are in the essential image of $F$.
  In particular,
  $\Der_{F(\mathcal{B})} (\mathcal{C})$
  is a triangulated subcategory of $\Der (\mathcal{C})$.
\end{lem}

\section{The Beck--Chevalley Condition}
\label{sec:beckChev}

Suppose we have a square
\begin{equation}
  \label{eq:beckChevCand}
  \begin{tikzcd}
    \mathcal{C}_1
    \arrow[r, "G_1"]
    \arrow[d, "H_1"']
    &
    \mathcal{C}_2
    \arrow[d, "H_2"]
    \arrow[ld, "\zeta"', Rightarrow, shorten >=1.5ex, shorten <=1.5ex]
    \\
    \mathcal{C}_3
    \arrow[r, "G_2"']
    &
    \mathcal{C}_4
  \end{tikzcd}
\end{equation}
of categories and functors
that commutes up to a natural isomorphism
$\zeta \colon H_2 \circ G_1 \Rightarrow G_2 \circ H_1$.
Moreover,
suppose
$G_1 \colon \mathcal{C}_1 \rightarrow \mathcal{C}_2$
and
$G_2 \colon \mathcal{C}_3 \rightarrow \mathcal{C}_4$
have left adjoints
$F_1 \colon \mathcal{C}_2 \rightarrow \mathcal{C}_1$
and
$F_2 \colon \mathcal{C}_4 \rightarrow \mathcal{C}_3$,
respectively.
Then we also have the square diagram
\begin{equation*}
  \begin{tikzcd}
    \mathcal{C}_1
    \arrow[d, "H_1"']
    &
    \mathcal{C}_2
    \arrow[l, "F_1"']
    \arrow[d, "H_2"]
    \\
    \mathcal{C}_3
    &
    \mathcal{C}_4
    ,
    \arrow[l, "F_2"]
    \arrow[lu, "\xi"', Rightarrow, shorten >=1.5ex, shorten <=1.5ex]
  \end{tikzcd}
\end{equation*}
where $\xi \colon F_2 \circ H_2 \Rightarrow H_1 \circ F_1$,
defined as the composition
\begin{equation*}
  \begin{tikzcd}[row sep=5ex]
    F_2 \circ H_2
    \arrow[d, Rightarrow, "F_2 \circ H_2 \circ \eta^1"]
    \\
    F_2 \circ H_2 \circ G_1 \circ F_1
    \arrow[d, Rightarrow, "F_2 \circ \zeta \circ F_1"]
    \\
    F_2 \circ G_2 \circ H_1 \circ F_1
    \arrow[d, Rightarrow, "\varepsilon^2 \circ H_1 \circ F_1"]
    \\
    H_1 \circ F_1
    ,
  \end{tikzcd}
\end{equation*}
is the so called \emph{mate} of
$\zeta \colon H_2 \circ G_1 \Rightarrow G_2 \circ H_1$.

\begin{dfn}[Beck--Chevalley Condition]
  \label{dfn:BeckChev}
  We say that the square diagram
  \eqref{eq:beckChevCand}
  satisfies the \emph{Beck--Chevalley condition}
  if $\xi \colon F_2 \circ H_2 \Rightarrow H_1 \circ F_1$
  is a natural isomorphism.
\end{dfn}

Morally, a commutative square of categories and functors
satisfies the Beck--Chevalley condition,
if the horizontal arrows have left adjoints and the corresponding
square with left adjoints commutes as well,
which is not precisely the same as \cref{dfn:BeckChev},
but it is close enough for intuition.

There is a dual version of the Beck--Chevalley condition as well,
which involves right adjoints in place of left adjoints.
To this end,
we consider the square diagram
\begin{equation}
  \label{eq:beckChevCandDual}
  \begin{tikzcd}
    \mathcal{C}_1
    \arrow[r, "F_1"]
    \arrow[d, "H_1"']
    &
    \mathcal{C}_2
    \arrow[d, "H_2"]
    \\
    \mathcal{C}_3
    \arrow[r, "F_2"']
    \arrow[ru, "\zeta", Rightarrow, shorten >=1.5ex, shorten <=1.5ex]
    &
    \mathcal{C}_4
    ,
  \end{tikzcd}
\end{equation}
with $\zeta \colon F_2 \circ H_1 \Rightarrow H_2 \circ F_1$
a natural isomorphism.
Moreover,
suppose $F_1 \colon \mathcal{C}_1 \rightarrow \mathcal{C}_2$
and $F_2 \colon \mathcal{C}_3 \rightarrow \mathcal{C}_4$
have right adjoints
$G_1 \colon \mathcal{C}_2 \rightarrow \mathcal{C}_1$
and
$G_2 \colon \mathcal{C}_4 \rightarrow \mathcal{C}_3$,
respectively.
Then we have the square diagram
\begin{equation*}
  \begin{tikzcd}
    \mathcal{C}_1
    \arrow[d, "H_1"']
    \arrow[rd, "\xi", Rightarrow, shorten >=1.5ex, shorten <=1.5ex]
    &
    \mathcal{C}_2
    \arrow[l, "G_1"']
    \arrow[d, "H_2"]
    \\
    \mathcal{C}_3
    &
    \mathcal{C}_4
    \arrow[l, "G_2"]
  \end{tikzcd}
\end{equation*}
with $\xi \colon H_1 \circ G_1 \Rightarrow G_2 \circ H_2$
provided by the composition
\begin{equation*}
  \begin{tikzcd}[row sep=5ex]
    H_1 \circ G_1
    \arrow[d, Rightarrow, "\eta_2 \circ H_1 \circ G_1"]
    \\
    G_2 \circ F_2 \circ H_1 \circ G_1
    \arrow[d, Rightarrow, "G_2 \circ \zeta \circ G_1"]
    \\
    G_2 \circ H_2 \circ F_1 \circ G_1
    \arrow[d, Rightarrow, "G_2 \circ H_2 \circ \varepsilon_1"]
    \\
    G_2 \circ H_2
    .    
  \end{tikzcd}
\end{equation*}

\begin{dfn}[Dual Beck--Chevalley Condition]
  \label{dfn:BeckChevDual}
  We say that the square diagram
  \eqref{eq:beckChevCandDual}
  satisfies the \emph{dual Beck--Chevalley condition}
  if $\xi \colon H_1 \circ G_1 \Rightarrow G_2 \circ H_2$
  is a natural isomorphism.
\end{dfn}

\section{General Facts on Sheaves}

Let $X$ be a topological space
and let $\mathrm{Sh}(X)$ be the category of sheaves on $X$
with values in the category of abelian groups.
We write $D^+ (X) := D^+ (\mathrm{Sh}(X))$
for the bounded below derived category of $\mathrm{Sh}(X)$.

\begin{lem}
  \label{lem:kernelIso}
  Let $i \colon U \hookrightarrow X$
  and $j \colon A \hookrightarrow X$
  be inclusions with $U$ open and $A \cup U = X$.
  If $F$ is a sheaf on $X$,
  then the naturally induced map on kernels
  \begin{equation*}
    \begin{tikzcd}
      0
      \ar[r]
      &
      (\ker \circ \eta^j)_F
      \arrow[r]
      \arrow[d, dashed]
      &
      F
      \ar[d, "\eta^i_F"']
      \ar[r, "\eta^j_F"]
      &[7ex]
      j_* j^{-1} F
      \ar[d, "(j_* \circ j^{-1} \circ \eta^i)_F"]
      \\
      0
      \ar[r]
      &
      (\ker \circ \eta^j \circ i_* \circ i^{-1})_F
      \ar[r]
      &
      F
      \ar[r, "(\eta^j \circ i_* \circ i^{-1})_F"']
      &
      j_* j^{-1} i_* i^{-1} F
    \end{tikzcd}
  \end{equation*}
  is an isomorphism.
\end{lem}

\begin{lem}
  \label{lem:adjEquivProperSupp}
  Let $U$ be an open subset of $X$,
  let ${C := X \setminus U}$,
  let ${i \colon U \hookrightarrow X}$ be the inclusion,
  and let $\mathrm{Sh}(X, C)$ be the full subcategory of sheaves
  vanishing on $C$.
  We write
  ${(-)_U \colon \mathrm{Sh}(X) \rightarrow \mathrm{Sh}(X, C),\,
    F \mapsto F_U}$
  for the corresponding functor defined in
  \mbox{\cite[Page 93]{Kashiwara1990}}.
  Then we have $i_! = (-)_U \circ i_*$
  and moreover,
  the composition of adjunctions
  \begin{equation*}
    \begin{tikzcd}
      \mathrm{Sh}(U, C)
      \arrow[r, ""{name=I}, hook, bend right]
      &
      \mathrm{Sh}(X)
      \arrow[l, "(-)_U"'{name=b}, bend right]
      \arrow[phantom, from=I, to=b, "\dashv" rotate=90]
      \arrow[r, "i^{-1}"'{name=la}, bend right]
      &
      \mathrm{Sh}(U)
      \arrow[l, "i_*"'{name=ra}, bend right]
      \arrow[phantom, from=la, to=ra, "\dashv" rotate=90]
    \end{tikzcd}  
  \end{equation*}
  yields an exact adjoint equivalence:
  \begin{equation*}
    \begin{tikzcd}
      \mathrm{Sh}(U, C)
      \arrow[r, "i^{-1}"'{name=la}, bend right]
      &
      \mathrm{Sh}(U)
      .
      \arrow[l, "i_!"'{name=ra}, bend right]
      \arrow[phantom, from=la, to=ra, "\dashv" rotate=90]
    \end{tikzcd}  
  \end{equation*}  
\end{lem}

Now let $U, V \subseteq X$ be open subsets,
let $i \colon U \hookrightarrow X$
be the inclusion,
let $A \subseteq X$ be a closed subset,
let $Z := V \cap A$,
and suppose that $Z \subseteq U$.

\begin{lem}
  The commutative square
  \begin{equation*}
    \begin{tikzcd}
      \mathrm{Sh}(U)
      \arrow[r, "i_*"]
      \arrow[d, "\Gamma_Z(U; -)"']
      &
      \mathrm{Sh}(X)
      \arrow[d, "\Gamma_Z(X; -)"]
      \\
      \mathrm{Ab}
      \arrow[r, equal]
      &
      \mathrm{Ab}
    \end{tikzcd}
  \end{equation*}
  satisfies the Beck--Chevalley condition
  as defined in \cref{dfn:BeckChev},
  where $\mathrm{Ab}$ is the category of abelian groups.
\end{lem}

\begin{proof}
  Let $F$ be a sheaf on $X$.
  Then we have the composition of isomorphisms
  \begin{align*}
    \Gamma_Z (X; F)
    & \cong
      \Gamma_Z (V; F)
    \\
    & \cong
      \Gamma_Z (U \cap V; F)
    \\
    & =
      \Gamma_Z (U \cap V; F |_U)
    \\
    & \cong
      \Gamma_Z (U \cap V; i^{-1} F)
    \\
    & \cong
      \Gamma_Z (U; i^{-1} F)
      ,   
  \end{align*}
  which is the mate of the identity natural transformation.
\end{proof}

This lemma has the following corollary.

\begin{cor}
  \label{cor:BeckChevLocalCoho}
  The square diagram
  \begin{equation*}
    \begin{tikzcd}
      D^+ (U)
      \arrow[r, "R i_*"]
      \arrow[d, "H_Z (U; -)"'{name=HU}]
      &
      D^+ (X)
      \arrow[d, "H_Z (X; -)"{name=HX}]
      \\
      \mathrm{Ab}
      \arrow[r, equal]
      &
      \mathrm{Ab}
      \arrow[phantom, from=HU, to=HX, "\cong"]
    \end{tikzcd}
  \end{equation*}
  satisfies the Beck--Chevalley condition.
\end{cor}

Now let $\mathrm{Sh}(X)$ be the category of $\F$-linear sheaves on $X$
for some fixed field $\F$,
let $F$ be an object of $D^+ (X) := D^+(\mathrm{Sh}(X))$,
let $n \in \Z$,
let $x \in X$ be a point of $X$,
and let $S_x$ be the skyscraper sheaf at $x$.
We consider the map
\begin{equation}
  \label{eq:mapDualStalk}
  \Hom_{D^+(X)} (F, S_x [-n]) \rightarrow
  \left(\varinjlim_{x \in U} H^n (U; F)\right)^*
\end{equation}
which takes any homomorphism
${\psi \colon F \rightarrow S_x [-n]}$
to the family of maps
${\{H^n (U; \psi) \colon H^n(U; F) \rightarrow \F \mid x \in U\}}$
and then to the naturally induced map of type
${\varinjlim_{x \in U} H^n (U; F) \rightarrow \F}$.

\begin{lem}
  \label{lem:mapDualStalk}
  The map \eqref{eq:mapDualStalk} is an isomorphism.
\end{lem}

\begin{proof}
  As $S_x$ is injective,
  as $H^n(U; S_x [-n]) \cong \{0\}$
  for all opens $U$ excluding $x$,
  and by the exactness of the dual space functor
  (or the UCT for cohomology)
  the map \eqref{eq:mapDualStalk} is an isomorphism.
\end{proof}

\subsection{Mayer--Vietoris Sequence for Local Sheaf Cohomology}
\label{sec:sheavesMVS}

In addition to the original Mayer--Vietoris sequence
computing the homology of a union of open subsets,
there is a generalization computing the relative homology
of a component-wise union of pairs of open subsets by
\cite[Theorem 10.7.7]{MR2456045}.
Here we provide a counterpart for local sheaf cohomology.
To this end,
let $X$ be a topological space
and let $(X_i, A_i)$ for $i = 1,2$ be pairs of open subsets of $X$
in the sense that $A_i \subseteq X_i$ for $i = 1,2$.
Moreover,
let
\begin{align*}
  X_0 & := X_1 \cap X_2,
  &
    A_0 & := A_1 \cap A_2,
  \\
  X_3 & := X_1 \cup X_2,
        \quad \text{and}
  &
    A_3 & := A_1 \cup A_2.
\end{align*}

\begin{lem}
  \label{lem:nineMVS}
  If $F$ is a flabby sheaf on $X$,
  then the commutative diagram
  \begin{equation*}
    \begin{tikzcd}
      &
      0
      \arrow[d]
      &
      0
      \arrow[d]
      &[1.5ex]
      0
      \arrow[d]
      \\
      0
      \arrow[r]
      &
      \Gamma_{\overline{A}_3} (X_3; F)
      \arrow[r]
      \arrow[d]
      &
      \Gamma_{\overline{A}_1} (X_1; F) \oplus \Gamma_{\overline{A}_2} (X_2; F)
      \arrow[r, "(1 ~ -1)"]
      \arrow[d]
      &
      \Gamma_{\overline{A}_0} (X_0; F)
      \arrow[r]
      \arrow[d]
      &
      0
      \\
      0
      \arrow[r]
      &
      \Gamma(X_3; F)
      \arrow[r]
      \arrow[d]
      &
      \Gamma(X_1; F) \oplus \Gamma(X_2; F)
      \arrow[r, "(1 ~ -1)"]
      \arrow[d]
      &
      \Gamma(X_0; F)
      \arrow[r]
      \arrow[d]
      &
      0
      \\
      0
      \arrow[r]
      &
      \Gamma(A_3; F)
      \arrow[r]
      \arrow[d]
      &
      \Gamma(A_1; F) \oplus \Gamma(A_2; F)
      \arrow[r, "(1 ~ -1)"]
      \arrow[d]
      &
      \Gamma(A_0; F)
      \arrow[r]
      \arrow[d]
      &
      0
      \\
      &
      0
      &
      0
      &
      0
    \end{tikzcd}
  \end{equation*}
  has exact rows and columns,
  where $\overline{A}_i$ is the complement $X \setminus A_i$
  of $A_i$ in $X$ for $i = 0,1,2,3$.
\end{lem}

\begin{proof}
  The columns are exact
  as $F$ is a flabby sheaf and by the definition
  \cite[(2.3.14)]{Kashiwara1990}
  of
  $\Gamma_{\overline{A}_i} (X_i; F)$
  for $i = 0,1,2,3$.
  The two bottom rows are exact by
  the sheaf condition and as $F$ is a flabby sheaf.
  By the nine lemma the top row is exact as well.
\end{proof}

\begin{cor}
  \label{cor:sesFlabbyLocalSections}
  If $G$ is a complex of flabby sheaves on $X$,
  then we have the short exact sequence
  \begin{equation*}
    \begin{tikzcd}
      0
      \arrow[r]
      &
      \Gamma_{\overline{A}_3} (X_3; G)
      \arrow[r]
      &
      \Gamma_{\overline{A}_1} (X_1; G) \oplus \Gamma_{\overline{A}_2} (X_2; G)
      \arrow[r, "(1 ~ -1)"]
      &[1.5ex]
      \Gamma_{\overline{A}_0} (X_0; G)
      \arrow[r]
      &
      0
    \end{tikzcd}
  \end{equation*}
  of cochain complexes.
\end{cor}

\begin{prp}[Mayer--Vietoris for Local Sheaf Cohomology]
  If $F$ is an object of $D^+(X)$,
  then there is a long exact sequence
  \begin{equation*}
    \!\!\!\!\!\!\!\!\!\!\!\!\!\!\!\!\!\!
    \begin{tikzcd}
      0
      \arrow[r]
      &
      H^n_{\overline{A}_3} (X_3; F)
      \arrow[r]
      &
      H^n_{\overline{A}_1} (X_1; F)
      \oplus
      H^n_{\overline{A}_2} (X_2; F)
      \arrow[r]
      &
      H^n_{\overline{A}_0} (X_0; F)
      \arrow[dll, out=0, in=180, looseness=1.4, overlay]
      \\
      &
      H^{n+1}_{\overline{A}_3} (X_3; F)
      \arrow[r]
      &
      H^{n+1}_{\overline{A}_1} (X_1; F)
      \oplus
      H^{n+1}_{\overline{A}_2} (X_2; F)
      \arrow[r]
      &
      H^{n+1}_{\overline{A}_0} (X_0; F)
      \arrow[dll, out=0, in=180, looseness=1.4, overlay]
      \\
      &
      H^{n+2}_{\overline{A}_3} (X_3; F)
      \arrow[r]
      &
      H^{n+2}_{\overline{A}_1} (X_1; F)
      \oplus
      H^{n+2}_{\overline{A}_2} (X_2; F)
      \arrow[r]
      &
      H^{n+2}_{\overline{A}_0} (X_0; F)
      \arrow[dll, out=0, in=180, looseness=1.4, overlay]
      \\
      &
      H^{n+3}_{\overline{A}_3} (X_3; F)
      \arrow[r]
      &
      \cdots
    \end{tikzcd}
  \end{equation*}
  for some $n \in \Z$.
\end{prp}

\begin{proof}
  This follows directly from
  \cref{cor:sesFlabbyLocalSections}
  and the zig-zag lemma.
\end{proof}

\section{Presentable Objects}
\label{sec:presentable}

Let $\mathcal{C}$ be an abelian category
and let $\mathcal{J}$ be a full replete additive subcategory of projectives
in $\mathcal{C}$.

\begin{dfn}
  \label{dfn:presentable}
  We say that an object 
  $X$ of $\mathcal{C}$
  is \emph{$\mathcal{J}$-presentable}
  if there is an exact sequence
  (called \emph{presentation})
  \begin{equation*}
    Q \rightarrow P \rightarrow X \rightarrow 0
  \end{equation*}
  with $P$ and $Q$ in $\mathcal{J}$.
  We write $\mathrm{pres}(\mathcal{J})$
  for the full subcategory of $\mathcal{J}$-presentable objects
  in $\mathcal{C}$.
\end{dfn}

Next we show that the category
of $\mathcal{J}$-presentable objects of $\mathcal{C}$
is closed under cokernels.
This \cref{lem:presCoker} and the post-ceding \cref{lem:presAbelian}
provide a slight generalization of the very first result
from \cite[Section 4.1]{MR2355771}
with essentially the same proof.
Another closely related result is \cite[Exercise 8.23]{MR2182076}.

\begin{lem}
  \label{lem:presCoker}
  The cokernel of a homomorphism between
  $\mathcal{J}$-presentable objects
  is again $\mathcal{J}$-presentable.
\end{lem}

\begin{proof}
  Let
  \begin{equation}
    \label{eq:cokSeq}
    X \xrightarrow{\varphi} Y \xrightarrow{\psi} Z \rightarrow 0
  \end{equation}
  be an exact sequence with
  $X$ and $Y$ being $\mathcal{J}$-presentable.
  We choose presentations
  \begin{align*}
    P_1 \rightarrow P_0 \rightarrow X \rightarrow 0 & \\
    \text{and} \quad
    Q_1 \xrightarrow{\delta} Q_0 \xrightarrow{\epsilon} Y \rightarrow 0 &
  \end{align*}
  as in \cref{dfn:presentable}.
  As $P_0$ is projective in $\mathcal{C}$
  there is a homomorphism
  $\tilde{\varphi} \colon P_0 \rightarrow Q_0$
  such that the square
  \begin{equation*}
    \begin{tikzcd}
      P_0
      \arrow[r, "\tilde{\varphi}", dashed]
      \arrow[d]
      &
      Q_0
      \arrow[d]
      \\
      X
      \arrow[r, "\varphi"']
      &
      Y
    \end{tikzcd}
  \end{equation*}
  commutes.
  We aim to show that the sequence
  \begin{equation*}
    P_1 \oplus Q_1 \xrightarrow{(\tilde{\varphi}, \delta)^{\flat}}
    Q_0 \xrightarrow{\psi \circ \epsilon}
    Z \rightarrow
    0
  \end{equation*}
  is exact.
  To this end,
  let $\xi \colon Q_0 \rightarrow W$
  be a homomorphism
  such that
  \begin{equation}
    \label{eq:cokVanishingComp}
    \xi \circ (\tilde{\varphi}, \delta)^{\flat} = 0.
  \end{equation}
  We consider the commutative diagram
  \begin{equation}
    \label{eq:cokDiagram}
    \begin{tikzcd}
      &
      Q_1
      \arrow[d, "\delta"]
      &
      P_1 \oplus Q_1
      \arrow[d, "{(\tilde{\varphi}, \delta)^{\flat}}"]
      \\
      P_0
      \arrow[r, "{\tilde{\varphi}}"]
      \arrow[d, two heads]
      &
      Q_0
      \arrow[r, equal]
      \arrow[d, "\epsilon", two heads]
      &
      Q_0
      \arrow[rdd, "\xi", bend left]
      \arrow[d, "\psi \circ \epsilon"]
      \\
      X
      \arrow[r, "\varphi"']
      \arrow[rrrd, "0"', dashed, bend right]
      &
      Y
      \arrow[r, "\psi"']
      \arrow[rrd, "\xi'", dashed, bend right]
      &
      Z
      \arrow[rd, "\xi''",dashed]
      \\
      & & &
      W
      .
    \end{tikzcd}
  \end{equation}
  Now \eqref{eq:cokVanishingComp} implies in particular that
  $\xi \circ \delta =0$,
  hence there is a unique natural transformation
  $\xi' \colon Y \rightarrow W$
  such that $\xi' \circ \epsilon = \xi$
  as indicated in \eqref{eq:cokDiagram}.
  Now
  \[\xi' \circ \epsilon \circ \tilde{\varphi} =
    \xi \circ \tilde{\varphi} =
    0
  \]
  by \eqref{eq:cokVanishingComp}.
  Moreover,
  as the vertical arrow on the left hand side of \eqref{eq:cokDiagram}
  is an epimorphism
  $\xi' \circ \varphi = 0$ as well.
  By the exactness of \eqref{eq:cokSeq}
  there is a unique natural transformation
  $\xi'' \colon Q_0 \rightarrow W$
  such that $\xi'' \circ \psi \circ \epsilon = \xi'$
  as indicated in \eqref{eq:cokDiagram}.
\end{proof}

\begin{cor}
  \label{cor:presEpi}
  A homomorphism of $\mathcal{J}$-presentable objects
  is an epimorphism in $\mathrm{pres}(\mathcal{J})$
  iff it is an epimorphism in $\mathcal{C}$.
\end{cor}

This corollary has yet another corollary.

\begin{cor}
  \label{cor:presProj}
  A $\mathcal{J}$-presentable object $P$,
  which is projective in $\mathcal{C}$,
  is also projective in $\mathrm{pres}(\mathcal{J})$.
\end{cor}

\begin{cor}
  \label{cor:Jproj}
  If any $\mathcal{J}$-presentable object $X$ of $\mathcal{C}$
  admits a projective cover
  $P \rightarrow X$
  by an object $P$ in $\mathcal{J}$,
  then $\mathcal{J}$
  is the subcategory of projectives in $\mathrm{pres}(\mathcal{J})$.
\end{cor}

\begin{proof}
  Suppose $Q$ is an object of $\mathcal{J}$.
  Then $Q$ is projective in $\mathcal{C}$,
  hence it is projective in $\mathrm{pres}(\mathcal{J})$
  by \cref{cor:presProj}.
  Now suppose $Q$ is projective in $\mathrm{pres}(\mathcal{J})$
  and let ${\varphi \colon P \rightarrow Q}$ be a projective cover
  with $P$ in $\mathcal{J}$.
  Then $P$ is projective in $\mathrm{pres}(\mathcal{J})$
  by \cref{cor:presProj},
  hence ${\varphi \colon P \rightarrow Q}$
  also is a projective cover in $\mathrm{pres}(\mathcal{J})$.
  Moreover, the identity ${\op{id}_Q \colon Q \rightarrow Q}$
  is a projective cover as well
  and thus ${P \cong Q}$ by the uniqueness of projective covers
  \mbox{\cite[Corollary 3.5]{MR3431480}}.
\end{proof}

\begin{lem}
  \label{lem:presAbelian}
  Suppose that for any homomorphism
  $\psi \colon Q \rightarrow P$
  with $Q$ and $P$ in $\mathcal{J}$
  there is an exact sequence
  \begin{equation*}
    R \rightarrow Q \xrightarrow{\psi} P
  \end{equation*}
  with $R$ in $\mathcal{J}$.
  Then the full subcategory $\mathrm{pres}(\mathcal{J})$
  of $\mathcal{C}$
  is an abelian subcategory.
\end{lem}

\begin{proof}
  By \cref{lem:presCoker}
  the cokernel of a natural transformation
  of $\mathcal{J}$-presentable objects
  is a again a $\mathcal{J}$-presentable object.
  Thus, it suffices to show that the kernel of a homomorphism
  of $\mathcal{J}$-presentable objects
  is again $\mathcal{J}$-presentable.
  To this end,
  let
  \begin{equation}
    \label{eq:kernelSeq}
    0 \rightarrow X \xrightarrow{\varphi} Y \xrightarrow{\psi} Z
  \end{equation}
  be an exact sequence 
  with $Y$ and $Z$ being $\mathcal{J}$-presentable.
  By \cref{dfn:presentable} we may choose presentations
  \begin{align*}
    Q_1 \xrightarrow{\alpha}
    Q_0 \rightarrow Y \rightarrow 0 &
    \\
    \text{and} \quad
    R_1 \xrightarrow{\delta}
    R_0 \rightarrow Z \rightarrow 0 &
  \end{align*}
  with $Q_0$, $Q_1$, $R_0$, and $R_1$
  objects of $\mathcal{J}$.
  Since $Q_0$
  is projective
  there is a homomorphism
  $\tilde{\psi} \colon Q_0 \rightarrow R_0$
  such that the square
  \begin{equation*}
    \begin{tikzcd}
      Q_0
      \arrow[r, "\tilde{\psi}", dashed]
      \arrow[d]
      &
      R_0
      \arrow[d]
      \\
      Y
      \arrow[r, "\psi"']
      &
      Z
    \end{tikzcd}
  \end{equation*}
  commutes.
  By our assumptions on $\mathcal{J}$
  there is an exact sequence
  \begin{equation}
    \label{eq:hofibExact1}
    P_0
    \xrightarrow{\kappa_0}
    Q_0 \oplus R_1
    \xrightarrow{(\tilde{\psi}, \delta)^{\flat}}
    R_0
    .
  \end{equation}
  Using this assumption on $\mathcal{J}$ once more
  we also obtain an exact sequence
  \begin{equation}
    \label{eq:hofibExact2}
    P_1
    \xrightarrow{\kappa_1}
    P_0 \oplus Q_1
    \xrightarrow{(\mathrm{pr}_1 \circ \kappa_0, \alpha)^{\flat}}
    Q_0
    ,
  \end{equation}
  where ${\mathrm{pr}_1 \colon Q_0 \oplus R_1 \rightarrow Q_0}$
  is the projection to the first summand $Q_0$.
  We consider the commutative diagram
  \begin{equation}
    \label{eq:kernelDiagram}
    \begin{tikzcd}[row sep=5ex, column sep=12ex]
      P_1
      \arrow[r, "-(\mathrm{pr}_2 \circ \kappa_1)"]
      \arrow[d, "\mathrm{pr}_1 \circ \kappa_1"']
      &
      Q_1
      \arrow[d, "\alpha"]
      \\
      P_0
      \arrow[r, "\mathrm{pr}_1 \circ \kappa_0"']
      \arrow[d, dashed]
      &
      Q_0
      \arrow[r, "\tilde{\psi}"']
      \arrow[d]
      &
      R_0
      \arrow[d]
      \\
      X
      \arrow[r, "\varphi"']
      \arrow[d]
      &
      Y
      \arrow[r, "\psi"']
      \arrow[d]
      &
      Z
      \arrow[d]
      \\
      0
      &
      0
      &
      0
      ,
    \end{tikzcd}
  \end{equation}
  where $\mathrm{pr}_i$ denotes the corresponding projection for $i = 1,2$.
  As the center column of \eqref{eq:kernelDiagram} is a complex
  and as \eqref{eq:kernelSeq} is exact,
  the vertical dashed arrow on the left hand side of \eqref{eq:kernelDiagram}
  exists as indicated.
  Moreover,
  the center and the right column of \eqref{eq:kernelDiagram} are exact.
  Using the exact sequences \eqref{eq:hofibExact1} and \eqref{eq:hofibExact2}
  it follows from a diagram chase in \eqref{eq:kernelDiagram}
  that the left column is exact as well,
  hence $X$
  is $\mathcal{J}$-presentable.
\end{proof}

\section{Abelianization of Triangulated Categories}
\label{sec:abelianization}

The following \cref{prp:abelianization}
is a slight generalization of the first lemma from
\mbox{\cite[Section 4.2]{MR2355771}}
with essentially the same proof.

\begin{prp}
  \label{prp:abelianization}
  Let $\mathcal{C}$ be an abelian category,
  let $\mathcal{T}$ be a triangulated category,
  and let $F \colon \mathcal{T} \rightarrow \mathcal{C}$
  be a fully faithful cohomological functor,
  such that $F(X)$ is projective for all objects $X$ of $\mathcal{T}$.
  Moreover,
  let $\mathcal{J}$ be the essential image of
  $F \colon \mathcal{T} \rightarrow \mathcal{C}$.
  Then the category $\mathrm{pres}(\mathcal{J})$
  of $\mathcal{J}$-presentable objects
  is an abelian subcategory of $\mathcal{C}$
  and the functor
  \begin{equation*}
    \mathcal{T} \rightarrow
    \mathrm{pres}(\mathcal{J}), \,
    \begin{cases}
      X \mapsto F(X) \\
      \varphi \mapsto F(\varphi)
    \end{cases}
  \end{equation*}
  is the abelianization of $\mathcal{T}$.
  In other words,
  the precomposition functor
  $(-) \circ F$
  from the category of exact functors
  $\mathrm{pres}(\mathcal{J}) \rightarrow \mathcal{A}$
  to the category of cohomological functors
  $\mathcal{T} \rightarrow \mathcal{A}$  
  is an equivalence of categories
  for any abelian category $\mathcal{A}$.
\end{prp}

\begin{proof}
  Now suppose we have a homomorphism
  $\psi \colon Q \rightarrow P$
  in $\mathcal{C}$
  with $Q$ and $P$ in $\mathcal{J}$.
  Since $F \colon \mathcal{T} \rightarrow \mathcal{C}$
  is fully faithful there is a homomorphism
  $\tilde{\psi} \colon \tilde{Q} \rightarrow \tilde{P}$
  with $F\big(\tilde{\psi}\big) \cong \psi$.
  Moreover,
  there exists a distinguished triangle
  \begin{equation*}
    \tilde{R} \xrightarrow{\tilde{\varphi}}
    \tilde{Q} \xrightarrow{\tilde{\psi}}
    \tilde{P} \rightarrow
    \tilde{R}[1]
  \end{equation*}
  in $\mathcal{T}$
  by \emph{TR2} and \emph{TR3}.
  Writing $R := F\big(\tilde{R}\big)$ and
  $\varphi \colon R \xrightarrow{F(\tilde{\varphi})}
  F\big(\tilde{Q}\big) \xrightarrow{~\sim~} Q$
  we obtain the exact sequence
  \begin{equation*}
    R \xrightarrow{\varphi}
    Q \xrightarrow{\psi}
    P
    ,
  \end{equation*}
  since
  ${F \colon \mathcal{T} \rightarrow \mathcal{C}}$ is cohomological.
  Thus,
  $\mathrm{pres}(\mathcal{J})$
  is an abelian subcategory of $\mathcal{C}$
  by \cref{lem:presAbelian}.
  Now let $\mathcal{A}$ be another abelian category
  and let
  ${G \colon \mathcal{T} \rightarrow \mathcal{A}}$
  be a cohomological functor.
  We construct a preimage
  ${\overline{G} \colon \mathrm{pres}(\mathcal{J}) \rightarrow \mathcal{A}}$
  of
  $G$ under ${(-) \circ F}$ as follows.
  Suppose
  $X$
  is $\mathcal{J}$-presentable.
  As ${F \colon \mathcal{T} \rightarrow \mathcal{C}}$
  is fully faithful
  we may choose a presentation
  \begin{equation*}
    F (P_1) \xrightarrow{F(\delta)}
    F (P_0) \rightarrow X \rightarrow 0
  \end{equation*}
  with $P_0$ and $P_1$ objects of $\mathcal{T}$.
  With this we define
  $\overline{G}(X)$
  to be the cokernel of
  $G(\delta) \colon
  G(P_1) \rightarrow G(P_0)$.
  It remains to show that
  $\overline{G} \colon
  \mathrm{pres}(\mathcal{J}) \rightarrow \mathcal{A}$
  is exact.
  To this end,
  let
  \begin{equation*}
    0 \rightarrow X \rightarrow Y \rightarrow Z \rightarrow 0
  \end{equation*}
  be a short exact sequence
  of $\mathcal{J}$-presentable objects in $\mathcal{C}$.
  Again,
  we may choose presentations
  \begin{align}
    \nonumber
    F (P_1) \rightarrow
    F (P_0) \rightarrow X \rightarrow 0 &
    \\
    \label{eq:presentationZ}
    \text{and} \quad
    F (Q_1) \xrightarrow{\delta}
    F (Q_0) \xrightarrow{\epsilon} Z \rightarrow 0 &
  \end{align}
  since ${F \colon \mathcal{T} \rightarrow \mathcal{C}}$ is fully faithful.
  As
  $F (Q_0)$
  is projective in $\mathcal{C}$
  there is a homomorphism
  $\sigma \colon F (Q_0) \rightarrow Y$
  such that the diagram
  \begin{equation*}
    \begin{tikzcd}
      &
      F (Q_0)
      \arrow[ld, "\sigma"', dashed]
      \arrow[d, "\epsilon"]
      \\
      Y
      \arrow[r, two heads]
      &
      Z
      \arrow[r]
      &
      0
    \end{tikzcd}
  \end{equation*}
  commutes.
  In conjunction with the nine lemma
  this yields a commutative diagram
  \begin{equation}
    \label{eq:nineLemma}
    \begin{tikzcd}
      &
      0
      \arrow[d]
      &
      0
      \arrow[d]
      &
      0
      \arrow[d]
      \\
      0
      \arrow[r]
      &
      U
      \arrow[r]
      \arrow[d]
      &
      V
      \arrow[r]
      \arrow[d]
      &
      W
      \arrow[r]
      \arrow[d]
      &
      0
      \\
      0
      \arrow[r]
      &
      F(P_0)
      \arrow[r]
      \arrow[d]
      &
      F (P_0 \oplus Q_0)
      \arrow[r]
      \arrow[d]
      &
      F (Q_0)
      \arrow[r]
      \arrow[d, "\epsilon"]
      &
      0
      \\
      0
      \arrow[r]
      &
      X
      \arrow[r]
      \arrow[d]
      &
      Y
      \arrow[r]
      \arrow[d]
      &
      Z
      \arrow[r]
      \arrow[d]
      &
      0
      \\
      &
      0
      &
      0
      &
      0
    \end{tikzcd}
  \end{equation}
  with exact rows and columns.
  By the exactness of \eqref{eq:presentationZ} at
  $F(Q_0)$
  there is a homomorphism
  $\xi \colon F(Q_1) \rightarrow W$
  such that the diagram
  \begin{equation*}
    \begin{tikzcd}
      &
      0
      \arrow[d]
      \\
      &
      W
      \arrow[d, tail]
      \\
      F(Q_1)
      \arrow[ru, "\xi", dashed]
      \arrow[r, "\delta"']
      &
      F(Q_0)
    \end{tikzcd}
  \end{equation*}
  commutes.
  Moreover,
  as
  $F (Q_1)$
  is projective
  there is a homomorphism
  $\sigma' \colon F (Q_0) \rightarrow V$
  such that the diagram
  \begin{equation*}
    \begin{tikzcd}
      &
      F (Q_1)
      \arrow[ld, "\sigma'"', dashed]
      \arrow[d, "\xi"]
      \\
      V
      \arrow[r, two heads]
      &
      W
      \arrow[r]
      &
      0
    \end{tikzcd}
  \end{equation*}
  commutes.
  In conjunction with
  \eqref{eq:nineLemma}
  we obtain the commutative diagram
  \begin{equation}
    \label{eq:horseshoe}
    \begin{tikzcd}
      0
      \arrow[r]
      &
      F (P_1)
      \arrow[r]
      \arrow[d]
      &
      F (P_1 \oplus Q_1)
      \arrow[r]
      \arrow[d]
      &
      F (Q_1)
      \arrow[r]
      \arrow[d, "\delta"]
      &
      0
      \\
      0
      \arrow[r]
      &
      F (P_0)
      \arrow[r]
      \arrow[d]
      &
      F (P_0 \oplus Q_0)
      \arrow[r]
      \arrow[d]
      &
      F (Q_0)
      \arrow[r]
      \arrow[d, "\epsilon"]
      &
      0
      \\
      0
      \arrow[r]
      &
      X
      \arrow[r]
      \arrow[d]
      &
      Y
      \arrow[r]
      \arrow[d]
      &
      Z
      \arrow[r]
      \arrow[d]
      &
      0
      \\
      &
      0
      &
      0
      &
      0
    \end{tikzcd}    
  \end{equation}
  with exact rows and columns
  and the upper two rows split-exact.
  As
  ${G \colon \mathcal{T} \rightarrow \mathcal{A}}$
  is cohomological it is an additive functor in particular.
  As a result the functor
  $\overline{G} \colon
  \mathrm{pres}(\mathcal{J}) \rightarrow \mathcal{A}$
  maps the upper two rows of \eqref{eq:horseshoe}
  to short exact sequences in $\mathcal{A}$.
  Now the cokernel seen as a functor
  from the category of homomorphisms in $\mathcal{A}$ to $\mathcal{A}$
  is right-exact,
  hence the sequence
  \begin{equation*}
    \overline{G}(X) \rightarrow
    \overline{G}(Y) \rightarrow
    \overline{G}(Z) \rightarrow 0
  \end{equation*}
  is exact.
  Finally we show the exactness of
  \begin{equation}
    \label{eq:extLeftExact}
    0 \rightarrow
    \overline{G}(X) \rightarrow
    \overline{G}(Y) \rightarrow
    \overline{G}(Z)    
  \end{equation}
  as well.
  As
  ${F \colon \mathcal{T} \rightarrow \mathcal{C}}$
  is cohomological we may use the diagram \eqref{eq:horseshoe}
  to obtain the commutative diagram
  \begin{equation}
    \label{eq:horseshoeUpsideDown}
    \begin{tikzcd}
      &
      0
      \arrow[d]
      &
      0
      \arrow[d]
      &
      0
      \arrow[d]
      \\
      0
      \arrow[r]
      &
      X
      \arrow[r]
      \arrow[d]
      &
      Y
      \arrow[r]
      \arrow[d]
      &
      Z
      \arrow[r]
      \arrow[d]
      &
      0
      \\
      0
      \arrow[r]
      &
      {F (P_1[1])}
      \arrow[r]
      \arrow[d]
      &
      {F (P_1[1] \oplus Q_1[1])}
      \arrow[r]
      \arrow[d]
      &
      {F (Q_1[1])}
      \arrow[r]
      \arrow[d, "{-\delta[1]}"']
      &
      0
      \\
      0
      \arrow[r]
      &
      {F (P_0[1])}
      \arrow[r]
      &
      {F (P_0[1] \oplus Q_0[1])}
      \arrow[r]
      &
      {F (Q_0[1])}
      \arrow[r]
      &
      0
    \end{tikzcd}    
  \end{equation}
  with exact rows and columns
  and the lower two rows split-exact.
  As
  ${G \colon \mathcal{T} \rightarrow \mathcal{A}}$
  additive
  the functor
  $\overline{G} \colon
  \mathrm{pres}(\mathcal{J}) \rightarrow \mathcal{A}$
  maps the lower two rows of \eqref{eq:horseshoeUpsideDown}
  to short exact sequences in $\mathcal{A}$.
  Moreover, as the kernel is left-exact
  as a functor from the category of homomorphisms in $\mathcal{A}$
  to $\mathcal{A}$
  and as
  ${G \colon \mathcal{T} \rightarrow \mathcal{A}}$
  is cohomological
  the sequence \eqref{eq:extLeftExact} is exact.
\end{proof}

\todos

\end{document}